\documentclass[11pt]{article}
\usepackage{amsmath,amsthm,amsfonts,amssymb}
\usepackage[all]{xy}
\usepackage[usenames]{color} 
\usepackage{makeidx}

\usepackage{xr}

\newcommand\brefGM[1]{\ref{GM#1}}

\newcommand\refGM[1]{I.\brefGM{#1}}

\newcommand\prefGM[1]{(\brefGM{#1})}

\newcommand\trefGM[2]{\refGM{#1}~\prefGM{#2}}

\newcounter{enumitemp}

\newcommand\pref[1]{(\ref{#1})}

\newtheorem{thm}{Theorem}[section]

\newtheorem{theorem}[thm]{Theorem}
\newtheorem*{theorem*}{Theorem}

\newtheorem*{TheoremF}{Theorem F}
\newtheorem*{TheoremG}{Theorem G}
\newtheorem*{TheoremH}{Theorem H}

\newtheorem{lemma}[thm]{Lemma}

\newtheorem{corollary}[thm]{Corollary}
\newtheorem{proposition}[thm]{Proposition}
\newtheorem*{proposition*}{Proposition}

\theoremstyle{definition}

\newtheorem{definition}[thm]{Definition} 

\newtheorem*{defn*}{Definition}
\newtheorem{defns}[thm]{Definitions}

\newtheorem{ex}[thm]{Example}

\newtheorem{remark}[thm]{Remark}
\theoremstyle{remark}

\newcounter{remarks}
{\paragraph*{Remarks}\ skip
 \begin{list}{\arabic{remarks}. }{\usecounter{remarks}%
 \setlength{\leftmargin}{0in}%
 \setlength{\rightmargin}{0in}%
 \setlength{\labelsep}{0pt}%
 \setlength{\labelwidth}{0pt}%
 \setlength{\listparindent}{0pt}%
 }
}
{
\end{list}
}

\newcommand\from\colon
\newcommand\inv{{-1}}
\newcommand\subgroup{<}
\newcommand\infinity\infty
\newcommand\na{\text{na}}
\newcommand\supp{\text{supp}}
\newcommand\disjunion\amalg
\newcommand\act\curvearrowright

\DeclareMathOperator{\Fix}{Fix}

\DeclareMathOperator{\Int}{int}

\DeclareMathOperator\closure{cl}

\newcommand{\C}{{\mathcal C}}
\newcommand{\T}{{\mathbb T}}

\newcommand{\Out}{\mathsf{Out}}
\newcommand{\Aut}{\mathsf{Aut}}

\newcommand{\Stab}{\mathsf{Stab}}

\newcommand{\F}{\mathcal F}

\renewcommand\L{\mathcal L}

\def\B{\mathcal B}
\newcommand{\A}{\mathcal A}

\renewcommand\T{\mathcal T}

\newcommand{\fG} {f : G \to G}
\newcommand{\ti} {\tilde}
\newcommand{\iNp} {indivisible Nielsen path}

\newcommand{\eg}{EG}
\newcommand{\noneg}{NEG}

\renewcommand\neg\noneg

\newcommand{\wt}{\widetilde}

\newcommand{\ct}{CT}
\newcommand{\cts}{CTs}

\newcommand\marginparLee[1]{\marginpar{\tiny #1 --- Lee}}

\newcommand\BH{\cite{BestvinaHandel:tt}}

\newcommand\BookOne{\cite{BFH:TitsOne}}
\newcommand\BookTwo{\cite{BFH:TitsTwo}}

\newcommand\recognition{\cite{FeighnHandel:recognition}}
\newcommand\Axes{\cite{HandelMosher:axes}}

\newcommand\Intro{\cite{HandelMosher:SubgroupsIntro}}
\newcommand\PartOne{Part I \cite{HandelMosher:SubgroupsI}}
\newcommand\PartTwo{Part II \cite{HandelMosher:SubgroupsII}}

\newcommand\PartFour{Part IV \cite{HandelMosher:SubgroupsIV}}

\DeclareMathOperator\interior{int}
\newcommand\bdy\partial

\newcommand\intersect\cap
\newcommand\union\cup
\newcommand\<\langle
\renewcommand\>\rangle
\newcommand\meet\wedge
\newcommand\composed{\circ}
\newcommand\cross\times
\newcommand\restrict{\bigm |}
\newcommand\wh{\widehat}
\newcommand\inject\hookrightarrow

\DeclareMathOperator\Length{Length}

 \newcommand\surjection\twoheadrightarrow
 
\DeclareMathOperator\MCG{\mathcal{MCG}}

\newcommand{\Lambdapp}{\Lambda^+_\phi}

\newcommand{\Lambdapmp}{\Lambda^{\pm}_\phi}

\newcommand\sing{{\text{sing}}}
\newcommand\gen{{\text{gen}}}
\newcommand\ext{{\text{ext}}}
\newcommand\good{{\text{good}}}

\renewcommand\int{{\text{int}}}

\newcommand\un{\text{\tiny un}}
\newcommand\st{\text{\tiny st}}

\title{Subgroup decomposition in $\Out(F_n)$\\ Part III: Weak Attraction Theory}

\author{Michael Handel and Lee Mosher}

\makeindex

\begin{document}

\maketitle

\begin{abstract}
This is the third in a series of four papers, announced in \cite{HandelMosher:SubgroupsIntro}, that develop a decomposition theory for subgroups of $\Out(F_n)$.

In this paper, given $\phi \in \Out(F_n)$ and an attracting-repelling lamination pair for $\phi$, we study which lines and conjugacy classes in $F_n$ are weakly attracted to that lamination pair under forward and backward iteration of $\phi$ respectively. For conjugacy classes, we prove Theorem F from the research annoucement, which exhibits a unique vertex group system called the ``nonattracting subgroup system'' having the property that the conjugacy classes it carries are characterized as those which are not weakly attracted to the attracting lamination under forward iteration, and also as those which are not weakly attracted to the repelling lamination under backward iteration. For lines in general, we prove Theorem G which characterizes exactly which lines are weakly attracted to the attracting lamination under forward iteration and which to the repelling lamination under backward iteration. We also prove Theorem H which gives a uniform version of weak attraction of lines.

\end{abstract}

\newcommand\arXiv{arXiv}

\subsection*{Introduction}

Many results about the groups $\MCG(S)$ and $\Out(F_n)$ are based on dynamical systems. The Tits alternative (\BookOne, \BookTwo\ for $\Out(F_n)$; \cite{McCarthy:Tits} and \cite{Ivanov:subgroups} independently for $\MCG(S)$) says that for any subgroup $H$, either $H$ is virtually abelian or $H$ contains a free subgroup of rank~$\ge 2$, and these free subgroups are constructed by analogues of the classical ``ping-pong argument'' for group actions on topological spaces. Dynamical ping-pong arguments were also important in Ivanov's classification of subgroups of $\MCG(S)$ \cite{Ivanov:subgroups}. And they will be important in \PartFour\ where we prove our main theorem about subgroups of $\Out(F_n)$, Theorem~C stated in the Introduction~\Intro. 

Ping-pong arguments are themselves based on understanding the dynamics of an individual group element~$\phi$, particularly an analysis of attracting and repelling fixed sets of~$\phi$, of their associated basins of attraction and repulsion, and of neutral sets which are neither attracted nor repelled. The proofs in \cite{McCarthy:Tits,Ivanov:subgroups} use the action of $\MCG(S)$ on Thurston's space $\mathcal{PML}(S)$ of projective measured laminations on $S$. 

The proof of Theorem C in \PartFour\ will employ ping-pong arguments for the action of $\Out(F_n)$ on the \emph{space of lines} $\B = \B(F_n)$, which is just the quotient of the action of $F_n$ on the space $\wt \B$ of two point subsets of $\bdy F_n$. The basis of those ping-pong arguments will be \emph{Weak Attraction Theory} which, given $\phi \in \Out(F_n)$ and a dual lamination pair $\Lambda^\pm \in \L^\pm(\phi)$, addresses the following dynamical question regarding the action of $\phi$ on~$\B$: 
\begin{description}
\item[General weak attraction question:] Which lines $\ell \in \B$ are weakly attracted\index{weakly attracted} to $\Lambda_+$ under iteration of $\phi$? Which are weakly attracted to $\Lambda_-$ under iteration of $\phi^\inv$?  And which are weakly attracted to neither $\Lambda_+$ nor $\Lambda_-$?
\end{description}
To say that $\ell$ is weakly attracted to $\Lambda_+$ (under iteration of~$\phi$) means that $\ell$ is weakly attracted to a generic leaf $\lambda \in \Lambda_+$, that is, the sequence $\phi^k(\ell)$ converges in the weak topology to $\ell$ as $k \to +\infinity$. Note that this is independent of the choice of generic leaf of $\Lambda_+$, since all of them have the same weak closure, namely~$\Lambda_+$.

Our answers to the above question are an adaptation and generalization of many of the ideas and constructions found in The Weak Attraction Theorem 6.0.1 of \BookOne, which answered a narrower version of the question above, obtained by restricting to a lamination $\Lambda_+$ which is topmost in $\L(\phi)$ and to birecurrent lines. That answer was expressed in terms of the structure of an ``improved relative train track representative'' of~$\phi$.

In this paper we develop weak attraction theory to completely answer the general weak attraction question. Our theorems are expressed both in terms of the structure of a \ct\ representative of $\phi$, and in more invariant terms. The theory is summarized in Theorems F, G and~H, versions of which were stated earlier in \Intro; the versions stated here are more expansive and precise. Theorem~F focusses on periodic lines and on the nonattracting subgroup system; Theorems~G and~H are concerned with arbitrary lines. Each of these theorems has applications in \PartFour. 

\paragraph{The nonattracting subgroup system: Theorem F.} We first answer the general weak attraction question restricted to ``periodic'' lines in $\B$, equivalently circuits in marked graphs, equivalently conjugacy classes in $F_n$. The statement uses two concepts from \hbox{\PartOne}: \emph{geometricity} of general \eg\ strata (which was in turn based on geometricity of top \eg\ strata as developed in \BookOne), and \emph{vertex group systems}.


\begin{TheoremF}[Properties of the nonattracting subgroup system] For each rotationless $\phi \in \Out(F_n)$ and each $\Lambda^\pm \in \L^\pm(\phi)$ there exists a subgroup system $\A_\na(\Lambda^\pm)$, called the \emph{nonattracting subgroup system}, with the following properties:
\begin{enumerate}
\item\label{ItemThmCAnaVGS}
(Proposition \ref{PropVerySmallTree} \pref{ItemA_naNP}) $\A_\na(\Lambda^\pm)$ is a vertex group system.
\item\label{ItemThmCAnaGeom}
(Proposition~\ref{PropVerySmallTree}~\pref{ItemA_naNoNP})  $\Lambda^\pm$ is geometric if and only if 
$\A_\na(\Lambda^\pm)$ is not a free factor system.
\item\label{ItemThmCAnaIsAna}
(Corollary~\ref{CorNAWellDefined})
For each conjugacy class $c$ in $F_n$ the following are equivalent: \\
$\bullet$\quad $c$ is not weakly attracted to $\Lambda^+_\phi$ under iteration of $\phi$; \\
$\bullet$\quad $c$ is carried by $\A_\na(\Lambda^\pm)$.
\item\label{ItemThmFAnaWellDef}
(Corollary~\ref{CorPMna}) $\A_\na(\Lambda^\pm)$ is uniquely determined by items~\pref{ItemThmCAnaVGS} and~\pref{ItemThmCAnaIsAna}.
\item\label{ItemThmCAnaPM}
(Corollaries \ref{CorNAWellDefined} and \ref{CorNAIndOfPM}) For each conjugacy class $c$ in $F_n$, $c$ is not weakly attracted to $\Lambda^+$ by iteration of $\phi$ if and only if $c$ is not weakly attracted to $\Lambda^-$ by iteration of~$\phi^\inv$.
\end{enumerate}
\end{TheoremF}
\noindent
Furthermore (Definition~\ref{defn:Z}), choosing any \ct\ $f \from G \to G$ representing $\phi$ with \eg-stratum $H_r$ corresponding to $\Lambda^+_\phi$, the nonattracting subgroup system $\A_\na(\Lambda^\pm)$ has a concrete description in terms of $f$ and the indivisible Nielsen paths of height~$r$ (the latter are described in (\recognition\ Corollary~4.19) or Fact~\refGM{FactEGNPUniqueness}). The description given in Definition~\ref{defn:Z} is our first definition of $\A_\na(\Lambda^\pm)$, and it is not until Corollary~\ref{CorPMna} that we prove $\A_\na(\Lambda^\pm)$ is well-defined independent of the choice of \ct\ (item~\pref{ItemThmFAnaWellDef} above). Corollary~\ref{CorNAIndOfPM} (item~\pref{ItemThmCAnaPM} above) shows moreover that $\A_\na(\Lambda^\pm)$ is indeed well-defined independent of the choice of nonzero power of $\phi$, depending only on the cyclic subgroup $\<\phi\>$ and the lamination pair~$\Lambda^\pm$. 

\subparagraph{Notation:} The nonattracting subgroup system $\A_\na(\Lambda^\pm)$ depends not only on the lamination pair $\Lambda^\pm$ but also on the outer automorphism $\phi$ (up to nonzero powers). Often we emphasize this dependence by building $\phi$ into the notation for the lamination itself, writing $\Lambda^\pm_\phi$ and $\A_\na(\Lambda^\pm_\phi)$.

\paragraph{The set of nonattracted lines: Theorems G and~H.}  Theorem~G, a vague statement of which was given in the introduction, is a detailed description of the set $\B_\na(\Lambdapp;\phi)$ of all lines $\gamma \in \B$ that are not attracted to $\Lambda^+_\phi$ under iteration by $\phi$. Theorem~H is a less technical and more easily applied distillation of Theorem~G, and is applied several times in \PartFour.

As stated in Lemma~\ref{LemmaThreeNASets}, there are three somewhat obvious subsets of  $\B_\na(\Lambdapp;\phi)$. One is the subset $\B(\A_\na(\Lambdapmp))$ of all lines supported by the nonattracting subgroup system $\A_\na(\Lambdapmp)$. Another is the subset $\B_\gen(\phi^\inv)$ of all generic leaves of attracting laminations for $\phi^{-1}$. The third is the subset $\B_\sing(\phi^\inv)$ of all singular lines for $\phi^{-1}$: by definition these lines are the images under the quotient map $\wt\B \mapsto \B$ of those endpoint pairs $\{\xi,\eta\} \in \wt \B$ such that $\xi,\eta$ are each nonrepelling fixed points for the action of some automorphism representing~$\phi^\inv$. 

In Definition~\ref{DefnConcatenation} we shall define an operation of ``ideal concatenation'' of lines: given a pair of lines which are asymptotic in one direction, they define a third line by concatenating at their common ideal point and straightening, or what is the same thing by connecting their opposite ideal points by a unique line. 

Theorem~G should be thought of as stating that $\B_\na(\Lambdapp;\phi)$ is the smallest set of lines that contains $\B(\A_\na(\Lambdapmp)) \union \B_\gen(\phi^\inv) \union \B_\sing(\phi^\inv)$ and is closed under this operation of ideal concatenation. It turns out that only a limited amount of such concatenation is possible, namely, extending a line of $\B(\A_\na(\Lambdapmp))$ by concatenating on one or both ends with a line of $\B_\sing(\phi^\inv)$, producing a set of lines we denote $\B_\ext(\Lambda^\pm_\phi;\phi^\inv)$ (see Section~\ref{SectionNAConcatenation}).

\begin{TheoremG}[\textbf{Theorem \ref{ThmRevisedWAT}}]
If $\phi, \phi^{-1} \in \Out(F_n)$ are rotationless and if $\Lambda^\pm_\phi \in \L^\pm(\phi)$ then 
$$\B_\na(\Lambda^+_\phi) = \B_\ext(\Lambda^\pm_\phi;\phi^\inv) \union \B_\gen(\phi^\inv) \union \B_\sing(\phi^\inv)
$$
\end{TheoremG}
Note that the first of the three terms in the union is the only one that depends on the lamination pair $\Lambda^\pm_\phi$; the other two depend only on $\phi^\inv$.

\bigskip

For certain purposes in \PartFour\ the following corollary to Theorem~G is useful in being easier to directly apply. In particular item~\pref{ThmHUnifWA} provides a topologically uniform version of weak attraction:

\begin{TheoremH}[\textbf{Corollary~\ref{CorOneWayOrTheOther}}]
Given rotationless $\phi,\phi^\inv \in \Out(F_n)$ and a dual lamination pair $\Lambda^\pm \in \L^\pm(\phi)$, the following hold:
\begin{enumerate}
\item Any line $\ell \in \B$ that is not carried by $\A_\na(\Lambda^\pm)$ is weakly attracted either to $\Lambda^+$ by iteration of $\phi$ or to $\Lambda^-$ by iteration by $\phi^\inv$.
\item \label{ThmHUnifWA}
For any neighborhoods $V^+,V^- \subset \B$ of $\Lambda^+, \Lambda^-$, respectively, there exists an integer $m \ge 1$ such that for any line $\ell \in \B$ at least one of the following holds: \ $\gamma \in V^-$; \ $\phi^m(\gamma) \in V^+$; \ or $\gamma$ is carried by $\A_\na(\Lambda^\pm)$.
\end{enumerate}
\end{TheoremH}

\setcounter{tocdepth}{2}
\tableofcontents

\section{The nonattracting subgroup system}
\label{SectionWeakAttraction}

Consider a rotationless $\phi \in \Out(F_n)$ and a dual lamination pair $\Lambda^\pm_\phi \in \L^\pm(\phi)$. Since $\phi$ is rotationless its action on $\L(\phi)$ is the identity and therefore so is its action on $\L(\phi^\inv)$; the laminations $\Lambda^+_\phi$ and $\Lambda^-_\phi$ are therefore fixed by~$\phi$ and by~$\phi^\inv$. In this setting we shall define the \emph{nonattracting subgroup system} $\A_\na(\Lambda^\pm_\phi)$, an invariant of $\phi$ and $\Lambda^\pm_\phi$.

One can view the definition of $\A_\na(\Lambda^\pm_\phi)$ in two ways. First, in Definition~\ref{defn:Z}, we define $\A_\na(\Lambda^\pm_\phi)$ with respect to a choice of a \ct\ representing $\phi$; this \ct\ acts as a choice of ``coordinate system'' for $\phi$, and with this choice the description of $\A_\na(\Lambda^\pm_\phi)$ is very concrete. We derive properties of this definition in results to follow, from Proposition~\ref{PropVerySmallTree} to Corollary~\ref{CorNASubgroups}, including most importantly the proofs of items~\pref{ItemThmCAnaVGS}, \pref{ItemThmCAnaGeom} and~\pref{ItemThmCAnaIsAna} of Theorem~F. Then, in Corollaries~\ref{CorPMna} and~\ref{CorNAIndOfPM}, we prove that $\A_\na(\Lambda^\pm_\phi)$ is invariantly defined, independent of the choice of \ct\, and furthermore independent of the choice of a positive or negative power of~$\phi$, in particular proving items~\pref{ItemThmFAnaWellDef} and~\pref{ItemThmCAnaPM} of Theorem~F. The independence result is what allows us to regard the nonattracting subgroup system as an invariant of a dual lamination pair rather than of each lamination individually (but still with implicit dependence on $\phi$ up to nonzero power).

\smallskip\noindent
\textbf{Weak attraction.} Recall (Section~\refGM{SectionLineDefs})\footnote{Cross references such as ``Section I.X.Y.Z'' refer to Section X.Y.Z of \PartOne. Cross references to the Introduction \Intro, to \PartOne, and to \PartTwo\ are to the June 2013 versions.} 
the notation $\B$ for the space of lines of~$F_n$ on which $\Out(F_n)$ acts naturally, and recall that that a line $\ell \in \B$ is said to be \emph{weakly attracted} to a generic leaf $\lambda \in \Lambda^+_\phi \subset \B$ under iteration by $\phi$ if the sequence $\phi^n(\ell)$ weakly converges to $\lambda$ as $n \to +\infinity$, that is, for each neighborhood $U \subset \B$ of $\lambda$ there exists an integer $N > 0$ such that if $n \ge N$ then $\phi^n(\ell) \in U$. Note that since any two generic leaves of $\Lambda^+_\phi$ have the same weak closure, namely $\Lambda^+_\phi$, this property is independent of the choice of~$\lambda$; for that reason we often speak of $\ell$ being weakly attracted to $\Lambda^+_\phi$ by iteration of $\phi$.

This definition of weak attraction applies to $\phi^\inv$ as well, and so we may speak of $\ell$ being weakly attracted to $\Lambda^-_\phi$ under iteration by $\phi^\inv$. This definition also applies to iteration of a \ct\ $f \from G \to G$ representing $\phi$ on elements of the space $\wh\B(G)$ (Section~\refGM{def:DoubleSharp}), which contains the subspace $\B(G)$ identified with $\B$ by letting lines be realized in~$G$, and which also contains all finite paths and rays in~$G$. We may speak of such paths being weakly attracted to $\Lambda^+_\phi$ under iteration by $f$. Whenever $\phi$ and the $\pm$ sign are understood, as they are in the notations $\Lambda^+_\phi$ and $\Lambda^-_\phi$, we tend to drop the phrase ``under iteration by \ldots''. 

%
%

\begin{remark} Suppose that $\phi$ is a rotationless iterate of some possibly nonrotationless $\eta \in \Out(F_n)$ and that $\Lambda^+_\phi$ is $\eta$-invariant.  Then $\gamma$ is weakly attracted to $\Lambda^+_\phi$ under iteration by $\eta$ if and only $\gamma$ is weakly attracted to $\Lambda^+_\phi $ under iteration by $\phi$. Our results therefore apply to $\eta$ as well as $\phi$.
\end{remark}

\subsection{The nonattracting subgroup system $\A_{\na}(\Lambdapp)$}
\label{SectionAsubNA}
 
The Weak Attraction Theorem  6.0.1 of \BookOne\ answers the ``general weak attraction question'' posed above in the restricted setting of a lamination pair $\Lambda^\pm_\phi$ which is topmost with respect to inclusion, and under restriction to birecurrent lines only. The answer is expressed in terms of an ``improved relative train track representative'' $g \from G \to G$, a ``nonattracting subgraph'' $Z \subset G$, a (possibly trivial) Nielsen path $\hat \rho_r$, and an associated set of paths denoted $\<Z,\hat \rho_r\>$. The construction and properties of $Z$ and $\<Z,\hat \rho_r\>$ are given in \cite[Proposition~6.0.4]{BFH:TitsOne}.  

In Definition~\ref{defn:Z} and the lemmas that follow, we generalize the subgraph $Z$, the path set $\<Z,\hat \rho_r\>$, and the nonattracting subgroup system beyond the topmost setting. 

\medskip
\textbf{Notation:} Throughout Subsections~\ref{SectionAsubNA} and~\ref{SectionAppsAndPropsANA} we fix a rotationless $\phi \in \Out(F_n)$, a lamination pair~$\Lambda^\pm_\phi \in \L^\pm(\phi)$, and a \ct\ representative $f \from G \to G$ with \eg\ stratum $H_r$ corresponding to $\Lambda^+_\phi$. For a review of \cts, completely split paths, and the terms of a complete splitting, we refer the reader to Section~\refGM{SectionRTTDefs}, particularly Definition~\refGM{DefCompleteSplitting}.

For the nonattracting subgroup system we shall use various notations in various contexts. The notation $\A_\na(\Lambda^+_\phi)$ is used from the start, presuming immediately what we shall eventually show in Corollary~\ref{CorNAWellDefined} regarding its independence from the choice of a \ct\ representative, but leaving open for a while the issue of whether it depends on the choice of $\pm$ sign. After the latter independence is established in Corollary~\ref{CorNAIndOfPM} we will switch over to the notation $\A_\na(\Lambda^\pm_\phi)$. When we wish to emphasize dependence on $\phi$ we sometimes use the notation $\A_\na(\Lambda^\pm;\phi)$ or $\A_\na(\Lambda^+;\phi)$; and when we wish to de-emphasize this dependence we sometimes use $\A_\na(\Lambda^\pm)$ or $\A_\na(\Lambda^+)$.

\medskip

\begin{defns} \label{defn:Z} \textbf{The graph $Z$, the path $\hat \rho_r$, the path set $\<Z,\hat \rho_r\>$, and the subgroup system $\A_\na(\Lambda^+_\phi)$.} \quad

\noindent
We shall define the \emph{nonattracting subgraph} $Z$ of $G$, and a path $\hat\rho_r$, either a trivial path or a height~$r$ indivisible Nielsen path if one exists. Using these we shall define a graph $K$ and an immersion $K \mapsto G$ by consistently gluing together the graph $Z$ and the domain of $\hat\rho_r$. We then define $\A_\na(\Lambdapp)$ in terms of the induced \hbox{$\pi_1$-injection} on each component of~$K$. We also define a groupoid of paths $\<Z,\hat\rho_r\>$ in~$G$, consisting of all concatenations whose terms are edges of $Z$ and copies of the path~$\hat\rho_r$ or its inverse, equivalently all paths in $G$ that are images under the immersion $K \to G$ of paths in~$K$.

\smallskip
\textbf{Definition of the graph $Z$.}
The \emph{nonattracting subgraph} $Z$ of $G$ is defined as a union of certain strata $H_i \ne H_r$ of $G$, as follows. If $H_i$ is an irreducible stratum then $H_i \subset Z$ if and only if no edge of $H_i$ is weakly attracted to $\Lambda$; equivalently, using Fact~\trefGM{FactAttractingLeaves}{ItemSplitAnEdge}, we have $H_i \subset G \setminus Z$ if and only if for some (every) edge $E_i$ of $H_i$ there exists $k \ge 0$ so that some term in the complete splitting of $f^k_\#(E_i)$ is an edge in~$H_r$. If $H_i$ is a zero stratum enveloped by an \eg\ stratum $H_s$ then $H_i \subset Z$ if and only if $H_s \subset Z$. 

\smallskip\textbf{Remark.} 
$Z$ automatically contains every stratum $H_i$ which is a fixed edge, an \neg-linear edge, or an \eg\ stratum distinct from $H_r$ for which there exists an indivisible Nielsen path of height $i$. For a fixed edge this is obvious. If $H_i$ is an \neg-linear edge $E_i$ then this follows from (Linear Edges) which says that $f(E_i) = E_i \cdot u$ where $u$ is a closed Nielsen path, because for all $k \ge 1$ it follows that the path $f^k_\#(E_i)$ completely splits as $E_i$ followed by Nielsen paths of height $<i$, and no edges of $E_r$ occur in this splitting. For an \eg\ stratum $H_i$ with an indivisible Nielsen path of height $i$ this follows from Fact~\trefGM{FactNielsenBottommost}{ItemBottommostEdges} which says that for each edge $E \subset H_i$ and each $k \ge 1$, the path $f^k_\#(E)$ completely splits into edges of $H_i$ and Nielsen paths of height $< i$; again no edges of $E_r$ occur in this splitting.

\smallskip\textbf{Remark.} 
Suppose that $H_i$ is a zero stratum enveloped by the \eg\ stratum $H_s$ and that $H_i \subset Z$. Applying the definition of $Z$ to $H_i$ it follows that $H_s \subset Z$. Applying the definition of $Z$ to $H_s$ it follows that no $s$-taken connecting path in $H_i$ is weakly attracted to $\Lambdapp$. Applying (Zero Strata) it follows that no edge in $Z$ is weakly attracted to $\Lambda$.

\smallskip\textbf{Definition of the path $\hat \rho_r$.} If there is an \iNp\ $\rho_r$ of height $r$ then it is unique up to reversal by Fact~\refGM{FactEGNPUniqueness} and we define $\hat \rho_r = \rho_r$. Otherwise, by convention we choose a vertex of~$H_r$ and define $\hat \rho_r$ to be the trivial path at that vertex. 

\smallskip\textbf{Definition of the path set $\<Z,\hat \rho_r\>$.} Consider $\wh\B(G)$, the set of lines, rays, circuits, and finite paths in $G$ (Definition~\refGM{SectionLineDefs}). Define the subset $\<Z,\hat \rho_r\> \subset \wh\B(G)$ to consist of all elements which decompose into a concatenation of subpaths each of which is either an edge in $Z$, the path $\hat \rho_r$ or its inverse $\hat \rho_r^{-1}$. Given $\ell \in \wh\B(G)$, if $\ell \in \<Z,\hat\rho_r\>$ we will also say that $\ell$ is \emph{carried by} $\<Z,\hat\rho_r\>$; see e.g.\ Lemma~\ref{LemmaZPClosed}~\pref{item:ZP=NA}.

\smallskip\textbf{Definition of the subgroup system $\A_\na(\Lambdapp)$.} If $\hat\rho_r$ is the trivial path, let $K = Z$ and let $h \from K \inject G$ be the inclusion. Otherwise, define $K$ to be the graph obtained from the disjoint union of $Z$ and an edge $E_\rho$ representing the domain of the Nielsen path $\rho_r \from E_\rho \to G_r$, with identifications as follows: given an endpoint $x \in E(\rho)$, if $\rho_r(x) \in Z$ then identify $x \sim \rho_r(x)$; also, given distinct endpoints $x,y \in E(\rho)$, if $\rho_r(x)=\rho_r(y)$ then identify $x \sim y$ (these points have already been identified if $\rho_r(x)=\rho_r(y) \in Z$). Define $h \from K \to G$ to be the map induced by the inclusion $Z \inject G$ and by the map $\rho_r \from E_\rho \to G$. By Fact~\refGM{FactEGNPUniqueness} the initial oriented edges of $\rho_r$ and $\bar\rho_r$ are distinct in $H_r$, and since no edge of $H_r$ is in $Z$ it follows that the map $h$ is an immersion. The restriction of $h$ to each component of $K$ therefore induces an injection on the level of fundamental groups. Define $\A_\na(\Lambdapp)$, the \emph{nonattracting subgroup system}, to be the subgroup system determined by the images of the fundamental group injections induced by the immersion $h \from K \to G$, over all noncontractible components of $K$. 

\smallskip
\textbf{Remark: The case of a top stratum.} In the special case that $H_r$ is the top stratum of $G$, there is a useful formula for $\A_\na(\Lambda^+_\phi)$ which is obtained by considering three subcases. First, when $\hat\rho_r$ is trivial we have $K=Z=G_{r-1}$. Second is the geometric case, where $\hat\rho_r$ is a closed Nielsen path whose endpoint is an interior point of $H_r$ (Fact~\trefGM{FactEGNielsenCrossings}{ItemEGNielsenPointInterior}), and so the graph $K$ is the disjoint union of $Z=G_{r-1}$ with a loop mapping to $\rho_r$. Third is the ``parageometric'' case, where $\hat\rho_r$ is a nonclosed Nielsen path having at least one endpoint which is an interior point of $H_r$ (Fact~\trefGM{FactEGNielsenCrossings}{ItemEGParageometricOneInterior}), and so $K$ is obtained by attaching an arc to $Z=G_{r-1}$ by identifying at most one endpoint of the arc to $G_{r-1}$; note in this case that union of noncontractible components of $K$ deformation retracts to the union of noncontractible components of $G_{r-1}$. From this we obtain the following formula:
$$\A_\na(\Lambda^+_\phi) = \begin{cases}
[\pi_1 G_{r-1}] & \quad\text{if $\Lambda^+_\phi$ and $H_r$ are nongeometric} \\
[\pi_1 G_{r-1}] \union \{[\<\rho_r\>]\} &\quad\text{if $\Lambda^+_\phi$ and $H_r$ are geometric}
\end{cases}
$$
where in the geometric case $[\<\rho_r\>]$ denotes the conjugacy class of the infinite cyclic subgroup generated by an element of $F_n$ represented by the closed Nielsen path $\rho_r$. 

This completes Definitions~\ref{defn:Z}.
\end{defns}


\begin{remark} 
\label{RemarkGeometricK}
In the special case that the stratum $H_r$ is geometric, the 1-complex $K$ lives naturally as an embedded subcomplex of the geometric model $X$ for $H_r$ (Definition~\refGM{DefGeomModel}), as follows. By item~\prefGM{ItemInteriorBasePoint} of that definition, we may identify $K$ with the subcomplex $Z \union j(\bdy_0 S) \subset G \union j(\bdy_0 S) \subset X$ in such a way that the immersion $K \to G$ is identified with the restriction to $K$ of the deformation retraction $d \from X \to G$. The subgroup system $\A_\na(\Lambdapp) = [\pi_1 K]$ may therefore be described as the conjugacy classes of the images of the inclusion induced injections $\pi_1 K_i \to \pi_1 X \approx F_n$, over all noncontractible components $K_i \subset X$. Noting that $j \from S \to X$ maps each boundary component $\bdy_1 S,\ldots,\bdy_m S$ to $G_{r-1} \subset Z \subset K$ and maps $\bdy_0 S$ to $j(\bdy_0 S) \subset K$, we have $j(\bdy S) \subset K$. It follows in the geometric case that Proposition~\refGM{PropGeomVertGrSys} applies to $[\pi_1 K]$, the conclusion of which will be used in the proof of the following proposition.
\end{remark}

Recall the characterization of geometricity of $\Lambdapp$ given in Proposition \refGM{PropGeomEquiv}, expressed in terms of the free factor support of the boundary components of $S$. Our next result, among other things, gives a different characterization of geometricity of a lamination $\Lambdapp \in \L(\phi)$, expressed in terms of the nonattracting subgroup system $\A_\na(\Lambdapp)$. 


\begin{proposition}[Properties of the nonattracting subgroup system]
\label{PropVerySmallTree} Given a \ct\ $f \from G \to G$ representing $\phi$ with \eg\ stratum $H_r$ corresponding to $\Lambdapp$, the subgroup system $\A_\na(\Lambdapp)$ satisfies the following:
\begin{enumerate}
\item \label{ItemA_naNP}
$\A_\na(\Lambdapp)$ is a vertex group system.
\item \label{ItemA_naNoNP}
$\A_\na(\Lambdapp)$ is a free factor system if and only if the stratum $H_r$ is not geometric.
\item \label{ItemA_naMalnormal}
$\A_\na(\Lambdapp)$ is malnormal, with one component for each noncontractible component of~$K$.
\end{enumerate}
\end{proposition}

\begin{proof} First we show that any subgroup $A$ for which $[A] \in \A_\na(\Lambdapp)$ is nontrivial and proper, as required for a vertex group system. Nontriviality follows because only noncontractible components of $K$ are used. To prove properness: if $\hat\rho_r$ is trivial then any circuit containing an edge of $H_r$ is not carried by $\A_\na(\Lambdapp)$; if $\hat\rho = \rho_r$ is nontrivial then any circuit containing an edge of $H_r$ but not containing $\rho_r$ is not carried by $\A_\na(\Lambdapp)$.

We adopt the notation of Definition~\ref{defn:Z}. By applying Fact~\refGM{FactEGNielsenCrossings} and Proposition~\refGM{PropGeomEquiv}, when $\Lambdapp$ is not geometric then $\hat\rho_r$ is either trivial or a nonclosed Nielsen path, and when $\Lambdapp$ is geometric then $\hat\rho_r$ is a closed Nielsen path. We prove \pref{ItemA_naNP}---\pref{ItemA_naMalnormal} by considering these three cases of $\hat\rho_r$ separately.

\smallskip

\textbf{Case 1: $\hat\rho_r$ is trivial.} In this case $K=Z$, and $\A_\na(\Lambdapp)$ is the free factor system associated to the subgraph $Z \subset G$. Item~\pref{ItemA_naMalnormal} follows immediately.

\smallskip

\textbf{Case 2: $\hat\rho_r = \rho_r$ is a nonclosed Nielsen path.} We prove that $\A_\na(\Lambdapp)$ is a free factor system following an argument of \BookOne\ Lemma~5.1.7. By Fact~\trefGM{FactEGNielsenCrossings}{ItemEGNielsenNotClosed} there is an edge $E \subset H_r$ that is crossed exactly once by $\rho_r$. We may decompose $\rho_r$ into a concatenation of subpaths $\rho_r = \sigma E \tau$ where $\sigma,\tau$ are paths in $G_r \setminus \Int(E)$. Let $\wh G$ be the graph obtained from $G  \setminus \Int(E)$ by attaching an edge $J$, letting the initial and terminal endpoints of $J$ be equal to the initial and terminal endpoints of~$\rho_r$, respectively. The identity map on $G \setminus \Int(E)$ extends to a map $h \from \wh G \to G$ that takes the edge $J$ to the path $\rho_r$, and to a homotopy inverse $\bar h \from G \to \wh G$ that takes the edge $E$ to the path~$\bar\sigma J \bar\tau$. We may therefore view $\wh G$ as a marked graph, pulling the marking on $G$ back via $h$. Notice that $K$ may be identified with the subgraph $Z \union J \subset \wh G$, in such a way that the map $h \from \wh G \to G$ is an extension of the map $h \from K \to G$ as originally defined. It follows that $\A_\na(\Lambdapp)$ is the free factor system associated to the subgraph $Z \union J$.

In this case, as in Case~1, item~\pref{ItemA_naMalnormal} follows immediately because of the identification of $K$ with a subgraph of the marked graph $\wh G$.

\smallskip

\textbf{Case 3: $\hat\rho_r = \rho_r$ is a closed Nielsen path.} In this case $H_r$ is geometric. Adopting the notation of the geometric model $X$ for $H_r$, Definition~\refGM{DefGeomModel}, by Remark~\ref{RemarkGeometricK} we have $\A_\na(\Lambdapp)=[\pi_1 K]$ for a subgraph $K \subset L$ containing $j(\bdy S)$. Applying Proposition~\refGM{PropGeomVertGrSys} it follows that $[\pi_1 K]$ is a vertex group system. 

If $\A_\na \Lambdapp = [\pi_1 K]$ were a free factor system then, since each of the conjugacy classes $[\bdy_0 S],\ldots,[\bdy_m S]$ is supported by $[\pi_1 K]$, it would follow by Proposition~\trefGM{PropGeomEquiv}{ItemBoundarySupportCarriesLambda} that $[\pi_1 S] \sqsubset [\pi_1 K]$. However, since $S$ supports a pseudo-Anosov mapping class, it follows that $S$ contains a simple closed curve $c$ not homotopic to a curve in $\bdy S$. By Lemma~\trefGM{LemmaLImmersed}{ItemSeparationOfSAndL} we have $[j(c)] \not\in [\pi_1 K]$ while $[j(c)] \in [\pi_1 S]$. This is a contradiction and so $\A_\na\Lambdapp$ is not a free factor system.

In this case, item~\pref{ItemA_naMalnormal} is a consequence of Lemma~\trefGM{LemmaLImmersed}{ItemComplementMalnormal}.
\end{proof}

Item~\pref{Item:Groupoid} in the next lemma states that $\<Z,\hat \rho_r\>$ is a groupoid, by which we mean that the tightened concatenation of any two paths in $\<Z,\hat \rho_r\>$ is also a path in $\<Z,\hat \rho_r\>$ as long as that concatenation is defined. For example, the concatenation of two distinct rays in $\<Z,\hat \rho_r\>$ with the same base point tightens to a line in $\<Z,\hat \rho_r\>$. 

\begin{lemma} \label{LemmaZPClosed}  Assuming the notation of Definitions~\ref{defn:Z},
\begin{enumerate}
\item \label{Item:BEQuivZP}
The map $h$ induces a bijection between $\wh\B(K)$ and $\<Z,\hat \rho_r\>$.
\item \label{Item:Groupoid}
$\<Z,\hat \rho_r\>$ is a groupoid. 
\item \label{item:ZP=NA} 
The set of lines in $G$ carried by $\<Z,\hat \rho_r\>$ is the same as the set of lines carried by $\A_{\na}(\Lambdapp)$. The set of rays in $G$ asymptotically equivalent to a ray carried by $\<Z,\hat\rho_r\>$ is the same as the set of rays carried by $\A_\na(\Lambdapp)$.
\item\label{Item:circuits} The set of circuits in $G$ carried by $\<Z,\hat \rho_r\>$ is the same as the set of circuits carried by $\A_{\na}(\Lambdapp)$. 
\item \label{item:closed} The set of lines carried by $\<Z,\hat \rho_r\>$ is closed in the weak topology.
\item\label{item:UniqueLift}   If $[A_1], [A_2] \in \A_{\na}(\Lambdapp)$ and if $A_1 \ne A_2$ then $A_1 \cap A_2 = \{1\}$ and $\partial A_1 \cap \partial A_2 = \emptyset$.
\end{enumerate}
\end{lemma}

\begin{proof} We make use of four evident properties of the immersion $h:K \to G$. The first is that every path in $K$ with endpoints, if any, at vertices is mapped by $h$ to an element of $\<Z,\hat \rho_r\>$. The second is that $h$ induces a bijection between the vertex sets of $K$ and of $Z \cup \partial \hat \rho_r$. The third is that for each edge $E$ of $Z$,  there is a unique edge of $K$ that projects to $E$ and that no other subpath of $K$  has at least one endpoint at a vertex and projects to~$E$. The last is that if $\hat \rho_r$ is non-trivial then it has a unique lift to $K$ (because its  unique illegal turn of height $r$ does). Together these imply~\pref{Item:BEQuivZP} which immediately implies  \pref{Item:Groupoid}. 

Let $K_1,\ldots,K_J$ be the cores of the noncontractible components of $K$. Let $\ti h_j \from \wt K_j \inject \wt G$ be a lift of $h \restrict K_j$ to an embedding of universal covers. By Definition~\ref{defn:Z} we have $\A_\na(\Lambdapp) = \{[A_1],\ldots,[A_J]\}$ where $\wt K_j$ is the minimal subtree of the action of $A_j$ on $\wt G$, and where $\bdy \wt K_j = \bdy A_j \subset \bdy\wt G = \bdy F_n$. Given a line $\ell \in \B(G)$, the first sentence of item~\pref{item:ZP=NA} follows from the fact that a line is carried by $\A_\na(\Lambdapp)$ if and only if $\ell$ lifts to a line with both endpoints in some $\bdy A_j$, if and only $\ell$ lifts via some $\ti h_j$ to $\wt K_j$, if and only if $\ell$ lifts via $h$ to some $K_j$, if and only if $\ell \in \<Z,\hat\rho\>$. Given a ray $\rho \in \wh\B(G)$, the second sentence of item~\pref{item:ZP=NA} follows from the fact that the ray is carried by $\A_\na(\Lambdapp)$ if and only if $\rho$ lifts to a ray with end in some $\bdy A_j$, if and only if some subray of $\rho$ lifts via some $\ti h_j$ to $\wt K_j$, if and only if some subray of $\rho$ lifts via $h$ to some $K_j$, if and only if some subray of $\rho$ is in $\<Z,\hat\rho\>$. 

Item~\pref{Item:circuits} follows from \pref{item:ZP=NA} using the natural bijection between periodic lines and circuits. Item~\pref{item:closed} follows from~\pref{Item:BEQuivZP} and Fact~\refGM{FactLinesClosed}. Item \pref{item:UniqueLift} follows from Proposition~\ref{PropVerySmallTree}~\pref{ItemA_naMalnormal} and Fact~\refGM{FactBoundaries}.
\end{proof}

The following lemma is based on Proposition~6.0.4 and Corollary~6.0.7 of \BookOne.


\begin{lemma} \label{defining Z} Assuming the notation of Definitions~\ref{defn:Z}, we have: 
\begin{enumerate}
\item\label{ItemZPEdgesInv}
If $E$ is an edge of $Z$ then 
$f_\#(E) \in \langle Z,\hat \rho_r \rangle $.
\item\label{ItemZPPathsInv}
$\<Z,\hat \rho_r\>$ is $f_\#$-invariant.
\item\label{ItemZPAnyPaths}
If $\sigma \in\<Z,\hat \rho_r\>$ then $\sigma$ is not weakly attracted to $\Lambdapp$.
\item\label{ItemZPFinitePaths}
For any finite path $\sigma$ in $G$ with endpoints at fixed vertices, the converse to \pref{ItemZPAnyPaths} holds: if $\sigma$ is not weakly attracted to $\Lambda$ then $\sigma \in \<Z,\hat \rho_r\>$.
\item \label{ItemBijection} $f_\#$ restricts to bijections of the the following sets: lines in $\<Z,\hat \rho_r\>$; finite paths in $\<Z,\hat \rho_r\>$ whose endpoints are fixed by $f$; and circuits in $\<Z, \hat\rho_r\>$.
\end{enumerate}
\end{lemma} 

\begin{proof} In this proof we shall freely use that $\<Z,\hat\rho_r\>$ is a groupoid, Lemma~\ref{LemmaZPClosed}~\pref{Item:Groupoid}.

$\<Z,\hat \rho_r\>$ contains each fixed or linear edge by construction. Given an indivisible Nielsen path $\rho_i$ of height $i$, we prove by induction on $i$ that $\rho_i$ is in $\<Z,\hat\rho_r\>$. If $H_i$ is \neg\ this follows from (\neg\ Nielsen Paths) and the induction hypothesis. If $H_i$ is \eg\ then Fact~\trefGM{FactNielsenBottommost}{ItemBottommostEdges} applies to show that $H_i \subset Z$; combining this with Fact~\trefGM{FactNielsenBottommost}{ItemBottommostEdges} again and with the induction hypothesis we conclude that $\rho_i \in \<Z,\hat\rho_r\>$. 

Since all indivisible Nielsen paths and all fixed edges are contained in $\<Z,\hat\rho_r\>$, it follows that all Nielsen path are contained in $\<Z,\hat\rho_r\>$, which immediately implies that $\<Z,\hat \rho_r\>$ contains all exceptional paths. 

Suppose that $\tau = \tau_1 \cdot \ldots \cdot \tau_m$ is a complete splitting of a finite path that is not contained in a zero stratum. Each $\tau_i$ is either an edge in an irreducible stratum, a taken connecting path in a zero stratum, or, by the previous paragraph, a term which is not weakly attracted to $\Lambda$ and which is contained in $\<Z,\hat\rho_r\>$. If $\tau_i$ is a taken connecting path in a zero stratum $H_t$ that is enveloped by an \eg\ stratum $H_s$ then, by definition of complete splitting, $\tau_i$ is a maximal subpath of $\tau$ in $H_t$; since $\tau \not\subset H_t$ it follows that $m \ge 2$, and by applying (Zero Strata) it follows that at least one other term $\tau_j$ is an edge in $H_s$. In conjunction with the second Remark in Definitions~\ref{defn:Z}, this proves that $\tau$ is contained in $\<Z,\hat \rho_r\>$ if and only if each $\tau_i$ that is an edge in an irreducible stratum is contained in $Z$ if and only if $\tau$ is not weakly attracted to $\Lambdapp$. 
 
We apply this in two ways. First, this proves item \pref{ItemZPFinitePaths} in the case that $\sigma$ is completely split. Second, applying this to $\tau = f_\#(E)$ where $E$ is an edge in $Z$, item \pref{ItemZPEdgesInv} follows in the case that $f_\#(E)$ is not contained in any zero stratum. Consider the remaining case that $\tau=f_\#(E)$ is contained in a zero stratum $H_t$ enveloped by the \eg\ stratum $H_s$. By definition of complete splitting, $\tau=\tau_1$ is a taken connecting path. By Fact~\refGM{FactEdgeToZeroConnector} the edge $E$ is contained in some zero stratum $H_{t'}$ enveloped by the same \eg\ stratum $H_s$. Since $E \subset Z$, it follows that $H_s \subset Z$, and so $H^z_s \subset Z$, and so $\tau \subset Z$, proving \pref{ItemZPEdgesInv}.

Item \pref{ItemZPPathsInv} follows from item \pref{ItemZPEdgesInv}, the fact that $f_\#(\hat \rho_r) = \hat \rho_r$ and the fact that $\<Z,\hat \rho_r\>$ is a groupoid. 

 Every generic leaf of $\Lambdapp$ contains subpaths in $H_r$ that are not subpaths of $\hat \rho_r$ or $\hat \rho_r^{-1}$ and hence not subpaths in any element of $\<Z,\hat \rho_r\>$. Item \pref{ItemZPAnyPaths} therefore follows from item \pref{ItemZPPathsInv}. 
 
To prove \pref{ItemBijection}, for lines and finite paths the implication $\pref{ItemZPPathsInv} \Rightarrow \pref{ItemBijection}$ follows from Corollary~6.0.7 of \BookOne. For circuits, use the natural bijection between circuits and periodic lines, noting that this bijection preserves membership in $\<Z,\hat\rho_r\>$.
 
 It remains to prove \pref{ItemZPFinitePaths}. By \pref{ItemBijection}, there is no loss of generality in replacing $\sigma$ with $f^k_\#(\sigma)$ for any $k \ge 1$. By Fact~\refGM{FactEvComplSplit} this reduces \pref{ItemZPFinitePaths} to the case that $\sigma$ is completely split which we have already proved.
\end{proof} 


\subsection{Applications and properties of the nonattracting subgroup system.}
\label{SectionAppsAndPropsANA}
We now show that the nonattracting subgroup system $\A_\na(\Lambdapp)$ deserves its name.

\begin{corollary}\label{CorNAWellDefined} For any rotationless $\phi \in \Out(F_n)$ and $\Lambdapp \in \L(\phi)$, a conjugacy class $[a]$ in $F_n$ is not weakly attracted to $\Lambdapp$ if and only if it is carried by $\A_{\na}(\Lambdapp)$. 
\end{corollary}

\begin{proof} Let $f \from G \to G$ be a \ct\ representing a rotationless $\phi \in \Out(F_n)$ and assume the notation of Definitions~\ref{defn:Z}. By Lemma~\ref{LemmaZPClosed}~\pref{Item:circuits}, it suffices to show that a circuit in $G$ is not weakly attracted to $\Lambdapp$ under iteration by $f_\#$ if and only if it is carried by $\<Z,\hat\rho_r\>$. Both the set of circuits in $\<Z,\hat\rho_r\>$ and the set of circuits that are not weakly attracted to $\Lambdapp$ are $f_\#$-invariant. We may therefore replace $\sigma$ with any $f^k_\#(\sigma)$ and hence may assume that $\sigma$ is completely split. After taking a further iterate, we may assume that some coarsening of the complete splitting of $\sigma$ is a splitting into subpaths whose endpoints are fixed by~$f$. Lemma~\ref{defining Z}~(\ref{ItemZPFinitePaths}) completes the proof.
\end{proof}

\begin{corollary}\label{CorNASubgroups}
For any rotationless $\phi \in \Out(F_n)$ and $\Lambdapp \in \L(\phi)$, and for any finite rank subgroup $B \subgroup F_n$, if each conjugacy class in $B$ is not weakly attracted to $\Lambdapp$ then there exists a subgroup $A \subgroup F_n$ such that $B \subgroup A$ and $[A] \in \A_\na(\Lambdapp)$.
\end{corollary}

\begin{proof} By Corollary~\ref{CorNAWellDefined} the conjugacy class of every nontrivial element of $B$ is carried by the subgroup system $\A_\na(\Lambdapp)$ which, by Proposition~\ref{PropVerySmallTree}~\pref{ItemA_naNP}, is a vertex group system. Applying Lemma~\refGM{LemmaVSElliptics}, the conclusion follows.
\end{proof}

Using Corollary~\ref{CorNASubgroups} we can now prove some useful invariance properties of $\A_\na(\Lambdapp)$, for instance that $\A_\na(\Lambdapp)$ is an invariant of $\phi$ and~$\Lambdapp$ alone, independent of the choice of \ct\ representing~$\phi$.


\begin{corollary} \label{CorPMna}
For any rotationless $\phi \in \Out(F_n)$ and any lamination $\Lambdapp \in \L(\phi)$ we have:
\begin{enumerate}
\item \label{ItemAnaCharacterization}
The nonattracting subgroup system $\A_\na(\Lambdapp)$ is the unique vertex group system such that the conjugacy classes it carries are precisely those which are not weakly attracted to $\Lambda^+_\phi$ under iteration of $\phi$. 
\item \label{ItemAnaDependence}
$\A_\na(\Lambdapp)$ depends only on $\phi$ and $\Lambdapp$, not on the choice of a \ct\ representing $\phi$.
\item \label{ItemAnaNaturality}
The dependence in \pref{ItemAnaDependence} is natural in the sense that if $\theta \in \Out(F_n)$ then $\theta(\A_\na(\Lambda^+_\phi)) = \A_\na(\Lambda^+_{\theta\phi\theta^\inv})$ where $\Lambda^+_{\theta\phi\theta^\inv}$ is the image of $\Lambda^+_\phi$ under the bijection $\L(\phi) \mapsto \L(\theta\phi\theta^\inv)$ induced by $\theta$.
\end{enumerate}
\end{corollary}

\begin{proof}  
By Proposition~\ref{PropVerySmallTree}~\pref{ItemA_naNP}, $\A_\na(\Lambdapp)$ is a vertex group system. By Lemma~\refGM{LemmaVSElliptics}, $\A_\na(\Lambdapp)$ is determined by the set of conjugacy classes of elements of $F_n$ that are carried by $\A_\na(\Lambdapp)$, and by Corollary~\ref{CorNAWellDefined} these conjugacy classes are determined by $\phi$ and $\Lambdapp$ alone, independent of choice of a \ct\ representing $\phi$, namely they are the conjugacy classes weakly attracted to $\Lambdapp$ under iteration of~$\phi$. This proves~\pref{ItemAnaCharacterization} and~\pref{ItemAnaDependence}. Item~\pref{ItemAnaNaturality} follows by choosing any \ct\ $f \from G \to G$ representing $\phi$ and changing the marking on $G$ by the conjugator $\theta$ to get a \ct\ representing $\theta\phi\theta^\inv$.
\end{proof}

The following shows that not only is $\A_\na(\Lambdapp)$ is invariant under change of \ct, but it is invariant under inversion of $\phi$ and replacement of $\Lambdapp$ with its dual lamination.

\begin{corollary}\label{CorNAIndOfPM} Given $\phi, \phi^\inv \in \Out(F_n)$ both rotationless, and given a dual lamination pair $\Lambda^+ \in \L(\phi)$,  $\Lambda^-\in \L(\phi^\inv)$, we have $\A_\na(\Lambda^+;\phi) = \A_\na(\Lambda^-;\phi^\inv)$.
\end{corollary}

\subparagraph{Notational remark.} Based on Corollary~\ref{CorNAIndOfPM}, we introduce the notation $\A_\na \Lambda_\phi^\pm$ for the vertex group system $\A_\na(\Lambda^+;\phi) = \A_\na(\Lambda^-;\phi^\inv)$.

\begin{proof} For each nontrivial conjugacy class $[a]$ in $F_n$, we must prove that $[a]$ is weakly attracted to $\Lambda^+$ under iteration by $\phi$ if and only if $[a]$ is weakly attracted to $\Lambda^-$ under iteration by~$\phi^\inv$. Replacing $\phi$ with $\phi^\inv$ it suffices to prove the ``if'' direction. Applying Theorem~\refGM{TheoremCTExistence}, choose a \ct\ $\fG$ representing $\phi$ having a core filtration element $G_r$ such that $[G_r] = \F_\supp(\Lambda^+)$, and so $H_r \subset G$ is the \eg\ stratum corresponding to $\Lambda^+$. We adopt the notation of Definitions~\ref{defn:Z}. 

Suppose that $[a]$ is not weakly attracted to $\Lambda^+$ under iteration by $\phi$. Then the same is true for all $\phi^{-k}([a])$ and so $\phi^{-k}([a]) \in \<Z,\hat\rho_r\>$ for all $k \ge 0$ by Corollary~\ref{CorNAWellDefined}. 

Arguing by contradiction, suppose in addition that $[a]$ is weakly attracted to $\Lambda^-$ under iteration by $\phi^\inv$. Applying Corollary~\ref{LemmaZPClosed}~\pref{item:closed} it follows that a generic line $\gamma$ of $\Lambda^-$ is contained in $\<Z,\hat\rho_r\>$. However, since $\F_\supp(\gamma) = \F_\supp(\Lambda^-) = [G_r]$, it follows that $\gamma$ has height $r$. If $\hat\rho_r$ is trivial then $\gamma$ is a concatenation of edges of $Z$ none of which has height~$r$, a contradiction. If $\hat\rho_r = \rho_r$ is nontrivial then all occurences of edges of $H_r$ in $\gamma$ are contained in a pairwise disjoint collection of subpaths each of which is an iterate of $\rho_r$ or its inverse. By Fact~\refGM{FactEGNielsenCrossings}, at least one endpoint of $\rho_r$ is disjoint from $G_{r-1}$. If $\rho_r$ is not closed then we obtain an immediate contradiction. If $\rho_r$ is closed then $\gamma$ is a bi-infinite iterate of $\rho_r$, but this contradicts \BookOne\ Lemma~3.1.16 which says that no generic leaf of $\Lambda^+_\psi$ is periodic. 
\end{proof}

\subsection{Weak convergence and malnormal subgroup systems.} 
\marginparLee{Michael: This newly numbered section contains the new version of Lemma \ref{ItemNotZPLimit}, which used to be solely about $\A_\na$.}
We conclude Section~\ref{SectionWeakAttraction} with a result which generalizes Fact~\refGM{FactWeakLimitLines} from free factor systems to malnormal subgroup systems. This lemma could have been proved back in Part~1, but the issue first arises here because of applications to $\A_\na(\Lambdapp)$ (see the proof of Corollary~\ref{CorOneWayOrTheOther}~\pref{ItemUniformOWOTO}), malnormality of which is proved in Proposition~\ref{PropVerySmallTree}. 

In this final subsection of Section~\ref{SectionWeakAttraction}, $G$ denotes an arbitrary marked graph. Given a subgroup system $\A=\{[A_1],\ldots,[A_J]\}$, the \emph{Stallings graph} of $\A$ with respect to $G$ is an immersion $f \from \Gamma \to G$ of a finite graph $\Gamma$ whose components are core graphs $\Gamma=\Gamma_1 \union\cdots\union\Gamma_J$ such that the subgroup $A_j<F_n$ is conjugate to the image of $f_* \from \pi_1(\Gamma_j) \to \pi_1(G)\approx F_n$. The Stallings graph is unique up to homeomorphism of the domain. 

Recall from Section~\refGM{SectionLineDefs} that a ``ray of $F_n$'', i.e.\ an element of the set $\bdy F_n / F_n$, is realized in $G$ as a ray $\gamma = E_0 E_1 E_2 \cdots \in \wh\B(G)$ well defined up to asymptotic equivalence. Recall also from Section~\refGM{SectionLineDefs} the weak accumulation set of the ray~$\gamma$, a subset of $\B(G)$. Given a subgroup system $\A$, we say that $\gamma$ is \emph{carried by~$\A$} if there exists a lift $\ti \gamma = \wt E_0 \wt E_1 \wt E_2 \cdots \subset \wt G$ and a subgroup $A \subgroup F_n$ with $[A] \in \A$ such that the end of $\ti \gamma$ is in $\bdy A$. If $\A$ is the free factor system corresponding to a subgraph $H \subset G$ then $\gamma$ is carried by $\A$ if and only if some subray of $\gamma$ is contained in $H$. In the earlier context of Section~\ref{SectionWeakAttraction}, $\gamma$ is carried by $\A_\na(\Lambdapp)$ if and only if some subray of $\gamma$ is an element of $\<Z,\hat\rho_r\>$.

\begin{lemma}
\label{ItemNotZPLimit}
For any marked graph $G$ and any malnormal subgroup system $\A$ we have:
\begin{enumerate}
\item\label{ItemNotMalLimitLine}
Every sequence of lines $\gamma_i \in \B(G)$ not carried by~$\A$ has a subsequence that weakly converges to a line not carried by $\A$.
\item\label{ItemNotMalLimitRay}
The weak accumulation set of every ray not carried by $\A$ contains a line not carried by~$\A$.
\item\label{ItemNotMalFinitePathSet}
For each sufficiently large constant $L$, letting $\Sigma \subset \wh\B(G)$ be the set of all finite paths of length $\le L$ that do not lift to the Stallings graph of $\A$, the following hold: 
\begin{enumerate}
\item\label{ItemNotMalLimitPath} For each sequence $\gamma_i \in \wh\B(G)$ and each decomposition $\gamma_i = \alpha_i * \beta_i * \omega_i$, if $\beta_i \in \Sigma$ for each $i$, and if $\Length(\alpha_i) \to +\infinity$ and $\Length(\omega_i) \to +\infinity$ as $i \to +\infinity$, then some weak limit of some subsequence of $\gamma_i$ is a line not carried by~$\A$.
\item\label{ItemNotMalCarriedLine}
A line $\gamma \in \B(G)$ is not carried by $\A$ if and only if some subpath of $\gamma$ is in $\Sigma$.
\item\label{ItemNotMalCarriedRay}
A ray $\gamma \in \wh\B(G)$ is not carried by $\A$ if and only if $\gamma$ has infinitely many distinct subpaths in $\Sigma$.
\end{enumerate}
\end{enumerate}
\end{lemma}

%

\subparagraph{Remark.} There is also a converse to~\pref{ItemNotMalFinitePathSet}, namely that if $\A$ is not malnormal then no such finite set $\Sigma$ exists.

\begin{proof} First we show that \pref{ItemNotMalFinitePathSet}$\implies$\pref{ItemNotMalLimitLine} and~\pref{ItemNotMalLimitRay}. 

For any sequence of lines $\gamma_i \in \B(G)$ not carried by $\A$, apply~\pref{ItemNotMalCarriedLine} to choose a subpath $\beta_i$ of $\gamma_i$ not in $\Sigma$. We obtain a decomposition $\gamma_i = \alpha_i \beta_i \omega_i$ with $\Length(\alpha_i)=\Length(\omega_i)=\infinity$, and so \pref{ItemNotMalLimitLine} follows from~\pref{ItemNotMalLimitPath}. For any ray $\gamma \in \wh\B(G)$ not carried by $\A$, apply~\pref{ItemNotMalCarriedRay} to choose infinitely many distinct subpaths $\beta_i$ of $\gamma$ not in $\Sigma$. We obtain for each $i$ a decomposition $\gamma = \alpha_i \beta_i \omega_i$ with $\Length(\alpha_i) \to \infinity$ as $i \to \infinity$, and with $\Length(\omega_i) = \infinity$, and so \pref{ItemNotMalLimitRay} follows from~\pref{ItemNotMalLimitPath} with $\gamma_i=\gamma$. 

We turn to the proof of~\pref{ItemNotMalFinitePathSet}. Let $f \from \Gamma=\Gamma_1,\ldots,\Gamma_J \to G$ be the Stallings graph of $\A=\{[A_1],\ldots,[A_J]\}$.
Let $\T$ be the set of all minimal subtrees $T \subset \wt G$ with respect to the actions of all subgroups $A \subgroup F_n$ representing elements of $\A$, so the trees $T$ are precisely the images of all the lifts of all the maps $f \from \Gamma_j \to G$ to universal covers, $j=1,\ldots,J$. Since $\A$ is malnormal, there exists a constant $D$ such that for each $T \ne T' \in \T$ the diameter of $T \intersect T'$ is $\le D$ (see Section~\refGM{SectionSSAndFFS}). Consider any $L \ge 2D+2$ and let~$\wt\Sigma$ be the set of all paths $\ti\beta$ in $\wt G$ such that $\Length(\ti\beta) \le L$ and $\ti\beta \not\subset T$ for each $T \in \T$. The set $\Sigma$ of projections to $G$ of all paths in $\wt\Sigma$ is precisely the set of all paths in $G$ of length $\le L$ that do not lift to the Stallings graph of $\A$. 

The $\Leftarrow$ direction of~\pref{ItemNotMalCarriedLine} follows by observing that if a line is carried by~$\A$ then it lifts to some tree $T \in \T$ and so every subpath lifts to~$T$, hence no subpath is in $\Sigma$. For the~$\Rightarrow$ direction, suppose that no subpath of $\gamma$ is in~$\Sigma$, and so every subpath of length $\le L$ lifts to a path in one of the trees in $\T$. Decompose $\gamma$ as a bi-infinite concatenation of subpaths of length $D+1$
$$\gamma = \cdots \beta_{-1} \beta_0 \beta_1 \beta_2 \cdots
$$
Each of the subpaths $\beta_{i-1} \beta_i$ has length $2D+2 \le L$ and so lifts to one of the trees in $\T$. Choose a tree $T \in \T$ and a lift $\ti\beta_0 \ti\beta_1 \subset T$ of the subpath $\beta_0 \beta_1$. Proceeding by induction, suppose that $0 < i$ and that we have extended $\ti\beta_0 \ti\beta_1$ to a lift 
$$\ti\beta[-i+1,i] = \ti\beta_{-i+1} \cdots \ti\beta_0 \ti\beta_1 \cdots \ti\beta_i \subset T
$$
of the subpath $\beta[-i+1,j] = \beta_{-i+1} \cdots \beta_0 \beta_1 \cdots \beta_i$. Choose a tree $T' \in \T$ and a lift $\ti\beta'_i \ti\beta'_{i+1}$ of the subpath $\beta_i\beta_{i+1}$. Since $\ti\beta_i,\ti\beta'_i$ are two lifts of the same path, there exists $g \in F_n$ with corresponding covering transformation $\tau_g \from \wt G \to \wt G$ such that $\tau_g(\ti\beta'_i) = \ti\beta_i$. It follows that $\tau_g(T') \in \T$ and that the diameter of $T \intersect \tau_g(T')$ is greater than or equal to $\Length(\ti\beta_i) = D+1$ and therefore $\tau_g(T')=T$. The path $\ti\beta_{-i+1} \cdots \ti\beta_0 \ti\beta_1 \cdots \ti\beta_i \tau_g(\ti\beta'_{i+1})$ is therefore a lift of $\beta[-i+1,i+1]$ extending $\ti\beta_0\ti\beta_1$. In a similar fashion, by choosing a tree in $\T$ containing a lift of $\beta_{-i} \beta_{-i+1}$, we obtain a lift of $\beta[-i,i+1]$ extending $\ti\beta_0 \ti\beta_1$, completing the induction. Taking the union as $i \to \infinity$ we obtain a lift of $\gamma$ to $T$, and so $\gamma$ is carried by~$\A$.

The proof of~\pref{ItemNotMalCarriedRay} is almost the same as~\pref{ItemNotMalCarriedLine}, the primary difference being that if a ray $\gamma \in \wh\B(G)$ has only finitely many distinct subpaths in $\Sigma$ then after truncating some initial segment we may assume that $\gamma$ has no subpaths in $\Sigma$, and so after that truncation we may write $\gamma$ as a singly infinite concatenation $\gamma = \beta_0 \beta_1 \beta_2 \cdots$ of paths of length $D+1$; then proof then proceeds inductively as above.

To prove~\pref{ItemNotMalLimitPath}, since $\Sigma$ is finite we may pass to a subsequence of $\gamma_i$ so that $\beta_i=\beta$ is constant. By a diagonalization argument as in the proof of Fact~\refGM{FactWeakLimitLines} one obtains a subsequence of $\gamma_i$ that weakly converges to a line containing $\beta$ as a subpath, and so by~\pref{ItemNotMalCarriedLine} that line is not carried by~$\A$.
\end{proof}

\section{Nonattracted lines}
\label{SectionNALines}
In the previous section, given a rotationless $\phi \in \Out(F_n)$ and $\Lambdapp \in \L(\phi)$, we described the set of conjugacy classes that are not weakly attracted to $\Lambdapp$ under iteration by $\phi$ --- they are precisely the conjugacy classes carried by the nonattracting subgroup system $\A_\na(\Lambda^\pm_\phi)$.  

In this section we state and prove Theorem~\ref{ThmRevisedWAT}, the full fledged version of Theorem~G from the Introduction, which characterizes those lines that are not weakly attracted to $\Lambdapp$ under iteration by $\phi$. Our characterization starts with Lemma~\ref{LemmaThreeNASets} that lays out three particular types of such lines: lines carried by $\A_\na(\Lambda^\pm_\phi)$; singular lines of $\phi^\inv$; and generic leaves of laminations in $\L(\phi^\inv)$. Theorem~\ref{ThmRevisedWAT} will say that, in addition to these three subsets, by concatenating elements of these subsets in a very particular manner one obtains the entire set of lines not weakly attracted to $\Lambdapp$. The proof of this theorem will occupy the remaining subsections of Section~\ref{SectionWeakAttraction}.

The proof of Theorem~\ref{ThmRevisedWAT} requires a re-examination of the weak attraction theory of mapping classes of surfaces, based on Nielsen-Thurston theory, and carried out in Section~\ref{SectionNFHGeometric}. One ``folk theorem'' in this context is that for any finite type compact surface $S$ with exactly one boundary component, and for any mapping class $\phi  \in \MCG(S) \subgroup \Out(\pi_1 S)$, if $\phi$ is a pseudo-Anosov element of $\MCG(S)$ then $\phi$ is a fully irreducible element of $\Out(\pi_1 S)$. In Proposition~\ref{PropWeakGeomRelFullIrr} we will prove this folk theorem in a very general context that is expressed in terms of geometric models. In \PartFour, the conclusions of Proposition~\ref{PropWeakGeomRelFullIrr} will be incorporated into Theorem~J, which is the full fledged version of Theorem~I from the introduction.

\subsection{Theorem G --- Characterizing nonattracted lines}
\label{SectionTheoremGStatement}
From here up through Section~\ref{SectionNonattrFullHeight} we adopt the following:

\smallskip

\textbf{Notational conventions:} Let $\phi, \psi=\phi^{-1} \in \Out(F_n)$ be rotationless, let $\Lambda^\pm_\phi \in \L^\pm(\phi)$ be a lamination pair, and denote $\Lambda^+_\psi= \Lambda^-_\phi  \in \L(\psi)$ and $\Lambda^-_\psi = \Lambda^+_\phi \in \L(\phi)$. Applying \recognition\ Theorem~4.28 (or see Theorem~\refGM{TheoremCTExistence}), choose $f \from G \to G$ and $f' \from G' \to G'$ to be \cts\ representing $\phi$ and $\psi$, respectively, the first with \eg\ stratum $H_r \subset G$ associated to $\Lambda^+_\phi$, and the second with \eg\ stratum $H'_u \subset G'$ associated to $\Lambda^+_\psi$, so that 
\begin{align*}
[G_r] = \F_\supp(\Lambda^+_\phi) &= \F_\supp(\Lambda^-_\psi) = [G'_u] \\
[G_{r-1}] &= [G'_{u-1}]
\end{align*}
To check that this is possible, after choosing $f \from G \to G$ to satisfy the one condition $[G_r]=\F_\supp(\Lambda^\pm_\phi)$ we may then choose $f'$ to satisfy the two conditions $[G'_u]=[G_r]$ and $[G'_{t}] = [G_{r-1}]$ for some $t < u$, but then by (Filtration) in Definition~\refGM{DefCT} it follows that $[G'_{t}] = [G'_{u-1}]$. For other laminations in the set $\L(\psi)$, or strata or filtration elements of $G'$ that occur in the course of our presentation, we use notation like $\Lambda^-_t$, or $H'_t$ or $G'_t$ with the subscript~$t$, as in the previous paragraph.

\bigskip

The reader may refer to Section~\refGM{SectionPrincipalRotationless} for a refresher on basic concepts regarding the set $P(\phi)$ of principal automorphisms representing $\phi \in \Out(F_n)$, and on the set $\Fix(\wh\Phi) \subset \bdy F_n$ of points at infinity fixed by the continuous extension $\wh\Phi \from \bdy F_n \to \bdy F_n$ of an automorphism $\Phi \in \Aut(F_n)$.

We also recall/introduce some notations and definitions related to a rotationless outer automorphism $\psi \in \Out(F_n)$.
\begin{itemize}
\item $\B_\na(\Lambdapp) = \B_\na(\Lambdapp;\phi)$ denotes set of all lines in $\B$ that are not weakly attracted to $\Lambdapp$ under iteration by~$\phi$.
\item $\B_\sing(\psi)$ denotes the set of singular lines of $\psi$: by definition, $\ell \in \B$ is a singular line for $\psi$ if there exists $\Psi \in P(\psi)$ ($=$ the set of principal automorphisms representing $\psi$) such that $\bdy\ell \subset \Fix_N(\Psi)$.
\item $\B_\gen(\psi)$ denotes the set of all generic leaves of all elements of $\L(\psi)$. 
\end{itemize}


\begin{lemma} \label{LemmaThreeNASets}  Given rotationless $\phi,\psi=\phi^{-1} \in \Out(F_n)$ and $\Lambdapp \in \L(\phi)$, if $\gamma \in \B$ satisfies any of the following three conditions then $\gamma \in \B_\na(\Lambdapp)$:
\begin{enumerate}
\item\label{ItemNANA}
 $\gamma$ is carried by  $\A_{\na}(\Lambdapmp)$.
\item\label{ItemSingNA}
$\gamma \in \B_\sing(\psi)$ 
\item\label{ItemGenNA}
$\gamma \in \B_\gen(\psi)$.
\end{enumerate}

\end{lemma}

\begin{proof} Case~\pref{ItemNANA} is a consequence of the following: no conjugacy class carried by $\A_\na(\Lambdapmp)$ is weakly attracted to $\Lambdapp$ (Corollary~\ref{CorNAWellDefined}); axes of conjugacy classes are dense in the set of all lines carried by $\A_\na(\Lambdapmp)$ (as is true for any subgroup system); and $\B_\na(\Lambdapp)$ is a weakly closed subset of $\B$, which follows from the fact that being weakly attracted to $\Lambdapp$ is a weakly open condition on $\B$, an evident consequence of the definition of an attracting lamination.

For Case~\pref{ItemGenNA}, suppose that $\gamma$, and hence each $\phi^i_\#(\gamma)$, is a generic leaf of some $\Lambda^-_t \in \L(\psi)$. Choose $[a]$ to be a conjugacy class represented by a completely split circuit in $G'$ such that some term of its complete splitting is an edge of $H'_t$. By Fact~\trefGM{FactAttractingLeaves}{ItemSplitAnEdge}, $[a]$ is weakly attracted to $\gamma$ under iteration by~$\psi$. If $\gamma$ were weakly attracted to $\Lambdapp$ under iteration by $\phi$ then, since the $\phi^i_\#(\gamma)$'s all have  the same neighborhoods in~$\B$, the lamination $\Lambdapp$ would be in the closure of $\gamma$, and so $[a]$ would be weakly attracted to $\Lambdapp$ under iteration by $\psi=\phi^\inv$, contradicting Fact~\trefGM{FactAttractingLeaves}{ItemAllCircuitsRepelled}.

For Case~\pref{ItemSingNA}, choose $\Psi \in P(\psi)$  and a lift $\ti \gamma$ of $\gamma$ with endpoints in $\Fix_N(\wh\Psi) = \bdy\Fix(\Psi) \union \Fix_+(\wh\Psi)$. Assuming that $\gamma$ is  weakly attracted to $\Lambdapp$ under iteration by $\phi$, we argue to a contradiction. Since $\gamma$ is $\phi_\#$-invariant,  $\Lambdapp$ is contained in the weak closure of $\gamma$. Let $\ell$ be a generic leaf of~$\Lambdapp$. Since $\ell$ is birecurrent, $\ell$ is contained in the weak accumulation set of at least one of the endpoints, say~$P$, of $\ti \gamma$.  If $P \in \partial \Fix(\Psi)$ then $\ell$ is carried by ${\A}_{\na}(\Lambdapmp)$ in contradiction to Case~\pref{ItemNANA} and the obvious fact that $\ell$ is weakly attracted to~$\Lambdapp$. Thus $P \in \Fix_+(\wh\Psi)$. By Lemma~\refGM{LemmaFixPlusAccumulation}, for every line in the weak accumulation set of $P$, in particular for the generic leaf $\ell$ of $\Lambdapp$, there exists a conjugacy class $[a]$ whose iterates $\psi^k[\alpha]$ weakly accumulate on that line. It follows that $[a]$ is weakly attracted to $\Lambdapp$ under iteration by $\psi$. As in Case~\pref{ItemGenNA}, this contradicts Fact~\trefGM{FactAttractingLeaves}{ItemAllCircuitsRepelled}.   
\end{proof}

As shown in Example~\ref{ExNotClosedUnderConcat} below, by using concepts of concatenation one can sometimes construct lines in $\B_\na(\Lambda)$ not accounted for in the statement of Lemma~\ref{LemmaThreeNASets}. In the next definition we extend the usual concept of concatenation points to allow points at infinity.

\begin{definition}  \label{DefnConcatenation} Given any marked graph $K$ and oriented paths $\gamma_1,\gamma_2 \in \wh\B(K)$, we that $\gamma_1,\gamma_2$ are \emph{concatenable} if there exist lifts $\ti\gamma_i \subset \wt K$ with initial endpoints $P^-_i \in \wt K \union \bdy F_n$ and terminal endpoints $P^+_i \in \wt K \union \bdy F_n$ satisfying $P^+_1 = P^-_2$ and $P^-_1 \ne P^+_2$. The \emph{concatenation} of $\ti\gamma_1,\ti\gamma_2$ is the oriented path with endpoints $P^-_1, P^+_2$, denoted $\ti\gamma_1 \diamond \ti\gamma_2$. Its projection to~$K$,  denoted $\gamma_1 \diamond \gamma_2$, is called \emph{a concatenation of $\gamma_1,\gamma_2$}. This operation is clearly associative and so we can define multiple concatenations. This operation is also invertible, in particular any concatenation of the form $\gamma = \alpha \diamond \nu \diamond \beta$ can be rewritten as $\nu = \alpha^\inv \diamond \gamma \diamond \beta^\inv$.

Notice the use of the definite article upstairs, versus the indefinite article downstairs. ``The'' upstairs concatenation $\ti\gamma_1 \diamond \ti\gamma_2$ is well-defined once the lifts $\ti\gamma_1,\ti\gamma_2$ have been chosen. But because of the freedom of choice of those lifts, ``a'' downstairs concatenation $\gamma_1 \diamond \gamma_2$ is not generally well-defined: this fails precisely when $P^+_1=P^-_2$ is an endpoint of the axis of some element $\gamma$ of $F_n$ and neither $P^-_1$ nor $P^+_2$ is the opposite endpoint, in which case one can replace either of $\ti\gamma_1,\ti\gamma_2$ by a translate under $\gamma$ to get a different concatenation downstairs. This is a mild failure, however, and it is usually safe to ignore. 

A subset of $\wh\B(K)$ is \emph{closed under concatenation} if for any oriented paths $\gamma_1,\gamma_2 \in \wh\B(K)$, any of their concatenations $\gamma_1 \diamond \gamma_2$ is an element of $\wh\B(K)$.
\end{definition}

\begin{lemma}  \label{LemmaConcatenation} Continuing with the Notational Convention above, the set of elements of $\wh\B(G)$ that are not weakly attracted to $\Lambdapp$ under iteration by $\phi$ is closed under concatenation. In particular, $\B_\na(\Lambdapp)$ is closed under concatenation.
\end{lemma}

\begin{proof} Consider a concatenation $\gamma_1 \diamond \gamma_2$ with accompanying notation as in Definition~\ref{DefnConcatenation}. For each $m \ge 0$, the path $f^m_\#(\gamma_1 \diamond \gamma_2)$ is the concatenation of a subpath of $f^m_\#(\gamma_1)$ and a subpath of $f^m_\#(\gamma_2)$. Letting $\ell$ be a generic leaf of $\Lambdapp$, by Fact~\refGM{FactTiles} we may write $\ell$ as an increasing union of nested tiles $\alpha_1 \subset \alpha_2 \subset \cdots$ so that each $\alpha_j$ contains at least two disjoint copies of $\alpha_{j-1}$. By assumption $\gamma_1$ has the property that there exists an integer $J$ so that if $\alpha_j$ occurs in $f^m_\#(\gamma_1)$ for arbitrarily large $m$ then $j \le J$, and $\gamma_2$ satisfies the same property. This property is therefore also satisfied by $\gamma_1 \diamond \gamma_2$ (with a possibly larger bound $J$) and so $\gamma_1 \diamond \gamma_2$ is not weakly attracted to $\Lambdapp$.
\end{proof}

\begin{ex} \label{ExNotClosedUnderConcat} The set of lines satisfying \pref{ItemNANA}, \pref{ItemSingNA} and~\pref{ItemGenNA} of Lemma~\ref{LemmaThreeNASets} is generally not closed under concatenation. For example, suppose that for $i=1,2$ we have singular lines for $\psi$ of the form $\gamma'_i = \bar \alpha'_i \beta'_i \subset G'$ where $\alpha'_i\subset G'_u$ is a principal ray of $\Lambda^+_\psi$ (Definition~\refGM{DefSingularRay}) and $\beta'_i \subset G'_{u-1}$.   Let $\mu' \subset G'_{u-1}$ be any line that is asymptotic in the backward direction to $\beta'_1$ and in the forward direction to $\beta'_2$.  Then $\mu'$ is carried by $\A_{\na}(\Lambdapmp)$ and $\gamma'_3 =\gamma'_1 \diamond \mu' \diamond \bar \gamma'_2$ is not weakly attracted to $\Lambdapp$.  However, $\gamma'_3$ does not in general satisfy  any of \pref{ItemNANA}, \pref{ItemSingNA} and~\pref{ItemGenNA} of Lemma~\ref{LemmaThreeNASets}.    
\end{ex} 
 
We account for these kinds of examples as follows.  (See also Propositions~\ref{nonGeometricFullHeightCase} and \ref{geometricFullHeightCase}.) 
 
\begin{definition} \label{defn: extended na} Given a subgroup $A \subgroup F_n$ such that $[A] \in \A_{\na}(\Lambdapmp)$ and given $\Psi \in P(\psi)$, we say that $\Psi$ is \emph{$A$-related} if $\Fix_N(\wh\Psi) \cap \partial A \ne \emptyset$. Define the \emph{extended boundary} of $A$ to be 
$$\bdy_\ext(A,\psi) = \partial A \cup \left( \bigcup_\Psi \Fix_N(\wh\Psi)\right)
$$
where the union is taken over all $A$-related $\Psi \in P(\psi)$. Let $\B_\ext(A,\psi)$ denote the set of lines that have lifts with endpoints in $\bdy_\ext(A,\psi)$; this set is independent of the choice of $A$ in its conjugacy class. Define  
$$\B_\ext(\Lambdapmp;\psi)= \bigcup_{A \in \A_{\na}(\Lambdapmp)}  \B_\ext(A,\psi)
$$
\end{definition}
\noindent
Basic properties of $\B_\ext(\Lambdapmp;\psi)$ are established in the next section. 

We conclude this section with the statement of our main weak attraction result. The proof is given in Section~\ref{SectionNALinesGeneral}. 

\begin{theorem}[\textbf{Theorem G}]
\label{ThmRevisedWAT} 
If $\phi, \psi=\phi^{-1} \in \Out(F_n)$ are rotationless and $\Lambdapmp \in \L^\pm(\phi)$ then
$$\B_\na(\Lambdapp,\phi) = \B_\ext(\Lambdapmp;\psi) \union \B_\sing(\psi) \union \B_\gen(\psi)
$$
\end{theorem}

\smallskip

\begin{remark} The sets $\B_\ext(\Lambdapmp;\psi)$, $\B_\sing(\psi)$, and $\B_\gen(\psi)$ need not be pairwise disjoint. For example, every line carried by $G'_{u-1}$ is in $\B_\ext(\Lambdapmp;\psi)$ and some of these can be in $\B_\sing(\psi)$ or in $\B_\gen(\psi)$.
\end{remark}
  
\begin{remark} It is not hard to show that if $\gamma \in  \B_\na(\Lambdapp)$ is birecurrent then $\gamma$ is either carried by $A_{\na}(\Lambdapmp)$ or is a generic leaf of some element of $\L(\psi)$. This shows that Theorem~\ref{ThmRevisedWAT} contains the Weak Attraction Theorem (Theorem~6.0.1 of \BookOne) as a special case.
\end{remark}

\subsection{$\B_\ext(\Lambdapmp;\psi) \union \B_\sing(\psi) \union \B_\gen(\psi)$ is closed under concatenation}
\label{SectionNAConcatenation} 

We continue with the notation for an inverse pair of rotationless outer automorphisms $\phi, \, \psi=\phi^\inv \in \Out(F_n)$ established at the beginning of Section~\ref{SectionTheoremGStatement}. 

Much of the work in this section is devoted to revealing details of the structure of $\B_\ext(\Lambdapmp;\psi)$. After a few such lemmas/corollaries, the main result of this section is that the union of the three subsets of $\B_\na(\Lambdapp)$ occurring in Theorem~\ref{ThmRevisedWAT} is closed under concatenation; see Proposition~\ref{PropStillClosed}. 

We shall abuse notation for elements of the set $\A_\na(\Lambdapmp)$ as described in Section~\refGM{SectionSSAndFFS}, writing $A \in \A_\na(\Lambdapmp)$ to mean that $A$ is a subgroup of $F_n$ whose conjugacy class $[A]$ is an element of the set $\A_\na(\Lambdapmp)$. Since $\A_\na(\Lambda^+_\phi)$ is a malnormal subgroup system (Proposition~\ref{PropVerySmallTree}), this notational abuse should not cause any confusion. 

\begin{lemma} \label{LemmaTwoLifts} If $\Psi_1 \ne \Psi_2 \in \Aut(F_n)$ are representatives of $\psi$ and $P \in \Fix(\wh\Psi_1) \cap \Fix(\wh\Psi_2)$ then there exists a nontrivial $a \in \Fix(\Psi_1) \intersect \Fix(\Psi_2)$ determining an inner automorphism~$i_a$, and there exists $A \in A_{\na}(\Lambdapmp)$, such that $a \in A$ and $P \in \Fix(\hat i_a) \subset \bdy A$.
\end{lemma}

\begin{proof} Choosing $a$ so that $\Psi_1 = i_a \Psi_2$ it follows that $i_a = \Psi_1 \Psi_2^\inv$ fixes $P$, and so $a \in \Fix(\Psi_1) \intersect \Fix(\Psi_2)$ (Fact~\refGM{FactTwoLifts}), implying that $[a]$ is not weakly attracted to $\Lambdapp$. Applying Corollary~\ref{CorNAWellDefined}, the conjugacy class $[a]$ is carried by $\A_\na(\Lambdapmp)$, and so there exists $A \in \A_\na(\Lambdapmp)$ such that $a \in A$, which implies that $\Fix(\hat i_a) \subset \bdy A$.
\end{proof}

\begin{lemma}\label{LemmaAtMostOneA}  
For each $\Psi \in P(\psi)$ there exists at most one $A \in \A_\na(\Lambda^\pm_\phi)$ such that $\Psi$ is $A$-related.
\end{lemma}

\begin{proof} Suppose that for $j=1,2$ there exist $A_j \in \A_{\na}(\Lambdapmp)$ and $P_j \in \Fix_N(\wh\Psi) \cap \partial A_j$. The line $\ti\gamma$ connecting $P_1$ to $P_2$ projects to a line $\gamma \in \B_\sing(\psi)$ that by Lemma~\ref{LemmaThreeNASets} is not weakly attracted to $\Lambda^+$.  Since $P_j \in \partial A_j$, and since by Lemma~\ref{LemmaZPClosed}~\pref{item:ZP=NA} each line that is carried by $A_j$ is contained in $\<Z,\hat \rho_r\>$, the ends of $\gamma$ are contained in $\<Z,\hat \rho_r\>$, and so we may assume that $\gamma = \overline\rho_- \diamond \gamma_0 \diamond \rho_+$ where the rays $\rho_-,\rho_+$ are in $\<Z,\hat \rho_r\>$. After replacing $\gamma$ with a $\phi_\#$-iterate we may also assume that the central subpath $\gamma_0$ has endpoints at fixed vertices.  Since none of $\gamma$, $\rho_-$, $\rho_+$ are weakly attracted to $\Lambdapp$, by Lemma~\ref{LemmaConcatenation} neither is $\gamma_0 = \rho_- \diamond \gamma \diamond \overline \rho_+$. Lemma~\ref{defining Z}~\pref{ItemZPFinitePaths} implies that $\gamma_0$ is contained in  $\<Z,\hat \rho_r\>$ and Lemma~\ref{LemmaZPClosed}~\pref{Item:Groupoid} then shows that $\gamma$ is contained in $\<Z,\hat\rho_r\>$. By Lemma~\ref{LemmaZPClosed}~\pref{item:ZP=NA} it follows that $\gamma$ is carried by   $A_{\na}(\Lambdapmp)$, which means that $\bdy\ti\gamma = \{P_1,P_2\} \subset \partial A_3$ for some $A_3 \in \A_\na(\Lambda^\pm_\phi)$. By Proposition~\ref{PropVerySmallTree}~\pref{ItemA_naMalnormal} and Fact~\refGM{FactBoundaries} it follows that $A_1=A_3=A_2$.
\end{proof}

\begin{corollary}\label{CorRelatedBdy}
If $\Psi \in P(\psi)$, $A \in \A_\na(\Lambdapmp)$, and $\Psi$ is $A$-related, then $\Fix(\Psi) \subgroup A$ and each point of $\Fix_N(\wh\Psi) \setminus \bdy A$ is an isolated attractor for $\wh\Psi$.
\end{corollary}

\begin{proof} If $\Fix(\Psi)$ is trivial then by Lemma~\refGM{LemmaFixPhiFacts} each point of $\Fix_N(\wh\Psi)$ is an isolated attractor and we are done. Otherwise, noting that the conjugacy class of each nontrivial element of $\Fix(\Psi)$ is carried by $\A_\na(\Lambdapmp)$, applying Corollary~\ref{CorNASubgroups} we have $\Fix(\Psi) \subgroup A'$ for some $A' \in \A_\na(\Lambdapmp)$, and so $\bdy\Fix(\Psi) \subset \bdy A'$. It follows that $\Psi$ is $A'$-related. By Lemma~\ref{LemmaAtMostOneA} we have $A'=A$, and applying Lemma~\refGM{LemmaFixPhiFacts} completes the proof.
\end{proof}


\begin{corollary}\label{CorDisjBdyA}   If $A_1 \ne A_2 \in \A_\na(\Lambda^\pm_\phi)$ then $\bdy_\ext (A_1,\psi) \cap \bdy_\ext(A_2,\psi) = \emptyset$.
\end{corollary}

\begin{proof} We assume that $Q \in \bdy_\ext (A_1,\psi)\cap \bdy_\ext (A_2,\psi)$ and argue to a contradiction. After interchanging $A_1$ and $A_2$ if necessary, we may assume by 
Proposition~\ref{PropVerySmallTree}~\pref{ItemA_naMalnormal} and Fact~\refGM{FactBoundaries} 
that $Q \not \in \partial A_1$ and hence that $Q \in \Fix_N(\wh\Psi_1)$ for some $A_1$-related $\Psi_1 \in P(\psi)$. Lemma~\ref{LemmaAtMostOneA} implies that $\Psi_1$ is not $A_2$-related and so $Q \not \in \partial A_2$. The only remaining possibility is that $Q \in \Fix_N(\wh\Psi_2)$ for some $A_2$-related $\Psi_2 \in P(\psi)$. But then Lemma~\ref{LemmaTwoLifts} implies that $Q \in \partial A_3$ for some $A_3 \in \A_\na(\Lambda^\pm_\phi)$, and then Lemma~\ref{LemmaAtMostOneA} implies that $A_1 = A_3 = A_2$.   
\end{proof}

\begin{corollary} \label{psi contained in partial A}  If $\Psi \in P(\psi)$, $A \in \A_\na(\Lambda^\pm_\phi)$, and $\Fix_N(\wh\Psi) \cap \bdy_\ext (A,\psi) \ne \emptyset$ then $\Psi$ is $A$-related; in particular, $\Fix_N(\wh\Psi) \subset \bdy_\ext (A,\psi)$.
\end{corollary}

\begin{proof}  Choose $Q \in \Fix_N(\wh\Psi) \cap \bdy_\ext (A,\psi)$.    If $Q \in \partial A$ we're done so we may assume that $ Q \in \Fix_N(\wh\Psi')$  for some   $A$-related $\Psi'$.   If $\Psi = \Psi'$  we are done.  Otherwise, Lemma~\ref{LemmaTwoLifts} implies that  $Q \in \partial A'$ for some $A' \in \A_\na(\Lambda^\pm_\phi)$ and Corollary~\ref{CorDisjBdyA} implies that $A' = A$ so again we are done. 
\end{proof}
 
\begin{proposition} \label{PropStillClosed} If the oriented lines $\gamma_1, \gamma_2$ are in the set $\B_\ext(\Lambdapmp;\psi) \union \B_\sing(\psi) \union \B_\gen(\psi)$ and are concatenable then any concatenation $\gamma_1 \diamond \gamma_2$ is also in that set. 

More precisely, given lifts $\ti \gamma_j$ with initial and terminal endpoints $P^-_j$ and $P^+_j$ respectively, if $P^+_1 = P^-_2$ and $P^-_1 \ne P^+_2$ then either there exists $\Psi \in P(\psi)$ such that the three points $P^-_1,P^+_1=P^-_2, P^+_2$ are in $\Fix_N(\wh\Psi)$ or there exists $A \in \A_\na(\Lambda^\pm_\phi)$ such that those three points are in $\bdy_\ext(A,\psi)$.
\end{proposition} 

\begin{proof} The first sentence is an immediate consequence of the second, to whose proof we now turn.

\smallskip

\textbf{Case 1:} $\gamma_1 \in \B_\sing(\psi)$. \quad We have $P^-_1,P^+_1 \in \Fix_N(\wh\Psi)$ for some  $\Psi \in P(\psi)$.  There are three subcases. First, if $P^-_2,P^+_2 \in \bdy_\ext(A,\psi)$ for some $A \in \A_\na(\Lambda^\pm_\phi)$ then $P^-_1,P^+_1 \in \bdy_\ext(A,\psi)$  by Corollary~\ref{psi contained in partial A}  and we are done. The second subcase is that $P^-_2,P^+_2 \in \Fix_N(\wh\Psi')$ for some $\Psi' \in P(\psi)$; if $\Psi ' = \Psi$ then we are done; if $\Psi' \ne \Psi$ then $P^+_1=P^-_2 \in \bdy_\ext(A,\psi)$ for some $A \in \A_\na(\Lambda^\pm_\phi)$  by Lemma~\ref{LemmaTwoLifts}, and so $P^-_2,P^+_2 \in \bdy_\ext(A,\psi)$ by Corollary~\ref{psi contained in partial A}  so we are reduced to the first subcase. The final subcase is that  $\gamma_2$ is a generic line of  some  $\Lambda^-_t \in \L(\psi)$. Assuming without loss that  $P^+_2  \not  \in \Fix_N(\wh\Psi)$, the projection of $\Psi_\#(\ti \gamma_2)$ is a generic leaf of $\Lambda^-_t$ that is asymptotic to $\gamma_2$ but not equal to~$\gamma_2$. The next lemma says that this puts us in the second subcase and so we are done.

\begin{lemma}  \label{LemmaGenericAsymptotic}  
Assume that $\theta \in \Out(F_n)$ is rotationless. If $\ell',\ell''$ are each generic lines of elements of $\L(\theta)$, and if some end of $\ell'$ is asymptotic to some end of $\ell''$, then $\ell',\ell'' \in \B_\sing(\theta)$.  
\end{lemma}

\noindent
This lemma extends Lemma 3.3 of \Axes\ in which it is assumed that $\phi$ is irreducible. 

Putting off the proof of Lemma~\ref{LemmaGenericAsymptotic} for a bit, we continue with the proof of Proposition~\ref{PropStillClosed}. Having finished Case~1, by symmetry we may now assume that $\gamma_j \not \in \B_\sing(\psi)$ for $j=1,2$.   

\smallskip

\textbf{Case 2:} $P^-_1,P^+_1 \in \bdy_\ext(A,\psi)$ \textbf{for some} $A \in \A_\na(\Lambda^\pm_\phi)$.\quad   If $P^+_2 \in \bdy_\ext(A,\psi)$ we are done. By Corollary~\ref{CorDisjBdyA}, the only remaining possibility is that $\gamma_2$ is a generic leaf of an element of $\L(\psi)$. If $P^+_1 \in \Fix_N(\wh\Psi)$ for some $\Psi \in  P(\psi)$ then, as shown above, Lemma~\ref{LemmaGenericAsymptotic} implies that $P^+_2 \in \Fix_N(\wh\Psi)$ and hence that $\gamma_2 \in \B_\sing(\psi)$ which is a contradiction. Thus $P^+_1 \in \partial A$. Since $\gamma_2$ is birecurrent, Fact~\refGM{FactLinesClosed} implies that $\gamma_2$ is by carried by~$\A_\na(\Lambdapmp)$. Applying 
Lemma~\ref{LemmaZPClosed}~\pref{item:UniqueLift} 
it follows that $P^+_2 \in \partial A$ and we are done.  

\smallskip

By symmetry of $\gamma_1$ and $\gamma_2$, the only remaining case is:

\textbf{Case 3: $\gamma_1$ and $\gamma_2$ are generic leaves of elements of $\L(\psi)$.}  Since they have asymptotic ends, they are leaves of the same element of $L(\psi)$ so are singular lines by Lemma~\ref{LemmaGenericAsymptotic}.  As we have already considered this case, the proof is complete.
\end{proof}

It remains to prove Lemma~\ref{LemmaGenericAsymptotic}, but we first prove the following, which is similar to Lemma 5.11 of \BH, Lemma 4.2.6 of \BookOne\ and Lemma 2.7 of \Axes.


\begin{lemma}  \label{one illegal turn} Suppose that $\fG$ is a \ct,  that $H_r$ is an \eg\ stratum, that $\gamma \subset G$ is line of height $r$   with   exactly one illegal turn of height $r$ and that $f^{K_\gamma}(\gamma)$ is $r$-legal for some minimal $K_{\gamma}$.  Then $K_\gamma \le K $ for some  $K$ that is independent of $\gamma$.
\end{lemma}

\begin{proof}    If the lemma fails there exists a sequence $\gamma_i$ such that $K_i = K_{\gamma_i} \to \infty$.     Write $\gamma_i  = \bar \sigma_i \tau_i$ where the turn $(\sigma_i, \tau_i)$ is the illegal turn of height $r$.  After passing to a subsequence we may assume that $\sigma_i \to \sigma$ and $\tau_i \to \tau$ for some rays $\sigma$ and $\tau$.  The line $\gamma = \bar \sigma \tau$ has height $r$ and $f^k_\#(\gamma)$ has exactly one illegal turn of height $r$ for all $k \ge 0$.  Lemma~4.2.6 of \BookOne\ implies that there exists $m > 0$ and a splitting $f^m_\#(\gamma) = \bar R^- \cdot \rho\cdot  R^+$ where $\rho$ is the unique \iNp\ of height $r$.      It follows that  for all sufficiently large $i$, $f^m_\#(\gamma_i)$ has a decomposition into subpaths $f^m_\#(\gamma_i)= \bar R_i^- \rho R^+_i$  where the height $r$ illegal turn in $\rho$ is  the only height $r$ illegal turn in $\gamma_i$.    Since any such decomposition is a splitting, $f^k_\#( \gamma_i)$ has an illegal turn of height $r$ for all $k$  in contradiction to our choice of $\gamma_i$.  
\end{proof} 


\paragraph{Proof of Lemma \ref{LemmaGenericAsymptotic}.} By symmetry we need prove only that $\ell' \in \B_\sing(\theta)$. By Fact~\refGM{FactTwoEndsGeneric}, each end of each generic leaf of an element of $\L(\theta)$ has the same free factor support as the whole leaf, and so $\ell'$ and $\ell''$ must be generic leaves of the same $\Lambda \in \L(\theta)$. 

Let $\fG$ be a \ct\ representing $\theta$ and let $H_r$ be the \eg\ stratum corresponding to~$\Lambda$. For each $j \ge 0$, there are generic leaves  $\ell'_{j}$ and $\ell''_{j}$ of $\Lambda$ such that $f_\#^j(\ell'_{j}) =\ell'$ and  $f_\#^j(\ell''_{j}) = \ell''$. Fixing a common end of $\ell'$ and $\ell''$, the corresponding common ends of $\ell'_j$ and $\ell''_j$ determine a maximal common subray $R_j$ of $\ell_j'$ and $\ell_j''$. Denote the rays in $\ell_j'$ and $\ell_j''$ that are complementary to $R_j$ by $R_j'$ and $R_j''$ respectively.  Let $\gamma_j = \bar R_j' R_j''$. 

Suppose at first that each $\gamma_j$ is $r$-legal. Lemma 5.8 of \BH\  implies that  no height $r$ edges of $\gamma_j$ are cancelled when  $f^j(\gamma_j)$ is tightened to $\gamma_0$. Let $E_j,E_j'$ and $E_j''$ be the first height $r$  edges of  $R_j, R_j'$ and $R_j''$ respectively, let $w_j,w'_j$ and $w''_j$ be their initial vertices and let $d_j,d_j'$ and $d_j''$ be their initial directions. Let $\mu'_j$  be the finite subpath of $\ell'_j $ connecting $w_j'$ to $w_j$. To complete the proof in this case we will show that $\mu'_0$ is a Nielsen path, that $R_0$ is the principal ray determined by iterating $d_0$ (see Definition~\refGM{DefSingularRay}) and that $R_0'$ is the principal ray determined by iterating $d_0'$.

If $w_j=w'_j$ then $d_j,d'_j$ determine distinct gates, and otherwise $w_j,w'_j$ are each incident to an edge of height~$<r$. A similar statement holds for $w_j$, $w''_j$. In all cases it follows that $w_j,w'_j$ and $w''_j$ are principal vertices of~$f$.
Moreover, the following hold for all $i,j\ge 0$:
$$  f^{i}(w_{j+i}) = w_j \qquad f^{i}(w'_{j+i}) = w'_j \qquad
 f^{i}(d_{j+i}) = d_j \qquad f^{i}(d'_{j+i}) = d'_j \qquad
 f^{i}_\#(\mu'_{j+i}) = \mu'_j$$
The first two of these equalities imply that $w = w_j, w'=w'_j \in \Fix(f)$ are independent of $j$; the third and fourth imply that $E = E_j$ and $E' = E'_j$ are independent of $j$; in conjunction with Lemma~\trefGM{FactTiles}{ItemUndergraphPieces}, the last equality  implies that $\mu = \mu_j$ is a Nielsen path that is independent of $j$. It follows that $\ell'$ is the increasing union of the subpaths $f^j_\#(\bar E')\mu f^j_\#(E)$ and so $\ell'$ is a pair of principal rays connected by a Nielsen path. Applying Fact~\refGM{FactPrincipalLift} completes the proof that $\ell' \in \B_\sing(\theta)$ when each $\gamma_j$ is $r$-legal.
  
  It remains to consider the case that  that some $\gamma_l$ is not $r$-legal.  Assuming without loss that $\gamma_0$ is not $r$-legal, each $\gamma_j$ is not $r$-legal.   Lemma~\ref{one illegal turn} implies that   $f^k_\#(\gamma_j)$ has an illegal turn of height $r$ for all $k \ge 0$ and Lemma 4.2.6 of \BookOne\ implies that  there is a splitting $\gamma_j  = \tau_j' \cdot \rho_j \cdot \tau_j''$ where some $f_\#$-iterate of $\rho_j$ is the unique indivisible Nielsen path $\rho$ with height $r$.   Since  $f^{i}_\#(\rho_{j+i}) = \rho_j$ for all $i,j \ge 0$, $\rho_j = \rho$ for all $j$.   Let $E'$ be the first edge of height $r$ in the ray $\bar \tau_0'$ and let $E$ be the initial edge of $\rho_0$.  Both of these edges are contained in $\ell'$ and we let    $\mu$ be the subpath of $\ell'$ that connects their initial vertices.  Arguing as in the previous case,  $\mu$ is a Nielsen path and $\ell'$ is the increasing union of the subpaths $f^j_\#(\bar E')\mu f^j_\#(E)$ which proves that $\ell' \in \B_\sing(\theta)$.  \hfill\qed
  
\subsection{Application --- Proof of Theorem H}
\label{SectionAttractionRepulsion}

Before turning in later sections to the proof of Theorem~\ref{ThmRevisedWAT} (Theorem~G), we use it to prove the following:

\begin{corollary}[\textbf{Theorem H}]
\label{CorOneWayOrTheOther}
Given rotationless $\phi,\psi = \phi^\inv \in \Out(F_n)$, a dual lamination pair $\Lambdapmp \in \L^\pm(\phi)$, and a line $\gamma \in \B$, the following hold:
\begin{enumerate}
\item\label{ItemNonuniformOWOTO}
If $\gamma$ is not carried by $\A_{\na}(\Lambdapmp)$ then it is either weakly attracted to $\Lambdapp$ under iteration by $\phi$ or to $\Lambda^-_\phi$ under iteration by $\psi$.    
\item\label{ItemUniformOWOTO}
For any weak neighborhoods $V^+$ and $V^-$ of generic leaves of $\Lambda^+_\phi$ and $\Lambda^-_\phi$, respectively, there exists an integer $m \ge 1$ (independent of $\gamma$) such that at least one of the following holds: \ $\gamma \in V^-$; \ $\phi^m(\gamma) \in V^+$; \ or $\gamma$ is carried by $\A_{na}(\Lambdapmp)$. 
\end{enumerate}
\end{corollary}

\begin{proof}  For (1) we assume that $\gamma$ is not weakly attracted to $\Lambda^+_\phi$ under iteration of $\phi$ and that $\gamma$ is not weakly attracted to $\Lambda^+_\psi=\Lambda^-_\phi$ under iteration of $\psi=\phi^\inv$, and we prove that $\gamma$ is carried by $\A_{\na}(\Lambdapmp)$. Applying Theorem~\ref{ThmRevisedWAT} to both $\psi$ and $\phi$, we have 
\begin{align*}
\gamma &\in \B_\na(\Lambda^+_\phi,\phi) \intersect \B_\na(\Lambda^+_\psi,\psi) \\
  &= \left( \B_\ext(\Lambdapmp;\psi) \union \B_\sing(\psi) \union \B_\gen(\psi) \right) \intersect \left( \B_\ext(\Lambdapmp;\phi) \union \B_\sing(\phi) \union \B_\gen(\phi) \right) \qquad(*)
\end{align*}
and using this we proceed by cases.

\medskip

\textbf{Case 1:} $\gamma \in \B_\gen(\phi) \union \B_\gen(\psi)$. By symmetry we may assume $\gamma \in \B_\gen(\phi)$ and so $\gamma$ is a generic leaf of some $\Lambda^+_t \in \L(\phi)$. Since $\gamma \in \B_\na(\Lambda^+_\phi,\phi)$ it follows that $\Lambdapp \not\subset \Lambda^+_t$, and so the stratum $H_t$ associated to $\Lambda^+_t$ is contained in $Z$, by Fact~\refGM{FactLeafAsLimit} and the definition of $Z$. Since $\<Z,\hat \rho_r\>$ is $f_\#$-invariant (Lemma~\ref{defining Z}\pref{ItemZPPathsInv}) and the set of lines that it carries is closed (Lemma~\ref{LemmaZPClosed}\pref{item:closed}),  $\<Z,\hat \rho_r\>$ carries $\Lambda^+_t$ by (Fact~\refGM{FactLeafAsLimit}). Lemma~\ref{LemmaZPClosed}\pref{Item:circuits}  then implies that  $\A_{\na}(\Lambdapmp)$ carries $\Lambda^+_t$ and hence $\gamma$.

\medskip

Having settled Case~1 we may assume that
$$\gamma \in \left( \B_\ext(\Lambdapmp;\psi) \union \B_\sing(\psi) \right) \intersect \left( \B_\ext(\Lambdapmp;\phi) \union \B_\sing(\phi)\right)
$$
 
\textbf{Case 2:} $\gamma \in \B_\ext(\Lambdapmp;\psi) \union \B_\ext(\Lambdapmp;\phi)$. By symmetry we may assume $\gamma \in \B_\ext(\Lambdapmp;\psi)$, so there is a lift $\ti \gamma$ with endpoints $P,Q \in  \bdy_\ext(A,\psi)$ for some $A \in \A_{\na}(\Lambdapmp)$. 

\textbf{Case 2a:} At least one of $P$ or $Q$ is not in $\bdy A$, say $P \not\in \bdy A$. It follows that $P \in \Fix_N(\wh\Psi) \setminus \bdy A$ for some $A$-related $\Psi \in P(\psi)$. Applying Corollary~\ref{CorRelatedBdy} it follows that $P$ is an isolated attracting point of $\Fix_N(\wh\Psi)$. Since $P \in \bdy_\ext(A,\psi)$, by Corollary~\ref{CorDisjBdyA} and Lemma~\ref{LemmaTwoLifts} it follows that $\Phi = \Psi^{-1}$ is the only automorphism representing $\phi$ with $P \in \Fix(\wh\Phi) = \Fix(\wh\Psi)$. Since $P$ is a repeller for the action of $\wh\Phi$, we have $P \not \in \Fix_N(\wh\Phi)$, and so $\gamma \not \in \B_\sing(\phi)$ and $\gamma \not \in\B_\ext(\Lambda^\pm_\phi;\phi)$. By $(*)$ it follows that $\gamma \in \B_\gen(\phi)$, reducing to Case~1.

\textbf{Case 2b:} $P,Q \in \bdy A$, and so $\gamma$ is carried by $\A_{\na}(\Lambdapmp)$ and we are done.   

\medskip

Having settled Cases~1 and~2, we are reduced to the following:

\medskip

\textbf{Case 3:} $\gamma \in \B_\sing(\phi) \union \B_\sing(\psi)$. By symmetry we may assume $\gamma \in \B_\sing(\phi)$, and so there exists $\Phi \in P(\phi)$ and a lift $\ti \gamma$ with endpoints $P,Q \in \Fix_N(\wh\Phi)$. 

\textbf{Case 3a: $\Fix(\Phi)$ is nontrivial.} By Corollary~\ref{CorRelatedBdy} there exists $A \in \A_\na(\Lambdapmp)$ such that $\Fix(\Phi) \subgroup A$, and so $\Phi$ is $A$-related and $\gamma \in \B_\ext(A,\phi) \subset \B_\ext(\Lambdapmp;\phi)$, reducing to Case~2.

\textbf{Case 3b: $\Fix(\Phi)$ is trivial.} It follows that $P,Q$ are isolated attractors in $\Fix_N(\wh\Phi)$. Lemma~\ref{LemmaTwoLifts} combined with the assumption of Case 3b implies that $\Psi=\Phi^\inv$ is the only automorphism representing $\psi$ with $P \in \Fix(\wh\Psi)$. As in Case 2a, using $(*)$ we conclude that $\gamma \in \B_\gen(\psi)$, reducing to Case~1. 

\medskip

This completes the proof of (1). 

\medskip
 
We prove (2) by contradiction.  If  (2)  fails then there are neighborhoods $V^+,V^-$ of generic leaves of $\Lambdapp,\Lambda^-_\phi$ respectively, a sequence of lines $\gamma_i \in \B$ and a sequence of positive integers $m_i \to \infinity$, such that for all $i$ we have: $\gamma_i \not\in V^-$; $\phi^{2m_i}(\gamma_i) \not\in V^+$; and $\gamma_i$ is not carried by $\A_\na(\Lambda^\pm_\phi)$. We may assume that $V_+$ has the property $\phi(V_+) \subset V_+$, because generic leaves of $\Lambdapp$ have a neighborhood basis of such sets. Similarly, we may assume that $V_- \subset \phi(V_-)$.  Since $\A_\na(\Lambda^\pm_\phi)$ is $\phi$-invariant, none of the lines $\phi^{m_i}(\gamma_i)$ are carried by $\A_\na(\Lambda^\pm_\phi)$. Since $\A_\na(\Lambda^\pm_\phi)$ is malnormal we may apply Lemma~\ref{ItemNotZPLimit}, with the conclusion that after passing to a subsequence of $\phi^{m_i}(\gamma_i)$, some weak limit $\gamma$ is not carried by $\A_\na(\Lambda^\pm_\phi)$. 
  
To contradict (1) we show that $\gamma$ is weakly attracted to neither $\Lambdapp$ nor $\Lambda^-_\phi$. By symmetry we need show only that the sequence $\phi^m(\gamma)$ does not weakly converge to $\Lambdapp$. If it does then $\phi^M(\gamma) \in V^+$ for some~$M$. Since $V^+$ is open there exists $I$ such that $\phi^{m_i +M}(\gamma_i) \in V^+$ for all $i \ge I$. Since $\phi(V^+) \subset V^+$, it follows that $\phi^m(\gamma_i) \in V^+$ for all $m \ge m_i + M$ and $i \ge I$. We can choose $i \ge I$ so that $m_i \ge M$, and it follows that $\phi^{2m_i}(\gamma_i) \in V^+$, a contradiction. 
\end{proof}

\subsection{Nonattracted lines of \eg\ height: the nongeometric case.}
\label{SectionNonattrFullHeight}

We continue the \emph{Notational Conventions} established at the beginning of Section~\ref{SectionTheoremGStatement}. By combining Proposition~\refGM{PropGeomEquiv} with the following equations of free factor systems 
$$[G_r]=\F_\supp(\Lambda^\pm_\phi)=[G'_u], \qquad [G_{r-1}]=[G'_{u-1}]
$$
we may conclude that the stratum $H_r$ is geometric if and only if the stratum $H'_u$ is geometric. The realizations of a line $\gamma \in \B$ in the marked graphs $G,G'$ will be denoted $\gamma_G,\gamma_{G'}$ respectively, or just as $\gamma,\gamma'$ when we wish to abbreviate the notation, or even both just as $\gamma$ when we wish for further abbreviation.

The heart of Theorem~\ref{ThmRevisedWAT} (Theorem~G) is the special case concerned with those lines $\gamma$ such that $\gamma_G$ has height~$r$, equivalently $\gamma_{G'}$ has height~$u$, and in this section we focus on that case. We give necessary and sufficient conditions for $\gamma$ to be weakly attracted to $\Lambda^+_\phi$ under iteration of $\phi$, expressed in terms of the form of $\gamma_{G'}$. This is the analog of Proposition~6.0.8 of \BookOne\ which has the additional hypothesis that $\gamma$ is birecurrent, and the proof of which is separated into geometric and non-geometric cases. In our present setting we drop the birecurrence hypothesis, and we also separate the proof into the non-geometric case in Lemma~\ref{nonGeometricFullHeightCase} and the geometric case in Lemma~\ref{geometricFullHeightCase}. The conclusions of two lemmas describe $\gamma_{G'}$ in explicit detail which, while more than we need for our applications, is included because it is needed for the proof and it helps clarify the picture. Although in this section we do not yet derive the conclusions of Theorem~\ref{ThmRevisedWAT} (Theorem~G) for the height~$r$ case, that will be done as part of the derivation of those conclusions for the general case, carried out in Section~\ref{SectionNALinesGeneral}.

The nongeometric case, in which the strata $H_r,H'_u$ are not geometric, is entirely handled in the following Lemma~\ref{nonGeometricFullHeightCase}. The geometric case will take considerably more work and is handled in Section~\ref{SectionNFHGeometric}. 

\begin{lemma}
 \label{nonGeometricFullHeightCase} Assuming that the strata $H_r$, $H'_u$ are not geometric, and with the notation above, if $\gamma \in \B$ has realization $\gamma_G$ in $G$ of height $r$ and $\gamma_{G'}$ in $G'$ of height~$u$, and if $\gamma$ is not weakly attracted to $\Lambda^+_\phi$, then its realization $\gamma_{G'}$, satisfies at least one of the following:
\begin{enumerate}
\item \label{ItemNGFHCgeneric}
$\gamma_{G'}$ is a generic leaf of $\Lambda^-_\phi$.
\item \label{ItemIntermediatePath}
$\gamma_{G'}$ decomposes as $\overline R_1 \mu R_2$ where $R_1$ and $R_2$ are principal rays for $\Lambda^-_\phi$ and $\mu$ is either the trivial path or a nontrivial path of one of the forms $\alpha$, $\beta$, $\alpha\beta$, $\alpha\beta\bar\alpha$, such that $\beta$ is a nontrivial path of height $<s$ and $\alpha$ is a height $s$ indivisible Nielsen path. 
\item \label{ItemNGFHCOneRay}
$\gamma_{G'}$ or $\gamma_{G'}{}^\inv$ decomposes as $\overline R_1 \mu R_2$ where $R_1$ is a principal ray for $\Lambda^-_\phi$, $R_2$ is a ray of height $<s$, and $\mu$ is either trivial or a height $s$ Nielsen path. 
\end{enumerate}
\end{lemma}

\begin{proof} By Proposition~\refGM{PropGeomEquiv}, since $H_r$ is not a geometric stratum, neither is $H'_u$. Let $\alpha$ denote the unique (up to reversal) indivisible Nielsen path of height $s$ in $G'_u$, if it exists; by Fact~\refGM{FactGeometricCharacterization} and Fact~\trefGM{FactEGNielsenCrossings}{ItemEGNielsenNotClosed} it follows that $\alpha$ is not closed and that we may orient $\alpha$ so that its initial endpoint $v$ is an interior point of $H'_u$.

We adopt the abbreviated notation $\gamma$ for $\gamma_G$ and $\gamma'$ for $\gamma_{G'}$.

We first show that $\gamma$ has infinitely many edges in $H_r$ by proving that if a line $\gamma$ has height $r$ and only finitely many edges in $H_r$ then $\gamma$ is weakly attracted to $\Lambda^+$. 
To prove this, write $\gamma = \gamma_- \gamma_0 \gamma_+$ where $\gamma_-,\gamma_+ \subset G_{r-1}$ and $\gamma_0$ is a finite path whose first and last edges are in $H_r$. Since $f_\#$ restricts to a bijection on lines of height $r-1$, it follows that $f^k_\#(\gamma)$ has height $r$ for all $k$, and so $f^k_\#(\gamma_0)$ is a nontrivial path of height $r$ for all $k$. By Fact~\refGM{FactEvComplSplit} there exists $K \ge 0$ such that $f^K_\#(\gamma_0)$ completely splits into terms each of which is either an edge or Nielsen path of height $r$ or a path in $G_{r-1}$, with at least one term of height $r$. Let $f^K_\#(\gamma_0) = \gamma'_- \gamma'_0 \gamma'_+$ where $\gamma'_-,\gamma'_+$ are in $G_{r-1}$ and $\gamma'_0$ is the maximal subpath of $f^K_\#(\gamma_0)$ whose first and last edges are in $H_r$, so $\gamma'_0$ is nontrivial and completely split. Since $f^K_\#(\gamma) = [f^K_\#(\gamma_-) \, \gamma'_- \, \gamma'_0 \, \gamma'_+ \, f^K_\#(\gamma_+)]$ then, using RTT-(i), it follows that there is a splitting $f^K_\#(\gamma) = \gamma''_- \cdot \gamma'_0 \cdot \gamma''_+$ where $\gamma''_- = [f^K_\#(\gamma_-) \, \gamma'_-]$ and $\gamma''_+= [\gamma'_+ \, f^K_\#(\gamma_+)]$ are in $G_{r-1}$. By Fact~\trefGM{FactEGNielsenCrossings}{ItemEGNielsenNotClosed}, each term in this splitting of $f^K_\#(\gamma)$ which is an indivisible Nielsen path of height $r$ is adjacent to a term that is an edge in $H_r$. It follows that at least one term in the splitting of $f^K_\#(\gamma)$ is an edge in $H_r$, implying that $\gamma$ is weakly attracted to $\Lambda^+$. 

Since $\gamma$ contains infinitely many edges in $H_r$, and since $[G_r]=[G'_u]$ and the graphs $G_r, G'_u$ are both core subgraphs, the line $\gamma_{G'}$ contains infinitely many edges in $H'_u$. 

In the part of the proof of the nongeometric case of Proposition~6.0.8 of \BookOne\ that does not use birecurrence and so is true in our context, it is shown that there exists $M' > 0$ so that for every finite subpath $\gamma_i'$ of $\gamma'$ there exists a line or circuit $\tau_i'$ in $G'$ that contains at most $M'$ edges of $H'_{s}$ such that $\gamma_i'$ is a subpath of $g_\#^{k_i}(\tau_i')$ for some $k_i\ge 0$. If $G_{r-1} = \emptyset$ then $\tau_i'$ is a circuit; otherwise $\tau_i'$ is a line. (This is proved in two parts. First, in what is called step 2 of that proof, an analogous result is proved in $G$. Then the bounded cancellation lemma is used to transfer this result to $G'$; the case that $G_{r-1} = \emptyset$ is considered after the case that $G_{r-1} \ne \emptyset$.)

Choose a sequence of finite subpaths $\gamma'_i$ of $\gamma'$ that exhaust $\gamma'$ and let $\tau'_i$ and $k_i$ be as above so that $\gamma'_i$ is a subpath of $g^{k_i}_\#(\tau'_i)$ and so that $\tau'_i$ contains at most $M'$ edges of~$H'_u$. Since $\gamma'$ contains infinitely many $H'_u$ edges, we have $k_i \to +\infinity$ as $i \to +\infinity$. 

By Lemma~\refGM{LemmaEGUnifPathSplitting} there exists $d>0$ depending only on the bound $M'$ such that $g^d_\#(\tau'_i)$ has a splitting into terms each of which is either an edge or indivisible Nielsen path of height $s$ or a path in $G'_{u-1}$. By taking $i$ so large that $k_i \ge d$ we may replace each $\tau'_i$ by $g^d_\#(\tau'_i)$ and each $k_i$ by $k_i - d$, and hence we may assume that $\tau_i'$ has a splitting 
$$ \tau_i' = \tau'_{i,1} \cdot \ldots \cdot \tau'_{i,l_i}  \qquad\qquad (*)
$$
each of whose terms is an edge or Nielsen path of height $s$ or a path in $G'_{u-1}$. The number of edges that $\tau_i'$ has in $H_u$ is still uniformly bounded, and so $l_i$ is uniformly bounded. Passing to a subsequence, we may assume that $l_i$ and the ordered sequence of height $s$ terms in $\tau'_i$ are independent of~$i$. We may also assume that $l = l_i$ is minimal among all such choices of $\gamma_i'$ and $\tau_i'$.

\subparagraph{Case A: $l = 1$.}   In this case $\tau'_i = E$ is a single edge of $H'_u$ and so, by \refGM{FactLeafAsLimit}, $\gamma'$ is a leaf of $\Lambda^-_\phi$. If both ends of $\gamma'$ have height $s$ then $\gamma'$ is generic by Fact~\refGM{FactTwoEndsGeneric} and case (1) is satisfied. 

Suppose one end of $\gamma'$, say the positive end, has height $\le s-1$. We have a concatenation $\gamma' = \overline R_1 R_2$ where the ray $R_1$ starts with an edge of $H'_u$ and the ray $R_2$ is contained in $G'_{u-1}$. By Fact~\refGM{FactPrincipalVertices}, the concatenation point is a principal vertex. By Lemma~\trefGM{FactTiles}{ItemUndergraphPieces}, for each $m$ there is an $m$-tile in $\gamma'$ which is an initial segment of $R_1$. By Corollary~\refGM{CorTilesExhaustRay} it follows that $R_1$ is a principal ray. This shows that (3) is satisfied with trivial $\mu$.

\subparagraph{Case B: $l \ge 2$.} Choose a subpath $\nu_i' \subset \tau_i'$, with endpoints not necessarily at vertices, such that $g^{k_i}_\#(\nu_i') = \gamma_i'$. Let $\tau''_i$ be the subpath obtained from $\tau'_i$ by removing the initial segment $\tau'_{i,1}$ and the terminal segment $\tau'_{i,l}$ of the splitting $(*)$, so either $\tau''_i = \tau'_{i,2} \cdot \ldots \cdot \tau'_{i,l-1}$ or, when $l=2$, $\tau''_i$ is the trivial path at the common vertex along which $\tau'_{i,1}$ and $\tau'_{i,2}$ are concatenated. After passing to a subsequence, we may assume that $\tau''_i \subset \nu_i'$; if no such subsequence existed then we could reduce~$l$ by removing either $\tau'_{i,1}$ or $\tau'_{i,l}$ from $\tau'_i$. For the same reason, we may assume that $\gamma'$ has a finite subpath that contains $g^{k_i}_\#(\tau''_i)$ for all $i$. After passing to a subsequence, we may assume that $\mu = g^{k_i}_\#(\tau''_i)$ is independent of $i$. Since the sequence of height $s$ terms in $\tau''_i$ is independent of $i$, it follows that $\mu$ is either trivial or has a splitting into terms each of which is either $\alpha$, or $\bar\alpha$, or a path in $G'_{u-1}$. Since the endpoints of $\alpha$ are distinct, no two adjacent terms in this splitting can both be $\alpha$ or $\bar\alpha$, and so each subdivision point of the splitting is in $G'_{u-1}$. Since $v$ is an interior point of $H'_u$, for any occurence of $\alpha$ or $\bar\alpha$ as a term of $\mu$ the endpoint $v$ must be an endpoint of $\eta$. It follows that $\mu$ can be written in one of the forms given in item~\pref{ItemIntermediatePath}, after possibly inverting~$\gamma'$.

Write $\gamma'$ as $\overline R_1 \mu R_2$. If $\tau'_{1,1}$ is an edge $E$ in $H'_u$ then $E = \tau'_{i,1}$ for all $i$, and the ray $R_1$ is the increasing union of $g^{k_1}_\#(\bar E) \subset g^{k_2}_\#(\bar E) \subset \cdots$, so $R_1$ is a principal ray for $\Lambda^-_\phi$. Otherwise $\tau'_{i,1}$ is a path in $G'_{u-1}$ for all $i$ and $R_1$ is a ray in $G'_{u-1}$. Using $\tau'_{i,l}$ similarly in place of $\tau'_{i,1}$, $R_2$ is either a principal ray for $\Lambda^-_\phi$ or a ray in $G'_{u-1}$. At least one of $R_1$ and $R_2$ is a principal ray. If they are both principal rays then item~(2) holds, otherwise (3)~holds. 
\end{proof}

\subsection{Nonattracted lines of \eg\ height: the geometric case.} 
\label{SectionNFHGeometric}

We continue to analyze the special case of Theorem~\ref{ThmRevisedWAT} (Theorem~G) concerned with lines of maximal height and a top \eg-stratum. We state and prove Lemma~\ref{geometricFullHeightCase}, which covers the case of a geometric top stratum. Our proof applies also to Proposition~\ref{PropWeakGeomRelFullIrr}, which in \PartFour\ will be incorporated into the conclusions of Theorem~J.

The reader may wish to review the notations established in the beginning of Sections~\ref{SectionTheoremGStatement} and~\ref{SectionNonattrFullHeight}. We assume the strata $H_r$, $H'_u$ are geometric---equivalently the lamination pairs $\Lambda^\pm_\phi = \Lambda^\pm_\psi$ are geometric (Proposition~\refGM{PropGeomEquiv} and Definition~\refGM{DefGeometricLamination})---we let $\rho_r$, $\rho'_u$ be the closed indivisible Nielsen paths in $G_r$, $G'_u$ of heights $r,u$, respectively. By applying Proposition~\refGM{PropGeomEquiv}, up to reorienting these Nielsen paths we have $[\rho_r]=[\rho'_u]$.

 
\begin{lemma}[Height $r$ lines in the geometric case]
\label{geometricFullHeightCase} Assuming that $H_r$, $H'_u$ are geometric, and with notation as above, if $\gamma \in \B$ has height~$r$ in $G$ and is not weakly attracted to $\Lambdapp$ then its realization $\gamma_{G'}$ in $G'$ has at least one of the following forms:
\begin{enumerate}
\item \label{ItemGFHCIterate}
$\gamma_{G'}$ or $\bar\gamma_{G'}$ is the bi-infinite iterate of $\rho'_u$.
\item \label{ItemGFHCLeaf}
$\gamma_{G'}$ is a generic leaf of $\Lambda^-_\phi$.
\item \label{ItemGFHCTwoRays}
$\gamma_{G'}$ decomposes as $\overline R_1 \mu R_2$ where $R_1$ and $R_2$ are principal rays for $\Lambda^-_\phi$ and $\mu$ is either a trivial path, a finite iterate of $\rho'_u$ or its inverse, or a nontrivial path of height $<u$.
\item \label{ItemGFHCOneRay}
$\gamma_{G'}$ or $\bar\gamma_{G'}$ decomposes as $\overline R_1 R_2$ where $R_1$ is a principal ray for $\Lambda^-_\phi$ and the ray $R_2$ either has height $<u$ or is the singly infinite iterate or $\rho'_u$ or its inverse.
\end{enumerate}
\end{lemma}

Until further notice in this section, we adopt the notation and hypotheses of Lemma~\ref{geometricFullHeightCase}. For purposes of the proof, by replacing $\phi$ and $\psi=\phi^\inv$ with their restrictions to $\Out(\pi_1 G) = \Out(\pi_1 G')$ (Fact~\refGM{FactMalnormalRestriction}), we may replace $f \from G \to G$ with its restriction to $G_r$, and replace $f' \from G' \to G'$ with its restriction to $G'_u$. Hence we have reduced to the assumption that $H^{\vphantom{\prime}}_r,H'_u$ are the top strata. Under that assumption, a geometric model for $H'_u$ (Definition~\refGM{DefGeomModel}) is the same thing as a weak geometric model (Definitions~\refGM{DefWeakGeomModel}). We fix such a model, and we recall its static data; later we review its dynamic data. We adopt shorthand notation $Q = G'_{u-1}$, which has no contractible components (Fact~\refGM{FactGeometricCharacterization}).

\medskip\noindent
\textbf{Geometric model: Static data.} The static data consists of the following. First we have the finite subgraph $Q \subset G'$. We also have a compact surface $S$ with $m+1$ boundary components $\bdy S = \bdy_0 S \union \cdots \union \bdy_m S$, $m \ge 0$. The \emph{upper boundary} of $S$ is $\bdy_0 S$. The \emph{lower boundary} is $\bdy S - \bdy_0 S = \union_{i=1}^m \bdy_i S$ and is denoted $\bdy_\ell S$. We have a map $\alpha \from \bdy_\ell S \to Q$ such that for each $i=1,\ldots,m$ its restriction $\alpha_i \from \bdy_i S \to Q$ is a homotopically nontrivial closed edge path. We have a quotient 2-complex $Y$ obtained by gluing $S$ and $Q$ using the attaching map $\alpha \from \bdy_\ell S \to Q$. Let $j \from Q \disjunion S \to Y$ denote the quotient map. We also have an embedding $G' \inject Y$ extending the embedding $G'_{u-1} = Q \inject Y$. Finally, we have a deformation retraction $d \from Y \to G'$ which takes $\bdy_0 S$, regarded as a closed curve based at the unique point $p'_u = G' \intersect \bdy_0 S$, to the closed indivisible height~$u$ Nielsen path $\rho'_u$. Note that $Y$ may be regarded as a ``marked 2-complex'' for $F_n$, by composing the marking of $G'$ with the homotopy equivalence $G' \inject Y$, and so up to inner automorphism we have an identification $\pi_1(Y) \approx F_n$. Altogether the static data will be denoted $Q \disjunion S \xrightarrow{j} Y \xrightarrow{d} G'$.

\medskip\noindent
\textbf{Strategy of the proof of Lemma~\ref{geometricFullHeightCase}.} For each abstract line $\gamma$ realized in $G'$ with full height~$u$, we shall define a canonical decomposition of~$\gamma$ as an alternating concatenation of ``overpaths'' and ``underpaths''. Very roughly speaking an overpath of $\gamma$ is a subpath that begins and ends with edges of $H'_u$, that pulls back to the surface~$S$, and that is maximal with respect to these properties. An overpath of $\gamma$ can be a finite subpath, a subray, or the whole line~$\gamma$. Distinct overpaths have disjoint interiors, and the underpaths of $\gamma$ are the maximal subpaths disjoint from the interiors of the overpaths. When an overpath of $\gamma$ is pulled back to $S$ and straightened with respect to a hyperbolic structure on $S$, the result is called a ``geodesic overpath'', and this can be a finite geodesic path or geodesic ray, in either case intersecting $\bdy S$ precisely in its finite endpoints, or a bi-infinite geodesic line. From the hypothesis that $\gamma$ is not weakly attracted to $\Lambda^+_\phi$ under iteration of $\phi$, one shows that none of the geodesic overpaths of $\gamma$ are weakly attracted to unstable geodesic lamination of the pseudo-Anosov homeomorphism of the geometric model. By applying Nielsen--Thurston theory (Proposition~\refGM{PropNTWA}) one shows that these geodesic overpaths all have a certain form, the key feature being that they do not cross the stable geodesic lamination, from which the desired form of $\gamma$ given in Lemma~\ref{geometricFullHeightCase} is deduced.

Formalizing this strategy requires a lot of descriptive work, leading up to the statement of Lemma~\ref{LemmaOverpathsAttraction} which gives the form of the ``geodesic overpaths'' alluded to above. After Lemma~\ref{LemmaOverpathsAttraction} is stated and proved, we will apply it complete the proof of Lemma~\ref{geometricFullHeightCase}. 

In order to describe overpath--underpath decompositions we use the peripheral Bass-Serre tree $F_n \act T$ associated to a geometric model as a bookkeeping device. We turn next to a review of these topics from \PartOne.

\medskip\noindent
\textbf{Vertex spaces and edge spaces.} The peripheral Bass-Serre tree will be described in terms of the vertex space---edge space decomposition of the universal cover of the geometric model $Y$. Our description roughly follows Definition~\refGM{DefPeripheralSplitting}, with similar justifications using the ``graph of spaces'' approach to Bass-Serre theory in \cite{ScottWall}.

In addition to the notation $Q = G'_{u-1}$, we adopt the shorthand notation $B = \bdy_\ell S$.

Consider the following pushout diagram (with the square on the right added for convenience):
$$\xymatrix{
\breve S \disjunion \breve Q \ar@{=}[r] \ar[dr]_{\breve q} & \breve Y \ar[r]^{\breve j}  
& \wt Y \ar[d]_{q} \ar[r]^{\tilde d}_{\supset} & \wt G' \ar[d] \\
& S \disjunion Q \ar[r]_{j} & Y \ar[r]^{d} & G'
}$$
The map $q$ is the universal covering map, $\breve S \disjunion \breve Q = \breve Y$ is the subspace of the Cartesian product $(S \disjunion Q) \cross \wt Y$ consisting of all pairs $(x,\ti y)$ such that $j(x)=q(\ti y)$, and the maps $\breve q$ and $\breve j$ are restrictions of the two projection maps of the Cartesian product. The sets $\breve S$ and $\breve Q$ defining the disjoint union $\breve S \disjunion \breve Q$ are respectively characterized by requiring $\breve q(x) \in S$ and $\breve q(x) \in Q$. The group $F_n$ acts on $\wt Y$ by deck transformations and on $S \disjunion Q$ trivially, inducing a diagonal action on $\breve S \disjunion \breve Q=\breve Y$, such that $\breve j$ is $F_n$-equivariant and $\breve q$ is a covering map with deck transformation group $F_n$. We also define $\breve B \subset \breve S \subset \breve Y$ to be the total lift of $B=\bdy_\ell S$ via the map $\breve q$. The components of $\breve S$, $\breve Q$, and $\breve B$ are indexed as follows, together with their respective images in $\wt Y$ and stabilizer groups under the action $F_n = \pi_1(Y) \act \wt Y$:
\begin{align*}
\breve S = \union_s \breve S_s, & \qquad \wh S_s = \breve j(\breve S_s) \subset \wt Y, \qquad \Gamma_s = \Stab(\breve S_s) = \Stab(\wh S_s) \subgroup F_n \\
\breve Q = \union_q \breve Q_q,  & \qquad \wh Q_q = \breve j(\breve Q_q) \subset \wt Y, \qquad \Gamma_q = \Stab(\breve Q_q) = \Stab(\wh Q_q) \subgroup F_n \\
\breve B = \union_b \breve B_b, & \qquad \wh B_b = \breve j(\breve B_b) \subset \wt Y, \qquad \Gamma_b = \Stab(\breve B_b) = \Stab(\wh B_b) \subgroup F_n
\end{align*}
The sets $\wh S_s$ are called the \emph{$S$-vertex spaces} of $\wt Y$, the sets $\wh Q_q$ are the \emph{$Q$-vertex spaces}, and the sets $\wh B_b$ are the \emph{edge spaces}. 

By restricting $\breve q$ we get universal covering maps $\breve S_s \to S$, and an isomorphism of the deck group $\Gamma_s \approx \pi_1(S)$ (well-defined up to inner automorphism). Similarly we have universal covering maps of each $\breve Q_q$ over some component $Q_q$ of $Q$ with deck group $\Gamma_q \approx \pi_1(Q_q)$; and of each $\breve B_b$ over some component $B_b$ of $B$ with infinite cyclic deck group $\Gamma_b \approx \pi_1(B_b)$. Since $\breve S_s$ is connected and is cocompact under the action of $\Gamma_s$, the same is true of $\wh S_s$; similar statements hold for the actions of $\Gamma_q$ and $\Gamma_b$. 

The domain and range restrictions of $\breve j$ are also denoted with subscripts. The restriction $\breve j_q \from \breve Q_q \to \wh Q_q$ is a homeomorphism. The union $\wh Q = \union_q \wh Q_q$ is a disjoint union, it is the component decomposition of $\wh Q$, and $\wh Q$ is equal to the total lift $\wt G'_{u-1}$ of $Q=G'_{u-1}$ under the universal covering map $q \from \wt Y \to Y$. On the other hand the restrictions $\breve j_s \from \breve S_s \to \wh S_s$ and $\breve j_b \from \breve B_b \to \wh B_b$ may fail to be be injective or even locally injective, and the unions $\union_s \wh S_s$ and $\union_b \wh B_b$ need not be disjoint unions. This failure stems from the failure of local injectivity of the attaching map $\alpha \from B \to Q$. Nontheless map $\alpha$ factors on each component of $B$ as a finite sequence of Stallings folds followed by a local injection, which lifts to equivariant factorizations of each $\breve j_s$ and each $\breve j_b$, each term of which is a homotopy equivalence, and so each $\breve j_s$ and $\breve j_b$ is a homotopy equivalence. In particular each $\wh B_b$ and each $\wh S_s$ is contractible. 

For each $s$ the full lift of the lower boundary $\bdy_\ell S = B$ to $\breve S_s$ is denoted $\bdy_\ell \breve S_s$, and we have a component decomposition $\bdy_\ell \breve S_s = \union^s \breve B_b$ where the symbol $\union^s$ means that the union is taken over all $b$ such that $\breve B_b \subset \bdy_\ell \breve S_s$. We denote $\bdy_\ell \wh S_s = \breve j_s(\bdy_\ell \breve S_s) \subset \wh S_s$. We also denote $\bdy_0 \breve S_s = \bdy \breve S_s - \bdy_\ell \breve S_s$ which is the total lift to $\breve S_s$ of the upper boundary $\bdy_0 S$, and we denote $\bdy_0 \wh S_s = \breve j_s(\bdy_0 \breve S_s)$. Using the fact that the map $j \from S \to Y$ embeds $S - \bdy_\ell S$ as an open connected subset of $Y$, similarly $\breve j_s$ embeds $\breve S_s  - \bdy_\ell \breve S_s$ as an open connected subset of $\wt Y$, we have $\breve j_s(\breve S_s  - \bdy_\ell \breve S_s) \subset \wh S_s$, and we have $\breve j_s(\breve S_s  - \bdy_\ell \breve S_s)  = \wh S_s - \bdy_\ell \wh S_s$. Furthermore, for any point pair $x \ne y \in \breve S_s$ such that $\breve j_s(x)=\breve j_s(y)$ there exists a component $\breve B_b$ of $\bdy_\ell S_s$ such that $x,y \in \breve B_b$. We therefore have a component decomposition $\bdy_\ell \wh S_s = \union^s \wh B_b = \union^s \breve j_s(\breve B_b)$. 

The embedding $G' \subset Y$ and deformation retraction $d \from Y \to G'$ lift to an $F_n$-equivariant embedding $\wt G' \subset \wt Y$ and deformation retraction $\ti d \from \wt Y \to \wt G'$, commuting with universal covering maps $q \from \wt Y \to Y$ and $\wt G' \to G'$, as shown in the right square of the above diagram. Denoting
$$\wt G'_s = \wt G' \intersect \wh S_s = \ti d(\wh S_s)
$$
we may restrict $\ti d$ to obtain a $\Gamma_s$-equivariant deformation retraction
$$\ti d_s \from \wh S_s \to \wt G'_s
$$
Note that $\wt G'_s$ is connected since $\wh S_s$ is connected, and $\Gamma_s$ acts cocompactly on the tree $\wt G'_s$ since it acts cocompactly on $\wh S_s$. It follows that we may naturally identify 
$$\bdy\Gamma_s = \bdy\wh S_s = \bdy\wt G'_s \subset \bdy F_n
$$
from which it follows in turn that each line in $\wt G'$ with ideal endpoints in $\bdy \Gamma_s$ is contained in the subgraph $\wt G'_s$. Denote 
$$\wt H'_u = \wt G' \setminus \wh Q = \bigl(\text{the full lift of $H'_u = G' \setminus Q$}\bigr)
$$
Let~$\wt H'_{u,s} \subset \wt H'_u$ be the subgraph of all edges of $\wt H'_u \intersect \wt G'_s$ (the latter intersection may contain some isolated vertices which we avoid by defining $\wt H'_{u,s}$ in this manner). Note that the components of the subgraph $\wt G'_s \setminus \wt H'_{u,s}$ are precisely the components $\wh B_b$ of $\bdy_\ell \wh S_s$, one for each component $\breve B_b$ of~$\bdy_\ell \breve S_s$. 

\medskip\noindent
\textbf{The Bass-Serre tree $F_n \act T$.} The Bass-Serre tree $T$ is a bipartite tree with vertices and edges as follows. First, $T$ has one $S$-vertex denoted $V_s$ for each $S$-vertex space $\wh S_s$. Also, $T$ has one $Q$-vertex denoted $V_q$ for each $Q$-vertex space $\wh Q_q$. Finally, $T$ has one edge denoted $E_b$ for each edge space $\wh B_b$, and the endpoints of $E_b$ are the unique $S$-vertex $V_s$ and the unique $Q$-vertex $V_q$ having the properties $\wh B_b \subset \wh S_s$ and $\wh B_b \subset \wh Q_q$. The action $F_n \act \wt Y$ induces the action $F_n \act T$. 

Note that $T$ can be characterized algebraically. The conjugacy class $[\pi_1 S]$ equals the set $\{\Gamma_s\}$ of $S$-vertex stabilizers, and the latter corresponds bijectively to $\{V_s\}$ since $\pi_1 S$ is its own normalizer in $F_n$ (Lemma~\trefGM{LemmaLImmersed}{ItemSeparationOfSAndL}). Also, the union of the conjugacy classes constituting the subgroup system $[\pi_1 Q]$ equals the set $\{\Gamma_q\}$ which corresponds bijectively to $\{V_q\}$ since $[\pi_1 Q]$ is a  malnormal subgroup system (Lemma~\trefGM{LemmaLImmersed}{ItemComplementMalnormal}). The tree $T$ thus has one $S$-vertex $V_s$ for each $\Gamma_s$, one $Q$-vertex $V_q$ for each $\Gamma_q$, with an edge $E_b$ connecting $V_s$ to $V_q$ if and only if $\Gamma_s \intersect \Gamma_q$ is a nontrivial subgroup of~$F_n$, that subgroup being the infinite cyclic subgroup $\Gamma_b$. 

\textbf{Remark.} Our notation here may be compared with the notation of Definition~\refGM{DefPeripheralSplitting} by setting $L = Q \union \bdy_0 S$. In effect, in forming $T$ we have stripped away the valence~1 vertices and incident edges of the Bass-Serre tree of Definition~\refGM{DefPeripheralSplitting} that are associated to the components of the top boundaries $\bdy_0 \breve S_s$ (see Remark~\refGM{RemarkNotMinimalReFree}). Other valence~$1$ vertices may remain in $T$, namely those associated to ``free lower boundary circles'' of $S$ (Section~\refGM{SectionFreeBoundaryInvariant}).

\medskip\noindent
\textbf{Lines realized in $T$, and over--under decompositions.} Given a line $\ti\gamma \subset \wt G'$, we define its realization $\ti\gamma_T \subset T$, and in parallel we define the over--under decomposition of $\ti\gamma$ in~$\wt G'$. 

In the degenerate case that $\ti\gamma \subset \wh Q$, we have $\ti\gamma \subset \wt Q_q$ for some $q$, in which case the line $\ti\gamma_T$ degenerates to the $Q$-vertex $V_q$, and the entire path $\ti\gamma$ consists of a single underpath $\ti\gamma_q = \ti\gamma$.

Henceforth we may assume $\ti\gamma \not\subset \wh Q$, equivalently $\ti\gamma$ contains an edge of $\wt H'_{u,s}$ for some $s$. 

Next we define the $S$-vertices in $\ti\gamma_T$ and their associated overpaths in $\ti\gamma$. We put the $S$-vertex $V_s \in T$ in the line $\ti\gamma_T$ if and only if $\ti\gamma \intersect \wt H'_{u,s}$ contains an edge, equivalently $\ti\gamma \intersect \interior(\wh S_s) \ne \emptyset$. If $V_s \in \ti\gamma_T$ then the associated overpath denoted $\ti\gamma_s \subset \ti\gamma$ is defined to be the longest subpath of $\ti\gamma$ having the property that each ideal endpoint of $\ti\gamma_s$ is in $\bdy\Gamma_s$ and the edge of $\ti\gamma_s$ incident to each finite endpoint of $\ti\gamma_s$ is in $\wt H'_{u,s}$. Note that $\ti\gamma_s \subset \wt G'_s$. Note also that distinct overpaths have disjoint interiors, because for any overpaths  $\ti\gamma_s, \ti\gamma_{s'} \subset \ti\gamma$ with $s \ne s'$ their intersection $\ti\gamma_s \intersect \ti\gamma_{s'}$ is clearly a path in $\wt G'_s \intersect \wt G'_{s'} \subset \wh Q$, \emph{and} this path is either empty or a common endpoint. If this were not true then: if $\ti\gamma_s = \ti\gamma_{s'}$ then this path has an edge in $\wt H'_u$, contradicting that $\wh Q$ contains no edges of $\wt H'_u$; whereas if $\ti\gamma_s \ne \ti\gamma_{s'}$ then the intersection $\ti\gamma_s \intersect \ti\gamma_{s'}$ has a finite endpoint $x$ with incident edge $E$ such that $x$ is also a finite endpoint of one of $\ti\gamma_s$ or $\ti\gamma_{s'}$ with incident edge $E$, and hence $E \subset \wt H'_u$, leading to the same contradiction. 

Next we define the underpaths of $\ti\gamma$ and their associated $Q$-vertices in $\ti\gamma_T$. The underpaths are the components of $\ti\gamma - \union_s \interior(\ti\gamma_s)$, a disjoint union of possibly degenerate subintervals of $\ti\gamma$, each contained in $\wh Q$. We put the $Q$-vertex $V_q \in T$ in the line $\ti\gamma_T$ if and only if one of the underpaths, denoted $\ti\gamma_q$, is contained in the $Q$-vertex space $\wh Q_q$. 

Finally we define the edges of $\ti\gamma_T$. We put the edge $E_b \subset T$ in $\ti\gamma_T$ if and only if its endpoints $V_s,V_q$ are in $\ti\gamma_T$ and the intersection $\ti\gamma_s \intersect \ti\gamma_q$ is nonempty, in which case that intersection is a point that we denote $p_b = \ti\gamma_s \intersect \ti \gamma_q \in \wh B_b$. 

This completes the definition of $\ti\gamma_T$, although we must still check that it is indeed a path in the tree~$T$, i.e.\ a locally injective edge path. Choosing an orientation of $\ti\gamma$, by construction we have decomposed $\ti\gamma$ into an alternating concatenation of overpaths and underpaths, what we call the \emph{over--under decomposition} of $\ti\gamma$. Associated to this decomposition we have an expression of $\ti\gamma_T$ as a concatenation of edges of $T$. We must check that this concatenation has no backtracking. Supposing that in $\ti\gamma_T$ the $S$-vertex $V_s$ is preceded by an edge $E_b$ and followed by an edge $E_{b'}$, it follows that the overpath $\ti\gamma_s$ is a finite path with endpoints $p_b \in \breve B_b$ and $p_{b'} \in \breve B_{b'}$; the desired inequality $E_b \ne E_{b'}$ follows from the inequality $\breve B_b \ne \breve B_{b'}$ which is true because, otherwise, it would follow that $\ti\gamma_s \subset \breve B_b = \breve B_{b'}$ contradicting that $\ti\gamma_s$ contains in edge of~$\wt H'_u$. And supposing that in $\ti\gamma_T$ the $Q$-vertex $V_q$ is preceded by an edge $E_b$ with opposite $S$-vertex $V_s$ and followed by an edge $E_{b'}$ with opposite $S$-vertex $V_{s'}$, by construction the over--under decomposition has three successive terms $\ti\gamma_s \ti\gamma_q \ti\gamma_{s'}$, and so by construction $\ti\gamma_s$ and $\ti\gamma_{s'}$ have disjoint interiors; but $\ti\gamma_s$ contains every $\wt H'_s$ edge in $\ti\gamma$ including at least one such edge, and $\ti\gamma_{s'}$ contains every $\wt H'_{s'}$ edge in $\ti\gamma$ including at least one such edge, and it follows that $V_s \ne V_{s'}$ and so $E_b \ne E_{b'}$. 

\medskip\noindent
\textbf{Endpoint behavior of overpaths.} For each full height line $\ti\gamma$ in $\wt G'$, we analyze the endpoint structure of each overpath $\ti\gamma_s \subset \ti\gamma$. First, $\ti\gamma_s$ is either a finite nondegenerate path with two finite endpoints, a ray with one finite endpoint and one ideal endpoint, or a line with two ideal endpoints. All finite endpoints of $\ti\gamma_s$ are in $\bdy_\ell \wh S_s$, and all ideal endpoints are in $\bdy\Gamma_s$. These endpoints satisfy the following: 

\medskip\noindent
\textbf{Properness of endpoints:} \quad
\begin{itemize}
\item[]\textbf{Finite--finite:} If $\ti\gamma_s$ is finite then its two endpoints are in distinct components of~$\bdy_\ell \wh S_s$.
\item[]\textbf{Finite--infinite:} If $\ti\gamma_s$ is a ray and if $\wh B_b$ is the component of $\bdy_\ell \wh S_s$ containing its finite endpoint then its ideal endpoint is not in $\bdy \Gamma_b$. 
\item[]\textbf{Infinite--infinite:} If $\ti\gamma_s$ is a line then for any component $\wh B_b$ of $\bdy_\ell \wh S_s$, the two ideal endpoints of $\ti\gamma_s$ are not both in $\bdy\Gamma_b$. 
\end{itemize}
To see why these hold, a finite path in $\wt G'$ having both endpoints in some $\wh B_b$ is entirely contained in $\wh B_b \subset \wh Q$. Similarly, any ray having finite endpoint in $\wh B_b$ and ideal endpoint in $\bdy \Gamma_b = \bdy \wh B_b$ is contained in $\wh B_b$, as is any line having both infinite endpoints in $\bdy \Gamma_b$. But no overpath is entirely contained in $\wh Q$.

\medskip\noindent
\textbf{Geodesic overpaths.} Henceforth in the proof of Lemma~\ref{geometricFullHeightCase} we fix a hyperbolic structure on $S$ with totally geodesic boundary. This lifts to a complete hyperbolic metric on each $\breve S_s$ with totally geodesic boundary. Denote
$$dj \from \breve Q \union \breve S \xrightarrow{\breve j} \wt Y \xrightarrow{\ti d} \wt G'
$$
By restricting $dj$ we obtain the following composition of quasi-isometries with uniform constants independent of $s$, and the associated composition of continuous extensions to Gromov compactifications:
$$dj_s \from 
\begin{cases}
\hphantom{\union \bdy \Gamma_s} \breve S_s \,\,\,\, \xrightarrow{\breve j_s} &\wh S_s \,\, \xrightarrow{\ti d_s} \,\, \wt G'_s \\
\breve S_s \union \bdy \Gamma_s \, \xrightarrow{\breve j_s} &\wh S_s \union \bdy \Gamma_s \,\, \xrightarrow{\ti d_s} \,\, \wt G'_s \union \bdy\Gamma_s
\end{cases}
$$
This allows us to identify $\bdy\Gamma_s$ with the space of asymptotic equivalence classes of geodesic rays in $\breve S_s$ and with the space of asymptotic equivalence classes of geodesic rays in $\wt G'_s$. 

Recall from Definition~\refGM{DefProperGeodesic} the concept of a proper geodesic in $\breve S_s$, namely a geodesic which is not contained in $\bdy\breve S_s$ and whose two endpoints (finite and/or ideal) are in $\bdy\breve S_s \union \bdy_\infinity \wh S_s$. Recall also that two proper geodesics $\ti\gamma_1,\ti\gamma_2 \subset \breve S_s$ are \emph{properly equivalent} if they have the same ideal endpoints and if the set of components of $\bdy\breve S_s$ containing a finite endpoint of $\ti\gamma_i$ is independent of $i=1,2$.

Consider a line $\ti\gamma$ in $\wt G'$ with realization $\ti\gamma_T$ in~$T$. For each $S$-vertex $V_s \in \ti\gamma_T$ with corresponding overpath $\ti\gamma_s \subset \ti\gamma$ we associate a \emph{geodesic overpath} $\breve\gamma_s \subset \breve S_s$, by choosing $\breve \gamma_s$ to be a geodesic whose endpoints in $\bdy \breve S_s \union \bdy\Gamma_s$ map to the endpoints of $\ti\gamma_s$ under the map $dj_s$. It follows that $\ti\gamma_s$ is obtained from $dj_s(\breve \gamma_s)$ by straightening. 
\begin{description}
\item[Properness of geodesic overpaths:] For each line $\ti\gamma \subset \wt G'$ with projection $\ti\gamma$ in $G$, exactly one of the following holds:
\begin{description}
\item[The line $\ti\gamma$ is a top boundary line:] There exists $s$ such that $\ti\gamma=\ti\gamma_s$ and $\breve \gamma_s$ is a component of $\bdy_0 \breve S_s$; equivalently, $\gamma$ is the line that winds bi-infinitely around $\rho$.
\item[Each geodesic overpath is proper:] \quad For each overpath $\ti\gamma_s$ of $\ti\gamma$, its associated geodesic overpath $\breve \gamma_s \subset \breve S_s$ is a proper geodesic. Furthermore, the finite endpoints of $\breve \gamma_s$ are in $\bdy_\ell \breve S_s$, and the choice of $\breve\gamma_s$ is unique up to proper equivalence.
\end{description}
\end{description}
To see why this holds, the statement on finite endpoints holds by construction of geodesic overpaths, and the statement on uniqueness holds because for each $s$ the map $\breve j_s$ induces a bijection between the components of $\bdy_\ell\breve S_s$ and the components of $\breve j_s(\bdy_\ell\breve S_s)$. For the rest, we need only rule out the possibility that $\breve \gamma_s$ is a component of $\bdy_\ell \breve S_s$, but then $\ti\gamma_s \subset \wh Q$, contradicting that $\ti\gamma_s$ contains an edge of $\wt H'_{u,s}$. 

\medskip\noindent
\textbf{Geometric model: Dynamic data.} Given the static data $Q \union S \xrightarrow{j} Y \xrightarrow{d} G'$ of a geometric model for $f'$ as specified earlier, the dynamic data of the geometric model consists of a pseudo-Anosov homeomorphism $\Theta \from S \to S$ such that the following dynamic relation holds:
\begin{description}
\item[$\Theta$ semiconjugates to $f'$:]
The maps $dj \circ \Theta$, $f' \circ dj \from S \to G'$ are homotopic.
\end{description}

Lemma~\ref{LemmaOverpathsAttraction} below will serve a second purpose, applying to the proof of Theorem~J via Proposition~\ref{PropWeakGeomRelFullIrr}. For this purpose we consider also dynamic data that is ``parasitically'' built upon the given static data, and is dynamically related not to the topological representative $f'$ of $\psi$ but instead to a topological representative of some other outer automorphism. Consider $\omega \in \Out(F_n)$ which preserves the free factor system~$\F$. Choose any topological representative $f_\omega \from (G',Q) \to (G',Q)$. Consider $\Theta \from S \to S$ a pseudo-Anosov homeomorphism satisfying the following:
\begin{description}
\item[$\Theta$ semiconjugates to $f_\omega$:] The maps $dj \circ \Theta$, $f_\omega \circ dj \from S \to G'$ are homotopic.
\end{description}
Note that this property depends only on $\omega$ and $\Theta$, and is independent of the choice of $f_\omega$. Also, this property implies that $\omega$ preserves the conjugacy class of the subgroup $j_*(\pi_1 S)$. 

Having chosen $\Theta$ which semiconjugates to $f_\omega$, let $\Lambda^\un, \Lambda^\st \subset S$ denote the unstable/stable geodesic lamination pair for $\Theta$ with respect to some fixed hyperbolic structure on $S$ with totally geodesic boundary. Proposition~\refGM{PropGeomLams} applies, with the conclusion that there exists a dual lamination pair $\Lambda^\pm_\omega \in \L^\pm(\omega)$ such that $\Lambda^\un,\Lambda^\st$ are taken to $\Lambda^+_\omega,\Lambda^-_\omega$, respectively, by the map $dj_\# \from \B(\pi_1 S) \to \B(G') \approx \B(F_n)$ (Proposition~\refGM{PropGeomLams} is currently written only for the case $f_\omega=f'$, but the proof clearly extends to the current situation.)

For each~$s$, the total lifts to $\breve S_s$ of $\Lambda^\un,\Lambda^\st$ will be denoted $\breve\Lambda^\un_s, \breve\Lambda^\st_s \subset \breve S_s$.

\begin{lemma}\label{LemmaOverpathsAttraction}
Consider the various objects specified above: static data $Q \union S \xrightarrow{j} Y \xrightarrow{d} G'$ of a geometric model for $f'$; an outer automorphism $\omega \in \Out(F_n)$ such that $\omega(\F)=\F$; a topological representative $f_\omega \from (G',Q) \to (G',Q)$; and a pseudo-Anosov homeomorphism $\Theta \from S \to S$ that semiconjugates to $f_\omega$, with unstable/stable lamination pair $\Lambda^\un,\Lambda^\st$ and corresponding dual lamination pair $\Lambda^+_\omega,\Lambda^-_\omega$. For each line $\ti\gamma \in \wt \B$ with image $\gamma \in \B$, if $\gamma$ is not weakly attracted to $\Lambda^+_\omega$ under iteration of $\omega$ then one of the following holds (up to reversal of orientation of~$\ti\gamma$):
\begin{description}
\item[(i) Degenerate $Q$-point:] There exists $q$ such that $\ti\gamma_T = V_q$, the underpath $\ti\gamma_q=\ti\gamma_s$ is a line in $\wh Q_q$, and $\gamma$ is in $Q$ and so is carried by~$\F$.
\item[(ii) Degenerate $S$-point:] There exists $s$ such that $\ti\gamma_T = V_s$, the overpath $\ti\gamma_s = \ti\gamma$ is a bi-infinite line in $\wt G'_s$, and the corresponding geodesic overpath $\breve\gamma_s \subset \breve S_s$ is either a leaf of $\breve\Lambda^\st_s$, or is contained in the interior of some principal region of $\breve \Lambda^\st_s$, or is equal to a component of $\bdy_0 \breve S_s$.
\item[(iii) One edge:] \quad
$\ti\gamma_T$ has the form $\xymatrix{V_s \ar@{-}[r]^{E_b} & V_q}$, and $\ti\gamma = \ti\gamma_s \ti\gamma_q$, and the corresponding proper geodesic ray $\breve\gamma_s \subset \breve S_s$ is contained in some crown principal region of~$\breve\Lambda^\st_s$. 
\item[(iv) Two edges, two $S$-endpoints:] \quad
$\ti\gamma_T$ has the form $\xymatrix{V_{s} \ar@{-}[r]^{E_{b}} & V_q \ar@{-}[r]^{E_{b'}} & V_{s'}}$, and $\ti\gamma = \ti\gamma_s \ti\gamma_q \ti\gamma_{s'}$, and the corresponding proper geodesic rays $\breve\gamma_s \subset \breve S_s$, $\breve\gamma_{s'} \subset \breve S_{s'}$ are contained in crown principal regions of $\breve\Lambda^\st_s$, $\breve\Lambda^\st_{s'}$ respectively. 
\end{description}
\end{lemma}

\begin{proof} Observe that for any overpath $\hat\gamma_s \subset \hat\gamma$ with corresponding geodesic overpath $\breve\gamma_s$, the geodesic $\breve\gamma_s$ does not cross $\breve\Lambda^\st_s$ transversely if and only if it has one of the forms that occur in one of conclusions (ii, iii, iv) of the lemma, namely: a leaf of $\breve\Lambda^\st_s$ as in (ii); a proper geodesic line as in (ii) or a proper geodesic ray as in (iii) or (iv), contained in the interior of a principal region of $\breve \Lambda^\st_s$; or a component of $\bdy_0 \breve S_s$ as in (ii). Although there is one other possibility that may occur for a general geodesic line in $\breve S_s$ that does not cross~$\breve\Lambda^\st_s$ transversely, namely a component of $\bdy_\ell \breve S_s$, such a line cannot be a geodesic overpath of anything because its straightened image under $dj_s$ is a line in $\wh Q \subset \wt G'$, and such a line is one big underpath with no overpaths. Note also that \emph{no} finite proper geodesic path in $\breve S_s$ is contained in a principal region of $\breve\Lambda^\st_s$ and \emph{every} finite proper geodesic path crosses $\breve\Lambda^\st_s$ transversely. 

To prove the lemma we argue by contradiction: assuming that none of conclusions (i,~ii,~iii,~iv) holds, we prove that $\gamma$ is weakly attracted to $\Lambda^+_\omega$ under iteration of~$\omega$. It follows from the assumption that $\gamma$ is not carried by~$\F$, for if it were carried then the first conclusion (i)~``Degenerate $Q$-point'' would hold. The over--under decomposition of $\ti\gamma$ therefore has at least one overpath $\ti\gamma_s$. Since none of conclusions (ii,~iii,~iv) hold, by the observation in the previous paragraph it follows that there exists some overpath $\ti\gamma_s$ of $\ti\gamma$ whose corresponding geodesic overpath $\breve\gamma_s$ crosses $\breve\Lambda^\st_s$ transversely. In particular $\breve\gamma_s$ is a proper geodesic in $\breve S_s$ (the only possibility for a nonproper geodesic overpath, that $\breve\gamma_s$ is a component of $\bdy_0 \wh S$, does not cross $\breve\Lambda^\st_s$ transversely).

From the assumption that $\Theta \from S \to S$ semiconjugates (via the map $dj$) to $f_\omega \from G' \to G'$ which topologically represents $\omega$, it follows that $\omega$ preserves the conjugacy class of the subgroup $(d \composed j)_*(\pi_1 S)$ in~$F_n$, which equals the conjugacy class of $\Gamma_s$. We may therefore choose $\Omega \in \Aut(F_n)$ representing $\omega$ such that $\Omega(\Gamma_s)=\Gamma_s$. Let $\ti f_\omega \from \wt G' \to \wt G'$ be the unique lift of $f_\omega$ that satisfies twisted equivariance with respect to the automorphism~$\Omega$, meaning that $\ti f_\omega(\gamma \cdot x) = \Omega(\gamma) \ti f_\omega(x)$ for all $\gamma \in F_n$, $x \in \wt G'$; equivalently, $\ti f_\omega$ and $\Omega$ induce the same homeomorphism of $\bdy \wt G' = \bdy F_n$, and in particular $\ti f_\omega$ preserves~$\bdy\Gamma_s$. It also follows that there is a unique lift $\breve\Theta \from \breve S_s \to \breve S_s$ of $\Theta$ whose action on $\bdy\Gamma_s$ equals the action of~$\ti f_\omega$, and therefore $\breve\Theta$ satisfies twisted equivariance with respect to the restricted automorphism $\Omega \restrict \Gamma_s$, that is, $\breve\Theta(\gamma \cdot x) = \Omega(\gamma) \breve\Theta(x)$ for all $x \in \breve S_s$, $\gamma \in \Gamma_s$. There is a homotopy between $dj \composed \Theta$ and $f_\omega \composed dj$, since $\Theta$ semiconjugates to $f_\omega$, and it  lifts to a $\Gamma_s$-equivariant homotopy between the maps $dj_s \composed \breve\Theta,\ti f_\omega \composed dj_s \from \breve S_s \to \wt G'_s$.

Consider the proper geodesic path $\gamma_s$ in $S$ that is obtained by projecting $\breve\gamma_s$, and so this path crosses $\Lambda^\st$ transversely. Applying Proposition~\refGM{PropNTWA} --- a version of Nielsen-Thurston Theory --- it follows that $\gamma_s$ is geodesically weakly attracted to $\Lambda^\un$ by iteration of $\Theta$. Unwinding the meaning of this statement in our current setting yields the following. A sequence $\delta^i_s$ ($i \ge 0$) of proper geodesics in $\breve S_s$ is said to be a \emph{proper geodesic iteration} of $\breve \gamma_s$ if this sequence is properly equivalent (respectively) to the proper geodesics obtained by straightening the sequence $\breve\Theta^i(\breve\gamma_s)$. In other words, $\delta^i_s$ has the same ideal endpoints in $\breve\Gamma_s$ as $\breve\Theta^i(\breve\gamma_s)$, and the finite endpoints of $\delta^i_s$ are in the same components of $\bdy_\ell \breve S_s$ as the finite endpoints of $\breve\Theta^i(\breve\gamma_s)$. 

\begin{description}
\item[Nielsen-Thurston Theory conclusion:] For any $\epsilon>0$ and $M>0$ there exists $K$ such that for any proper geodesic iteration $\delta^i_s$ of $\breve\gamma_s$, and for any $i \ge K$, there exists a subpath of $\delta{}^i_s$ and a subpath of a leaf of $\wt\Lambda^\un_s$, each of length $\ge M$, and having Hausdorff distance~$\le \epsilon$ from each other.
\end{description}
We also need:
\begin{description}
\item[Iteration Claim:] There exists a proper geodesic iteration $\delta^i_s$ of $\breve\gamma_s$ such that $\delta^i_s$ is a proper geodesic overpath of $(\ti f^i_\omega)_\#(\ti\gamma)$ for each $i \ge 0$.
\end{description}
Before proving this claim, we use it to finish the proof of Lemma~\ref{LemmaOverpathsAttraction}. The map $dj_s \from \breve S_s \to \wt G'_s$ is a quasi-isometry. It follows that if $\alpha,\beta$ are geodesic paths in $\breve S_s$, and if $\alpha$, $\beta$ have long subpaths that are Hausdorff close to each other, then their straightened images $(dj_s)_\#(\alpha)$, $(dj_s)_\#(\beta)$ in $\wt G'$ have long subpaths that are Hausdorff close to each other; furthermore, since $\wt G'$ is a tree, those straightened images have long subpaths that coincide. To be precise, for each $L>0$ there exists $\epsilon > 0$ and $M > 0$ such that if $\alpha,\beta$ have subpaths of length $\ge M$ having Hausdorff distance $\le \epsilon$ from each other, then $(dj_s)_\#(\alpha)$, $(dj_s)_\#(\beta)$ have coinciding subpaths of length $\ge L$. By combining the Iteration Claim and the Nielsen--Thurston Theory conclusion, for any $L$ we may choose $i$ sufficiently large so that $\alpha = \delta^i_s$ is a proper geodesic overpath corresponding to an overpath $(dj_s)_\#(\alpha)$ of the geodesic $(\ti f^i_\omega)_\#(\ti\gamma)$, and $\beta$ is a leaf of $\wt\Lambda^\un_s$ whose image $(dj_s)_\#(\beta)$ is a leaf of $\wt\Lambda^+$, and the overpath $(dj_s)_\#(\alpha)$ and leaf $(dj_s)_\#(\alpha)$ have coinciding subpaths of length $\ge L$. This proves that the sequence $(f^i_\omega)_\#(\gamma)$ converges weakly to a generic leaf of $\Lambda^+_\omega$; that is, $\gamma$ is weakly attracted to $\Lambda^+_\omega$ under iteration of $\omega$.

\medskip

We turn to the proof of the Iteration Claim. There is a natural subgroup $\Aut(F_n;T) \subgroup \Aut(F_n)$ that acts on the tree $T$, namely those automorphisms of $F_n$ that permute the subgroups $\Gamma_s$ and the subgroups $\Gamma_q$; this follows from the algebraic description of $T$ given earlier. 
\begin{description}
\item[Naturality Claim:] For each line $\ti\gamma \in \wt \B$ and each $\A \in \Aut(F_n;T)$, we have $\A(\ti\gamma)_T = \A(\ti\gamma_T)$.
\end{description}
Before proving this claim, we apply it to finish the proof of the Iteration Claim. By hypothesis we have $\omega(\F)=\F$ and therefore $\Omega$ permutes the subgroups $\Gamma_q$. We also have that $\Theta$ semiconjugates to $f_\omega$, which implies that $\omega$ fixes the conjugacy class of $j_*(\pi_1 S)$, which implies that $\Omega$ permutes the subgroups $\Gamma_s$. This shows that $\Omega^i \in \Aut(F_n;T)$ for all integers~$i$, and so the Naturality Claim applies to each $\Omega^i$. Let $\delta^0_s = \breve \gamma_s = \breve\Theta^0(\breve\gamma_s)$, and so $V_s \in \ti\gamma_T$. Since $\Omega(\Gamma_s)=\Gamma_s$ (under the action of $\Aut(F_n)$ on subgroups), and since $T$ is determined by its vertex and edge stabilizers, it follows that $\Omega(V_s)=V_s$ (under the action of $\Aut(F_n;T)$ on~$T$). Applying the Naturality Claim by induction it follows (for all $i \ge 0$) that $V_s \in \Omega^i(\ti\gamma)_T$, which implies that $(\ti f^i_\omega)_\#(\ti\gamma)$ has an overpath associated to $V_s$; let $\delta^i_s$ denote the corresponding geodesic overpath in $\breve S_s$. We must still show that $\delta^i_s$ is a proper geodesic iteration of $\breve\gamma_s = \delta^0_s$. Suppose that $\breve\gamma_s$ has a finite endpoint on the component $\breve B_b$ of $\bdy_\ell\breve S_s$, and so $E_b \subset \ti\gamma_T$. It follows that $\Omega^i(E_b) \subset \Omega^i(\ti\gamma_T) = \Omega^i(\ti\gamma)_T$, from which it follows in turn that $\delta^i_s$ has a finite endpoint on $\breve\Theta^i(\breve B_b)$, which also contains a finite endpoint of $\breve\Theta^i(\breve\gamma_s)$. Thus $\delta^i_s$ and $\breve\Theta^i(\breve \gamma_s)$ have finite endpoints in the same components of $\bdy_\ell \breve S_s$. Suppose that $\breve\gamma_s$ has an infinite endpoint $x \in \bdy\Gamma_s$, and let $y$ be its opposite endpoint, and so one of two cases holds: 
\begin{description}
\item[Case (a):] For some component $\breve B_b$ of $\bdy_\ell\breve S_s$ we have $y \in \breve B_b$.
\item[Case (b):] $y \in \bdy\Gamma_s$.
\end{description}
In Case~(a), we already know that $\delta^i_s$ has a finite endpoint $y_i \in \breve\Theta^i(B_b)$, and so the overpath of $(\ti f^i_\omega)_\#(\ti\gamma)$ associated to $V_s$ has finite endpoint $j(y_i) \in j(\breve\Theta^i(B_b))$. We also know that $\Omega^i(x)$ is an ideal endpoint of $(\ti f^i_\omega)_\#(\ti\gamma)$, and that $\Omega^i(x) \in \bdy\Gamma_s$. From the definition of overpaths it follows that the subray of $(\ti f^i_\omega)_\#(\ti\gamma)$ with finite endpoint $j(y_i)$ and infinite endpoint $\Omega^i(x)$ is the overpath of $(\ti f^i_\omega)_\#(\ti\gamma)$ associated to $V_s$, and therefore the corresponding geodesic overpath $\delta^i_s$ is a ray with infinite endpoint $\Omega^i(x)$, completing Case (a). In Case~(b) the line $\ti\gamma$ has ideal endpoints $x,y \in \bdy \Gamma_s$, and $\ti\gamma_T = \{V_s\}$, and the over--under decomposition of $\ti\gamma$ is just a single overpath line whose corresponding geodesic overpath is the line $\delta^0_s$ with ideal endpoints $x,y$. It follows that $(f^i_\omega)_\#(\ti\gamma)$ has ideal endpoints $\Omega^i(x),\Omega^i(y) \in \bdy\Gamma_s$ and $\Omega^i(\ti\gamma)_T = \Omega^i(\ti\gamma_T) = \{V_s\}$, and that the over--under decomposition of $\Omega^i(\ti\gamma_T)$ is a single overpath line whose corresponding geodesic overpath is the line $\delta^i_s$ with ideal endpoints $\Omega^i(x),\Omega^i(y)$, which equal the ideal endpoints of $\breve\Theta^i(\delta^0_s)$. This proves the Iteration Claim, subject to the Naturality Claim.

\medskip

To complete the proof of Lemma~\ref{LemmaOverpathsAttraction} it remains to prove the Naturality Claim. Denote the action of $\A$ on $T$ as $\A(V_s)=V_{s'}$, $\A(V_q)=V_{q'}$, $\A(E_b)=E_{b'}$. It follows that $\A(\bdy\Gamma_s)=\bdy\Gamma_{s'}$, $\A(\bdy\Gamma_q)=\bdy\Gamma_{q'}$, and $\A(\bdy\Gamma_b)=\Gamma_{b'}$. We consider several cases separately. First, $\ti\gamma_T$ degenerates to the $Q$-vertex $V_q$ if and only if $\bdy\ti\gamma \subset \bdy\Gamma_q$ if and only if $\A(\bdy\ti\gamma) \subset \bdy\Gamma_{q'}$ if and only if $\A(\ti\gamma_T)$ degenerates to $V_{q'}=\A(V_q)$. Next, $\ti\gamma_T$ degenerates to the $S$-vertex $V_s$ if and only if $\bdy\ti\gamma \subset \bdy\Gamma_s$ and for all components $\breve B_b$ of $\bdy_\ell \breve S_s$ we have $\bdy\ti\gamma \not\subset \bdy\Gamma_b$ if and only if $\A(\bdy\ti\gamma) \subset \bdy\Gamma_{s'}$ and for all components $\breve B_{b'}$ of $\bdy_\ell \breve S_{s'}$ we have $\A(\bdy\ti\gamma)  \not\subset \bdy\Gamma_{b'}$ if and only if $\A(\ti\gamma_T)$ degenerates to $V_{s'}=\A(V_s)$. 

For the nondegenerate cases it suffices to prove for each edge $E_b \subset T$ that $E_b \subset \ti\gamma_T$ if and only if $E_{b'}=\A(E_b) \subset \A(\ti\gamma_T)$, and we do this by a separation argument taking place in the Gromov compactification $\wt Y \union \bdy F_n$. To set up the argument we need further notation. Let $N \subset S$ be a regular neighborhood of $\bdy_\ell S$ and let $S^+ = \closure(S-N)$. We have $\bdy S^+ = \bdy_0 S^+ \union \bdy_\ell S^+$ where $\bdy_0 S^+ = \bdy_0 S$ and $\bdy_\ell S^+ = \bdy N - \bdy_\ell S$. The inclusions $\bdy_\ell S \rightarrow N \leftarrow \bdy_\ell S^+$ induce bijections of components. Let $\breve N \subset \breve S$ be the total lift of $N$ to the covering space $\breve S$, so $\breve N$ is an $F_n$-equivariant regular neighborhood of $\breve B = \bdy_\ell\breve S$. The inclusion $\breve B \subset \breve N$ induces a component bijection denoted $\breve B_b \leftrightarrow \breve N_b$. Let $\breve B^+_b = \bdy \breve N_b - \breve B_b$, and let $\breve B^+ = \union_b \breve B^+_b$, so the inclusion $\breve B^+ \subset \breve N$ induces the component bijection $\breve B^+_b \leftrightarrow \breve N_b$. The map $\breve j$ embeds $\breve B^+$ and $\breve S^+$ in $\wt Y$ with images $\wh B^+$ and $\wh S^+$, respectively, and with components denoted $\wh B^+_b = \breve j(\breve B_b)$, and $\wh S^+_s = \breve j(\breve S^+_s)$. Also let $\wh N_b = \breve j(\breve N_b)$ and let $\wh N = \union_b \wh N_b$.

Since $\A$ preserves the subgroup systems $[\pi_1 Q]$ and $[\pi_1 S]$, and since the outer automorphism of $\pi_1 S$ obtained by restricting $\A$ preserves the lower boundary subgroup system $[\bdy_\ell(\pi_1 S)]$, it follows that $\A$ is represented by a homotopy equivalence $\alpha \from Y \to Y$ restricting to maps $S^+ \mapsto S^+$, \, $Q \mapsto Q$, \, $B^+ \mapsto B^+$ and $j(N) \mapsto j(N) \union Q$. From this it follows that there is an $\A$-twisted equivariant lift $\ti \alpha \from \wt Y \to \wt Y$ restricting to $\wh S^+_b \mapsto \wh S^+_{b'}$, \, $\wh Q^+_q \mapsto \wh Q^+_{q'}$, \, $\wh B^+_b \mapsto \wh B^+_{b'}$ and $\wh N_b \mapsto \wh N_{b'} \union \wh Q_q$, where $V_q \in T$ is the $Q$-vertex incident to $E_b \subset T$.

Consider an edge $E_b \subset T$, with incident vertices $V_q$ and $V_s$. In $\wt Y$, the subset $\wt Y - \wh B^+_b$ has two components $\wh Y_{bq}$ and $\wh Y_{bs}$, where $\wh Y_{bq} \supset \wh Q_q$ and $\wh Y_{bs} \supset \wh S^+_s - \wh B^+_b$. In the Gromov compactification $\wt Y \union \bdy F_n$, the closure of the line $\wh B^+_b$ is compact arc $\wh B^+_b \union \bdy \Gamma_b$, and its complement $(\wt Y \union F_n) - (\wh B^+_b \union \bdy \Gamma_b)$ has two components $\wh Y_{bq} \union \bdy \wh Y_{bq}$ and $\wh Y_{bs} \union \bdy \wh Y_{bs}$ where $\bdy \wt Y_{bq}$ is the set of accumulation points of $\wt Y_{bq}$ in $\bdy F_n - \bdy \Gamma_b$, and similarly for $\bdy \wt Y_{bs}$. 

From the description of $\ti\alpha$ above (and the definition of realization of lines in $T$) it follows that $E_b \subset \ti\gamma_T$ if and only if $\bdy\ti\gamma$ has one point in $\bdy\wh Y_{bq}$ and one point in $\bdy\wh Y_{bs}$ if and only if $\A(\bdy\ti\gamma)$ has one point in $\bdy\wh Y_{b'q'}$ and one point in $\bdy\wh Y_{b's'}$ if and only if $E_{b'} \subset \A(\ti\gamma_T)$, completing the proof.
\end{proof}

\paragraph{Applying of Lemma~\ref{LemmaOverpathsAttraction}: Proof of Lemma~\ref{geometricFullHeightCase}.} We continue with the notations that were reviewed and established in the paragraphs surrounding the statement of Lemma~\ref{geometricFullHeightCase}.

As in the hypothesis of Lemma~2.19, let $\gamma \in \B$ be a line which is not weakly attracted to $\Lambda^+_\phi$, and whose realization in $G$ has height $r$ and so is not contained in $G_{r-1}$. It follows that $\gamma$ is not supported by $\F$ and so its realization in $G'$ (also denoted $\gamma)$ is not contained in $G'_{u-1}$. Choose a lift $\ti\gamma \subset \wt G'$ with realization $\ti\gamma_T$ in~$T$. We apply Lemma~\ref{LemmaOverpathsAttraction} with $\omega=\psi$, noting that $\Lambda^\un = \Lambda^+_\psi = \Lambda^-_\phi$. It follows that the line $\ti\gamma_T$ and the over--under decomposition of $\ti\gamma$ as related to $\Lambda^\un$ must match one of the forms in the conclusion of Lemma~\ref{LemmaOverpathsAttraction}. We go through the four conclusions one at a time, ruling out the first and using the rest to show that $\gamma$, in relation to $\Lambda^-_\phi$, matches one of the forms in the conclusion of Lemma~\ref{geometricFullHeightCase}. In some parts of the proof we assert that certain rays in $G'$ are principal rays of height~$u$, and these assertions are justified by application of Fact~\refGM{FactSingularRay}.

\medskip\textbf{First Case: Degenerate $Q$-vertex.} This case is ruled out since $\gamma$ is not in $Q=G'_{u-1}$. 

\smallskip\textbf{Second Case: Degenerate $S$-vertex.} In this case we have $\ti\gamma_T = V_s$ and $\ti\gamma = \ti\gamma_s \subset \wt G'_s$ with corresponding geodesic overpath $\breve\gamma_s \subset \breve S_s$ being either a leaf of $\breve\Lambda^\un_s$, or a geodesic line contained in the interior of a principle region of $\breve\Lambda^\un_s$, or a component of $\bdy_0 \breve S_s$. 

If $\breve\gamma_s$ is a leaf of $\wt\Lambda^\un_s$ then, by Proposition~\refGM{PropGeomLams}, $\gamma$ is a generic leaf of $\Lambda^-_\phi$, which matches conclusion~\pref{ItemGFHCLeaf} of Lemma~\ref{geometricFullHeightCase}.

If $\breve\gamma_s$ is a component of $\bdy_0 \breve S_s$ then $\ti\gamma \subset \wt G'_s$ projects to $\gamma$ in $G'$ that winds bi-infinitely around the circuit $\rho'_u$ or its inverse $\bar\rho'_u$, which matches conclusion~\pref{ItemGFHCIterate} of Lemma~\ref{geometricFullHeightCase}.

Suppose now that $\breve\gamma_s$ is a geodesic line contained in the interior of an upstairs principal region $\breve P$ of $\wt\Lambda^\un_s \subset \breve S_s$, and let $P \subset S$ be the downstairs principle region of $\Lambda^\un$ obtained by projecting~$\breve P$. Let $\bdy_\infinity \breve P \subset \bdy\Gamma_s \subset \bdy F_n$ denote the ideal points of $\breve P$. If $P$ is an ideal polygon then $\bdy_\infinity \breve P$ is a finite cyclically ordered set. Otherwise $P$ is a crown, there is a unique component $L$ of $\bdy \breve S_s$ which is also a component of $\bdy \breve P$, and $\bdy_\infinity \breve P$ is the union of the two points $\bdy_\infinity L$ with a countably infinite, linearly ordered, discrete subset of cusps, accumulating in opposite directions on the two points of $\bdy_\infinity L$. Since $\psi$ is rotationless, we may choose a representative $\Psi \in \Aut(F_n)$ representing $\psi$ so that $\Psi(\bdy\Gamma_s)=\bdy\Gamma_s$ and so that $\Psi$ fixes each point of $\bdy_\infinity \breve P$. Corresponding to $\Psi$ there is a lift $\ti f'_\Psi \from \wt G' \to \wt G'$ of $f'$ (see Section~\refGM{SectionLiftFacts}), and $\ti f'_\Psi$ preserves~$\wt G'_s$. There is also a unique corresponding lift $\breve\Theta \from \breve S_s \to \breve S_s$ of $\Theta$, using the correspondence under which $\ti d \composed \wt \Theta$ and $\ti f'_\Psi \composed \ti d \from \breve S_s \to \wt G'_s$ are $\Gamma_s$-equivariantly homotopic. This map $\breve\Theta$ preserves the principal region $\breve P$ and fixes each point of $\bdy_\infinity \breve P$, and each cusp of $\bdy_\infinity \breve P$ is an attracting point for the action of $\breve\Theta$ on $\bdy\Gamma_s$ (by Proposition~\refGM{PropNielsenThurstonTheory}).   

Consider the pair of ideal endpoints $\bdy_\infinity \ti\gamma = \bdy_\infinity \breve\gamma_s = \{\xi_1,\xi_2\} \subset \bdy_\infinity\breve P \subset \bdy F_n$. If $P$ is a crown and $\{\xi_1,\xi_2\} = \bdy_\infinity L$ as above, then $\breve\gamma_s=L$, contradicting that $\breve\gamma_s$ is contained in the interior of~$\breve P$. It follows that at least one of $\xi_1,\xi_2$ is a cusp of $\breve P$, say $\xi_1$. We orient $\ti\gamma$ and $\breve\gamma_s$ so that $\xi_1$ is each of their initial ideal endpoints. 

If $\xi_i$ is a cusp of $\breve P$ ($i=1,2$) then, since $\xi_i$ is an attracting point for the action of $\breve\Theta^\inv$ on $\bdy\Gamma_s$, it follows that: $\xi_i$~is~an attracting point for the action of $\Psi$ on all of $\bdy F_n$ (Fact~\refGM{LemmaFixPhiFacts}); $\xi_i$~is represented by a principal ray $\wt R_i = [\ti v_i,\xi_i) \subset \wt G'$ generated by an oriented edge $\wt E_i \subset \wt G'$ whose initial direction and initial vertex $\ti v_i$ are fixed by $\ti f'_\Psi$ (Fact~\refGM{FactSingularRay}). Also we have $\wt E_i \subset \wt H'_s$, for otherwise $\xi_i \not\in \bdy\Gamma_s$. 

Knowing that $\xi_1$ is a cusp of $\breve P$, we consider two cases depending on whether $\xi_2$ is also a cusp of $\breve P$.

Suppose that $\xi_2$ is a cusp of $\breve P$. Choose corresponding principal rays $\wt R_1,\wt R_2$ as above (the choice need not be unique, see Lemma~\refGM{LemmaPrincipalRayUniqueness}). Note that $\ti\mu = [\ti v_1,\ti v_2] \subset \wt G'$ is either a trivial path or a Nielsen path of $\ti f'_\Psi$. Let $\mu$ be the path in $G'$ to which $\ti\mu$ projects. The path~$\ti\mu$, if not trivial, decomposes uniquely into fixed edges and indivisible Nielsen paths of $\ti f'_\Psi$. Choose $\wt R_1,\wt R_2$ so as to minimize the number of terms of this decomposition of~$\ti\mu$. We claim that the interior of $\ti\mu$ is disjoint from the interiors of $\wt R_1$ and $\wt R_2$, implying that $\ti\gamma = \wt R_1^\inv \, \ti\mu \, \wt R_2$ (which we show matches conclusion~\pref{ItemGFHCTwoRays} of Lemma~\ref{LemmaOverpathsAttraction}, after verifying the form of~$\mu$). If the claim fails, if say $\interior(\ti\mu) \intersect \interior(\wt R_1) \ne \emptyset$, then the first term of the decomposition of $\ti\mu$ contains an edge of $\wt H'_u$. By Fact~\refGM{FactEGNPUniqueness} that term is a lift of $\rho'_u$ or $\bar\rho'_u$ that we denote $\alpha\bar\beta$, and so $\ti\mu = \alpha\bar\beta \ti\mu'$. Applying Lemma~\refGM{LemmaPrincipalRayUniqueness} it follows that $(\wt R_1 - \alpha) \union \beta$ is also a principle ray representing $\xi_1$, whose base point is connected to the base point of $\wt R_2$ by the path $\ti\mu'$, contradicting minimality. It remains to verify that the form of $\mu$ matches conclusion~\pref{ItemGFHCTwoRays} of Lemma~\ref{LemmaOverpathsAttraction}. If $\ti\mu$ is trivial we are done. Otherwise $\ti\mu$ is a Nielsen path for $\ti f'_\Psi$ and it projects to a Nielsen path $\mu$ for $f'_\Psi$. If $\mu$ has height~$u$ then it is an iterate of $\rho'_u$ or $\bar\rho'_u$ (by Fact~\refGM{FactEGNielsenCrossings}) and we are done. Otherwise $\mu$ has height~$\le u-1$ and we are also done.

Suppose that $\xi_2$ is not a cusp of $\breve P$, so $P$ is a crown and $\xi_2 \in \bdy_\infinity L$. We consider separately the cases $L \subset \bdy_\ell \breve S_s$ and $L \subset \bdy_0 \breve S_s$. If $L \subset \bdy_\ell \breve S_s$ then $L = \breve B_b$ for some edge $E_b \subset T$ incident to $V_s$; let $V_q$ be the opposite $Q$-vertex of $E_b$. It follows that the straightened image of $\ti d(L)$ in $\wt G'_s$ is the unique line in $\wh B_b = \wh Q_q \intersect \wt G'_s = \wh Q_q \intersect \wh\Sigma_s$, and that $\xi_2 \in \bdy\Gamma_b$ is an ideal endpoint of that line. We may therefore write $\ti\gamma$ as a back-to-back concatenation of rays $\ti\gamma = \wt R_1^\inv \wt R^{\vphantom{\inv}}_2$ where $R_2$ is the maximal subpath of $\ti\gamma$ contained in $\wh Q_q \intersect \wt G'_s$, and $\wt R_1 = \ti\gamma \setminus \wt R_2$ is the minimal subpath of $\ti\gamma$ containing every edge of $\wt H'_u \intersect \ti\gamma$. Since $\ti f'_\Psi$ is a principal lift fixing $\xi_1,\xi_2$ it follows that $\ti f'_\Psi$ fixes the common base point of $\wt R_1$ and $\wt R_2$ and fixes the initial direction of $\wt R_1$, and therefore $\wt R_1$ is a principal ray representing $\xi_1$, matching conclusion~\pref{ItemGFHCOneRay} of Lemma~\ref{LemmaOverpathsAttraction}. If $L \subset \bdy_0 \breve S_s$ then the same analysis works except that the straightened image of $\ti d(L)$ is a lift of the bi-infinite iterate of $\rho'_u$ or $\bar\rho'_u$, and the subray $\wt R_2 \subset \ti\gamma$ is the maximal subray of $\ti\gamma$ that is a lift of a singly infinite iterate of $\rho'_u$ or $\bar\rho'_u$. It still holds that $\ti f'_\Psi$ fixes the common base point of $\wt R_1$ and $\wt R_2$ and the initial direction of~$\wt R_1$, and that $\wt R_1$ is a principal ray representing $\xi_1$, also matching conclusion~\pref{ItemGFHCOneRay}.

\smallskip\textbf{Third Case: One edge.} Up to reversal of orientation we have an over--under decomposition $\ti\gamma = \ti\gamma_s \ti\gamma_q$, and $\ti\gamma_T = E_b$ with $Q$-endpoint $V_q$ and $S$-endpoints $V_s$. Note that the proper geodesic overpath $\breve\gamma_s$ corresponding to $\ti\gamma_s$ is contained in a principal region $\breve P$ of $\breve\Lambda^\un_s$ covering a crown principal region~$P \subset S$. We may choose the principle automorphism $\Psi \in \Aut(F_n)$ representing $\psi$, and the lift $\breve\Theta \from \breve S_s \to \breve S_s$ of $\Theta$, as in the case ``Degenerate $S$-point'', fixing each point of $\bdy_\infinity \breve P$, preserving $B_b$, and preserving $\bdy\Gamma_s$, $\bdy\Gamma_q$, and $\bdy\Gamma_b$. Let $\ti f'_\Psi$ be the principle lift of $f'$ that corresponds to $\Psi$. The rays $\wt R_1$ and $\wt R_2 = \ti\gamma_l \subset \wh Q_q$ have a common finite endpoint $x \in \wh Q_q \intersect \wt G'_s$ fixed by $\ti f'_\Psi$, and the initial direction of $\ti\gamma_s$ is in $\wt H'_{u,s}$ and is fixed by $\ti f'_\Psi$. It follows that $\wt R_1$ is a principal ray representing $\xi_1$, matching conclusion~\pref{ItemGFHCOneRay} of Lemma~\ref{LemmaOverpathsAttraction}.

\smallskip\textbf{Fourth Case: Two edges and two $S$-endpoints.} We have over--under decomposition $\ti\gamma = \ti\gamma_{s-} \union \ti\gamma_b \union \ti\gamma_{s+}$. By a very similar analysis as in the previous case, carried out on each of the rays $\wt R_1 = \ti\gamma_{s-}$ and $\wt R_2 = \ti\gamma_{s_+}$, one proves that these are both principal rays, and setting $\mu = \ti\gamma_q$ (which is either trivial or in $\wh Q_q$), we have matched conclusion~\pref{ItemGFHCTwoRays} of~Lemma~\ref{LemmaOverpathsAttraction}, completing the proof of the lemma
\qed

\paragraph{A further application of Lemma~\ref{LemmaOverpathsAttraction}.} Proposition~\ref{PropWeakGeomRelFullIrr} to follow will, in \PartFour, be incorporated as one of the conclusions of Theorem J, the relative, general version of Theorem I which was stated in the Introduction. 

As an example of Proposition~\ref{PropWeakGeomRelFullIrr}, consider a compact surface $S$ with nonempty boundary, and a subset $\bdy_\ell S \subset \bdy S$ consisting of all but one component of $\bdy S$. From elementary topology one knows that $S$ deformation retracts to an embedded finite graph containing $\bdy_\ell S$, and so the inclusion $\bdy_\ell S \subset S$ determines a free factor system $[\bdy_\ell S]$ in the free group $\pi_1 S$ having one rank~$1$ component for each component of $\bdy_\ell S$. For any $\phi \in \MCG(S) \subgroup \Out(\pi_1 S)$, if $\phi$ is a pseudo-Anosov element in $\MCG(S)$ then, regarded in $\Out(\pi_1 S)$, $\phi$ is fully irreducible relative to $[\bdy_\ell S]$. In the very special case that $\bdy S$ has just one component, one obtains the ``well known''  theorem saying that if $\phi$ is pseudo-Anosov then $\phi$ is fully irreducible in the absolute sense.


\begin{proposition}\label{PropWeakGeomRelFullIrr} Let $\F$ be a free factor system, let $Q \disjunion S \xrightarrow{j} Y \xrightarrow{d} G$ be the static data of a geometric model for some \ct\ whose top stratum is \eg-geometric, and suppose that $Q \subset G$ represents $\F$. Given $\omega \in \Out(F_n)$, if $\omega(\F)=\F$, and if there exists a topological representative $f_\omega \from G \to G$ of $\omega$ and a pseudo-Anosov homeomorphism $\Theta \from S \to S$ which semiconjugates to $f_\omega$, then $\omega$ is fully irreducible rel~$\F$. 
\end{proposition}


\begin{proof} We may assume that $f_\omega(Q)=Q$. Let $\Lambda^\un,\Lambda^\st$ be the unstable/stable lamination pair of $\Theta$. As was done just prior to the statement of Lemma~\ref{LemmaOverpathsAttraction}, we may apply the proof of Proposition~\refGM{PropGeomLams} to our present situation, with the conclusion that the induced map $dj_\# \from \B(\pi_1 S) \to \B(G)=\B(F_n)$ takes the laminations $\Lambda^\un,\Lambda^\st \subset \B(\pi_1 S)$ to a lamination pair $\Lambda^+_\omega$, $\Lambda^-_\omega$ for~$\omega$. Applying Proposition~\refGM{PropGeomLams}~\prefGM{ItemLamSurfSupport}, $\F_\supp(\Lambda^\pm_\omega) = \F_\supp[\pi_1 S]$. It follows that $\F_\supp(\F,\Lambda^\pm_\omega) = \F_\supp([\pi_1 Q],[\pi_1 S]) = \{[F_n]\}$, the last equation following from Lemma~\trefGM{LemmaScaffoldFFS}{ItemRelFFS}. The lamination pair $\Lambda^\pm_\omega$ therefore fills rel~$\F$. Note also by Lemma~\trefGM{LemmaScaffoldFFS}{ItemRelFFS} that every conjugacy class carried by the rank~$1$ subgroup system $[\bdy_0 S]$ fills rel~$\F$.

We next prove that for every nontrivial $\tau \in F_n$, if its conjugacy class $[\tau]$ is not weakly attracted to~$\Lambda^+_\omega$ under the action of $\omega$ on $\B$ then $[\tau]$ is carried by~$\F \union \{[\bdy_0 S]\}$. Let $\gamma$ be the line in $G'$ which wraps bi-infinitely around the circuit in $G'$ representing~$[\tau]$. Let $\ti\gamma \subset \wt G'$ be the lift of $\gamma$ which is the axis of $\tau$ in~$\wt G'$. Since the over--under decomposition of $\ti\gamma$ is evidently $\tau$-invariant, three cases can occur: (1) $\ti\gamma$ is a bi-infinite concatenation alternating between overpaths and underpaths; (2) all of $\ti\gamma$ is an underpath; (3) all of $\ti\gamma$ is an overpath. But Lemma~\ref{LemmaOverpathsAttraction} rules out case~(1), since $[\tau]$ is not weakly attracted to $\Lambda^+_\omega$. In case~(2) $[\tau]$ is carried by $\F$. In case~(3) it follows that $\tau \in \Gamma_s$ for some $s$ and that the corresponding geodesic overpath $\breve\gamma_s$ is the axis of the action of $\tau$ on $\breve S_s$. Lemma~\ref{LemmaOverpathsAttraction} implies that one of three possibilities holds: (3a) $\breve\gamma_s$ is a leaf of $\Lambda^\un$; (3b) $\breve\gamma_s$ is contained in the interior of a principle region of $\Lambda^\un$; (3c) $\breve\gamma_s$ is a component of $\bdy_0\breve S_s$. Cases (3a) and (3b) contradict that $\breve\gamma_s$ is the axis of $\tau$, and Case (3c) implies that $[\tau]$ is carried by $\{[\bdy_0 S]\}$.

The conclusion of the proof is a general argument. Assuming by contradiction that $\phi$ is not fully irreducible rel~$\F$, after passing to a rotationless power there is a $\phi$-invariant free factor system~$\F'$ with proper inclusions $\F \sqsubset \F' \sqsubset \{[F_n]\}$, and there is a \ct\ representative $f \from G \to G$ of $\omega$ with properly included core filtration elements $G_s \subset G_{s'} \subset G_t=G$ representing $\F \sqsubset \F' \sqsubset \{[F_n]\}$ respectively. Since $\Lambda^\pm_\phi$ fills rel~$\F$, the stratum of $G$ corresponding to $\Lambda^+_\omega$ is the top stratum~$H_t$ and is an \eg-stratum. By Definition~\ref{defn:Z} and Corollary~\ref{CorPMna}~\pref{ItemAnaDependence}, the nonattracting subgroup system $\A_\na(\Lambda^+_\omega)$ has the form $[G_{t-1}]$ if $\Lambda^+_\omega$ is nongeometric or $[G_{t-1}] \union \{[\<\rho\>]\}$ if $\Lambda^+_\omega$ is geometric (the second case holds in our situation, but we do not make use of that). In either case there exists a circuit that is not carried by the core graph $[G_s]$ but is carried by the core graph $[G_{t-1}]$. It follows that there exists a conjugacy class $[\tau]$ that is not carried by $\F$ but is carried by $\F'$, and is therefore not weakly attracted to~$\Lambda^+_\omega$. Applying the previous paragraph, $\tau$ is carried by $[\bdy_0 S]$, implying that $[\tau]$ fills rel~$\F$ and is therefore not carried by~$\F'$, a contradiction.
\end{proof}

\subsection{General nonattracted lines and the Proof of Theorem G} 
\label{SectionNALinesGeneral}

We are given rotationless $\phi, \, \psi=\phi^\inv \in \Out(F_n)$, a lamination pair $\Lambda^\pm_\phi \in \L^\pm(\phi)$, and a \ct\ $\fG$ representing $\phi$ with \eg\ stratum $H_r$ corresponding to $\Lambdapp$ (note that we are abandoning the notational conventions of Section~\ref{SectionTheoremGStatement}). 

From Definition~\ref{defn:Z} we have the path set $\<Z,\hat\rho_r\> \subset \wh\B$. From Lemma~\ref{LemmaZPClosed} this path set is a groupoid and each line $\gamma \in \<Z,\hat\rho_r\>$ is carried by $\A_{\na}(\Lambdapmp)$. Recall also Lemma~\ref{LemmaThreeNASets} which says that $\gamma\in\B_\na(\Lambdapp)$ as long as it satisfies at least one of the following conditions.
\begin{enumerate}  
\item\label{ItemGammaCarriedNA}
 $\gamma$ is carried by $\A_{\na}(\Lambdapp)$. 
\item $\gamma \in \B_\sing(\psi)$.
\item $\gamma \in \B_\gen(\psi)$.
\end{enumerate} 
Since each line carried by $\A_\na(\Lambdapp)$ is in $\B_\ext(\Lambdapmp;\psi)$ it follows that for each line $\gamma \in \<Z,\hat\rho_r\>$ we have $\gamma \in \B_\ext(\Lambda)$; we use this repeatedly in this section. 

To simplify the notation of the proof we define the set of \emph{good lines} in $\B$ to be 
$$\B_\good(\Lambdapmp;\psi) = \B_\ext(\Lambdapmp;\psi) \union \B_\sing(\psi) \union \B_\gen(\psi)
$$
and we repeatedly use Proposition~\ref{PropStillClosed} which with this notation says that $\B_\good(\Lambdapmp;\psi)$ is closed under concatenation.

The conclusion of Theorem~G says that $\B_\na(\Lambdapp) = \B_\good(\Lambdapmp;\psi)$. One direction of inclusion, namely $\B_\na(\Lambdapp) \supset \B_\good(\Lambdapmp;\psi)$, follows from Lemma~\ref {LemmaThreeNASets} and the fact that each line in $\B_\ext(\Lambdapmp;\psi)$ is a concatenation of lines in $\B_\sing(\psi)$ and lines carried by $\A_\na(\Lambdapmp)$. 

\smallskip

We turn now to proof of the opposite inclusion $\B_\na(\Lambdapp) \subset \B_\good(\Lambdapmp;\psi)$. Given $\gamma \in \B_\na(\Lambdapp)$, if the height of $\gamma$ is less than $r$ then $\gamma \in \<Z,\hat \rho_r\>$ and we are done. Henceforth we proceed by induction on height. Define an \emph{inductive concatenation} of $\gamma$ to be an expression of $\gamma$ as a concatenation of finitely many lines in $\B_\good(\Lambdapmp;\psi)$ and at most one line $\nu$ of height lower than $\gamma$. If we can show that $\gamma$ has an inductive concatenation, we prove that $\gamma$ is good as follows. In some cases $\nu$ does not occur in the concatenation and so $\gamma$ is good by Proposition~\ref{PropStillClosed}. Otherwise, using invertibility of concatenation, it follows that $\nu$ is expressed as a concatenation of good lines plus the line $\gamma$, all of which are known to be in $\B_\na(\Lambdapp)$. Applying Lemma~\ref{LemmaConcatenation} we therefore have $\nu \in \B_\na(\Lambdapp)$. Applying induction on height it follows that $\nu$ is good, and so again $\gamma$ is good by Proposition~\ref{PropStillClosed}.

The induction step breaks into two major cases, depending on whether or not the stratum of the same height as $\gamma$ is \noneg\ or \eg. For the case of an \noneg\ stratum we will use the following:

\begin{lemma}\label{LemmaInverseIsPrincipal} Suppose that $\phi, \psi=\phi^{-1} \in \Out(F_n)$ are rotationless, that $\fG$ is a \ct\ representing $\phi$, that $E_s$ is the unique edge in an \noneg\ stratum $H_s$, and that both endpoints of $E_s$ are contained in $G_{s-1}$. Let $\wt E_s$ be a lift of $E_s$, let $\ti f: \wt G \to \wt G$ be the  lift of $f$ that fixes the initial endpoint of $\wt E_s$ and let $\Phi$ be the  automorphism corresponding to $\ti f$.   Then $\Psi = \Phi^{-1}$  is principal. Moreover there is a line $\sigma \in \B_\sing(\psi)$ that has height $s$, that crosses  $E_s$ exactly once and that lifts to  a line with endpoints in $\Fix_N(\Psi)$. 
\end{lemma}

\begin{proof} By Fact~\refGM{FactContrComp} and Definition~\trefGM{DefCT}{ItemZeroStrata}, no component of $G_{s-1}$ is contractible. Letting $\wt C_1, \wt C_2 \subset \wt G$ be the components of the full pre-image of $G_{s-1}$ that contain the initial and terminal endpoints of $\wt E_s$ respectively,  there are nontrivial free factors $B_1,B_2$ that satisfy $\bdy B_j = \bdy \wt C_j$. Each of $\wt C_1, \wt C_2$ is preserved by $\ti f$ and so each of $B_1,B_2$ is $\Psi$-invariant. By Fact~\refGM{FactPeriodicNonempty} applied to $\Psi \restrict B_j$, there exists $m > 0$ and points $P_j \in \Fix_N(\wh\Psi^m) \cap \partial \wt C_j$ for $j=1,2$. Since the line  $\ti \sigma$ connecting $P_1$ to $P_2$ is not birecurrent it does not project to either an axis or a  generic leaf of some element of $\L(\phi^{-1})$.   Thus  $\Psi^m \in P(\psi)$.  Since $\psi$ is rotationless,   $\Psi \in P(\psi)$ and $\sigma \in \B_\sing(\psi)$.
\end{proof}

Fix now $s \ge r$ and assume as an induction hypothesis that all lines in $\B_\na(\Lambdapp)$ of height $<s$ are in $\B_\good(\Lambdapmp;\psi)$. Fix a line $\gamma \in \B_\na(\Lambdapp)$ of height $\le s$. Let $\ti \gamma$ be a lift of $\gamma$ and let $P$ and $Q$ be its initial and terminal endpoints respectively. 

\paragraph{Case 1:  $H_s$ is \noneg.} Let $E_s$ be the unique edge in $H_s$. If $E_s$ is closed and $\gamma$ is a bi-infinite iterate of $E_s$ then $E_s \subset Z$ and $\gamma \in \<Z,\hat\rho_r\>$ so $\gamma \in \B_\ext(\Lambdapmp;\psi)$. We may therefore assume that both endpoints of $E_s$ belong to $G_{s-1}$.

Orient $E_s$ so that its initial direction is fixed. Recall (Lemma 4.1.4 of \BookOne) that for each occurrence of $E_s$ or $\overline E_s$ in the representation of $\gamma$ as an edge path, the line $\gamma$ splits at the initial vertex of $E_s$, and we refer to this as the \emph{highest edge splitting vertex} determined by the occurrence of $E_s$. We also use this terminology for lifts of $E_s$ in the universal cover. By Fact~\refGM{FactPrincipalVertices}, highest edge splitting vertices are principal.

\paragraph{Case 1A:  Both ends of $\gamma$ have height $s$.}  In this case $\gamma$ has a splitting in which each term is finite.  Since $\gamma$ is not weakly attracted to $\Lambda^+$, neither is any of the terms in the splitting.  Lemma~\ref{defining Z}~\pref{ItemZPFinitePaths} implies that each term  is contained in $ \<Z,\hat \rho_r\>$ and so $\gamma$ is contained in $ \<Z,\hat \rho_r\>$ and we are done. 

\paragraph{Case 1B:  Exactly one end of $\gamma$ has height $s$.} We assume without loss that the initial end of $\gamma$ has height $s$. Pick a lift $\ti\gamma$, let $\wt E_s$ be the last lift of $E_s$ crossed by $\ti\gamma$, let $\ti x \in \ti\gamma$ be the highest edge splitting vertex determined by $\wt E_s$, and let $\ti \gamma = \wt R_-^\inv  \cdot \wt R_+$\ be the splitting at~$\ti x$. The ray $\wt R_-$ has height $s$ and crosses lifts of $E_s$ infinitely often, and as in case 1A the projected ray $R_-$ is contained in $\<Z,\hat\rho_r\>$. It follows that there exists a  subgroup $A \in \A_\na(\Lambdapp)$ such that $P \in \bdy A$. Let $\ti f$ be the lift of $f$ that fixes $\ti x$ and  let $\Phi$ be the corresponding element of $P(\phi)$. Lemma~\ref{LemmaInverseIsPrincipal} implies that $\Psi =\Phi^{-1} \in P(\psi)$. 

We claim that $A$ is $\Phi$-invariant. By Lemma~\ref{LemmaZPClosed}~\pref{item:UniqueLift} it suffices to show that $\wh\Phi(\partial A) \cap \partial A \ne \emptyset$.  This is obvious if $P \in \Fix(\wh\Phi)$ so assume otherwise. The ray $f_\#( R_-)$ is contained in $ \<Z,\hat \rho_r\>$ by Lemma~\ref{defining Z}~\pref{ItemZPPathsInv} so  $P$ and $  \wh\Phi(P)$ bound a line that projects into $ \<Z,\hat \rho_r\>$ and so is carried by $\A_{\na}(\phi)$.  Another application of Lemma~\ref{LemmaZPClosed}~\pref{item:UniqueLift} implies that $\wh\Phi(P) \in \partial A$ as desired.  

By Fact~\refGM{FactPeriodicNonempty} applied to $\Psi \restrict A$, the set $\Fix_N(\wh\Psi^m) \cap \partial A$ is nonempty for some $m > 0$.  Since $\psi$ is rotationless and $\Psi \in P(\psi)$, we may take $m=1$, from which it follows that $\Psi$ is $A$-related. By Lemma~\ref{LemmaInverseIsPrincipal}, there exist $P',Q' \in \Fix_N(\Psi)$ so that the line $\ti \sigma = \overline{P' Q'}$ crosses $\wt E_s$ in the same direction as $\ti \gamma$ and crosses no other edge of height $\ge s$, and $\ti\sigma$ projects to $\sigma \in \B_\sing(\psi)$. Let $\ti\sigma = \wt R'_-{}^\inv \cdot \wt R'_+$ be the highest edge splitting determined by $\ti x$. Assuming that $P \ne P'$, the line $\ti\mu = \overline{P P'}$ has endpoints in $\bdy A \union \Fix_N(\Psi)$ and so projects to $\mu \in \B_\ext(\Lambda^+_\phi)$. If $\gamma$ crosses $\wt E_s$ in the backwards direction then $\wt E_s$ is the last edge of both $\wt R_-$ and $\wt R'_-$ and each of $\wt R_+$ and $\wt R'_+$ have height $\le s-1$; otherwise each of $\wt R_+$ and $\wt R'_+$ is a concatenation of $\wt E_s$ followed by a ray of height $\le s-1$. In either case, assuming that $Q \ne Q'$, it follows that the line $\ti\nu = \overline{Q'Q}$ has height $\le s-1$. We therefore have an inductive concatenation $\gamma = \mu \diamond \sigma \diamond \nu$, with $\mu$ omitted when $P=P'$ and $\nu$ omitted when $Q=Q'$, and Case 1B is completed.

\bigskip
 
\paragraph{Case 1C: Neither end of $\gamma$ has height $s$.} We induct on the number $m$ of height $s$ edges in $\gamma$. The base case, where $m=0$, follows from induction on $s$. Let $\ti \gamma = \wt R_-^\inv \cdot \wt R_+$\ be the splitting determined by the last highest edge splitting vertex $\ti x$ in~$\ti \gamma$, let $\ti f$ be the lift of $f$ that fixes $\ti x$, and let $\Phi \in P(\phi)$ correspond to $\ti f$. As in Case~1B, from Lemma~\ref{LemmaInverseIsPrincipal} it follows that $\Psi =\Phi^{-1} \in P(\psi)$ and that there exist $P',Q' \in \Fix_N(\Psi)$ so that the line $\ti \sigma$ connecting $P'$ to $Q'$ crosses the last height $s$ edge of $\ti\gamma$ in the same direction as $\ti \gamma$ and crosses no other edge of height $\ge s$. Let $\ti \sigma = \wt R'_-{}^\inv \cdot \wt R'_+$ be the highest edge splitting determined by $\ti x$. The line $\mu_1 = \overline{P P'}$ is obtained by tightening $\wt R_-^\inv \wt R'_-$, and the line $\mu_2 = \overline{Q' Q}$ is obtained by tightening  and $\wt R'_+{}^\inv \wt R_+$. These lines have height $\le s$, cross fewer than $m$ edges of height $s$, and are not weakly attracted to~$\Lambdapp$ by Lemma \ref{LemmaConcatenation}, because the rays $R_- ,R_+, R'_-$ and $R_+'$ are not weakly attracted to~$\Lambdapp$. By induction on $m$ we have $\mu_1,\mu_2 \in \B_\good(\Lambdapmp;\psi)$. Since $\sigma \in \B_\sing(\psi)$, it follows that $\gamma = \mu_1 \diamond \sigma \diamond \mu_2 \in \B_\good(\Lambdapmp;\psi)$, completing Case 1C.

\paragraph{Case 2: $H_s$ is \eg.} Let $\Lambda^+_s \in \L(\phi)$ be the lamination associated to $H_s$ with dual lamination denoted $\Lambda^-_s \in \L(\psi)$. Applying Theorem~\refGM{TheoremCTExistence} with $\C$ being $[G_r] \sqsubset [G_s]$, let $f' \from G' \to G'$ be a \ct\ representing $\psi$ with \eg\ stratum $H'_{r'}$ associated to $\Lambda^-_\phi$ and \eg\ stratum $H'_{s'}$ associated to $\Lambda_s^-$ so that $[G_r] = [G'_{r'}]$ and $[G_s] = [G'_{s'}]$. Let $\gamma'$ be the realization of $\gamma$ in~$G'$, a line of height~$s'$. Using the $F_n$-equivariant identification $\bdy\wt G_s \approx \bdy \wt G'_{s'}$, there is a lift $\ti\gamma'$ of $\gamma'$ with endpoints $P,Q$.

\paragraph{Case 2A: $\gamma$ is not weakly attracted to $\Lambda^+_s$.} This is the case where we apply Lemmas~\ref{nonGeometricFullHeightCase} and \ref{geometricFullHeightCase}. 
In the situation where $\gamma'$ is a singular line of $\psi$ or a generic leaf of $\Lambda^-_s$, or in the geometric situation where $\gamma'$ is a bi-infinite iterate of the height $s'$ closed indivisible Nielsen path $\rho'_{s'}$, we have $\gamma' \in \B_\good(\Lambdapmp;\psi)$ and we are done. The situation where $\gamma'$ is a singular line of $\psi$ includes all cases of Lemmas~\ref{nonGeometricFullHeightCase} and \ref{geometricFullHeightCase} where $\gamma' = R_-^\inv \mu R_+$, each of $R_-,R_+$ is either a height $s'$ principal ray or a singly infinite iterate of a height $s'$ closed indivisible Nielsen path, and $\mu$ is either trivial or a height $s'$ Nielsen path. We may therefore assume that none of these situations occurs. In all remaining situations, we divide into two subcases depending on whether one or two ends of $\gamma'$ have height $s'$.

Consider first the subcase where only one end of $\gamma'$, say the initial end, has height $s'$. By applying Lemma~\ref{nonGeometricFullHeightCase}~\pref{ItemNGFHCOneRay} or Lemma~\ref{geometricFullHeightCase}~\pref{ItemGFHCOneRay} we obtain a decomposition $\gamma' =  R_-^\inv \mu R_+$ where $R_-$ is a height $s'$ principal ray, $\mu$ is either a trivial path or a height $s'$ Nielsen path, and the ray $R_+$ has height $<s'$. Lifting the decomposition of $\gamma'$ we obtain a decomposition $\ti\gamma' =  \wt R_-^\inv \ti\mu \wt R_+$ where $\wt R_-$, $\wt R_+$ have endpoints $P,Q$. Let $x$ be the initial point of $R_+$, lifting to the initial point $\ti x$ of $\wt R_+$. The component $\Gamma$ of the full pre-image of $G'_{s'-1}$ that contains $\ti x$ is $\ti f'$-invariant and infinite and so there is a free factor $B$ such that $Q \in \partial B = \partial \Gamma$.  Since $\Psi$ is principal, Lemma~\refGM{FactPeriodicNonempty} implies the existence of $Q' \in \Fix_N(\wh\Psi) \cap \partial \Gamma$. The line $\ti \tau'$  connecting $P$ to $Q'$ projects to $\tau' \in \B_\sing(\psi)$; let $\tau$ be the realization of $\tau'$ in $G$. The line $\ti\nu' = \ti\tau'{}^\inv \diamond \ti\gamma'$ is contained in $\Gamma$ and so projects to a line $\nu' = \tau'{}^\inv \diamond \gamma'$ of height $<s'$ whose realization in $G$ is a line $\nu$ of height $< s$. We obtain an inductive concatenation $\gamma = \tau \diamond \nu$, completing the first subcase of Case 2A.

Consider next the subcase where both ends of $\gamma'$ have height $s'$. Applying Lemma~\ref{nonGeometricFullHeightCase}~\pref{ItemIntermediatePath} or Lemma~\ref{geometricFullHeightCase}~\pref{ItemGFHCTwoRays}, and keeping in mind the situations that we have assumed not to occur, there is a decomposition $\gamma' = R_1^\inv \mu R_2$ where $R_1$, $R_2$ are both height $s'$ principal rays, and $\mu$ has one of the forms $\beta$, $\alpha\beta$, $\beta\bar\alpha$, $\alpha\beta\bar\alpha$ where $\beta$ is a nontrivial path of height $<s'$ and $\alpha$ (if it occurs) is a height $s'$ nonclosed indivisible Nielsen path oriented to have initial vertex in the interior of $H'_{s'}$ and terminal vertex in $G'_{s'-1}$. Absorbing occurrences of $\alpha$ into the incident principal rays $R_1,R_2$, we obtain rays $R_-,R_+$ containing $R_1,R_2$ respectively, and a decomposition $\gamma' = R_-^\inv \beta R_+$ which lifts to a decomposition $\ti\gamma' = \wt R_-^\inv \ti\beta \wt R_+$ where $\wt R_-$ has endpoint $P$ and $\wt R_+$ has endpoint $Q$. Let $\ti x$ be the initial point of $\wt R_-$. There is a principal lift $\ti f' \from \wt G' \to \wt G'$ with associated $\Psi \in P(\psi)$ such that $\wt R_1$ is a principal ray for $\ti f'$ fixing the initial point $\ti y$ of $\wt R_1$. Since either $\ti x=\ti y$ or the segment $[\ti x,\ti y]$ is a lift of $\alpha$, it follows that $\ti f'$ fixes $\ti x$ and that $\ti f'_\#(\wt R_-) = \wt R_-$. As in the previous subcase there is a ray based at $\ti x$ with height $< s'$ and terminating at some $Q' \in \Fix_N(\wh\Psi)$. The line $\ti \tau'$ connecting $P$ to $Q'$ projects to $\tau' \in \B_\sing(\psi)$ which is good, and the line $\ti\sigma' = \ti \tau'{}^\inv \diamond \gamma' $ has only one end with height $s'$. By the previous subcase, the realization $\sigma$ of $\bar \tau\diamond \gamma$ in $G$ is good and hence $\gamma =  \tau \diamond\sigma$ is good.
 
\paragraph{Case 2B:  $\gamma$ is weakly attracted to $\Lambda^+_s$.} In this case $H_s \subset Z$, for otherwise $\gamma$ is weakly attracted to $\Lambda^+_\phi$ as well, contrary to hypothesis.

\subparagraph{Special case:} We first consider the special case that $\gamma$ decomposes at a fixed vertex $v$ into two rays $\gamma = \gamma_1 \gamma_2$ so that $\gamma_1$ has height $<s$ and $\gamma_2 \in \<Z,\hat\rho_r\>$. In $\wt G$ there is a corresponding decomposition $\ti\gamma = \ti\gamma_1 \ti\gamma_2$ at a vertex~$\ti v$, and there is a lift $\ti f$ fixing $\ti v$ with corresponding $\Phi \in \Aut(F_n)$ representing $\phi$. Let $\Psi = \Phi^\inv$. 

Recall the notation established in Definition~\ref{defn:Z} of the graph immersion $h \from K \to G$ used to define $\A_\na(\Lambdapp)$. Since the ray $\gamma_2$ is an element of the path set $\<Z,\hat\rho_r\>$, it follows from Definition~\ref{defn:Z} that $\gamma_2$ lifts via the immersion $h \from K \to G$ to a ray in the finite graph $K$. The image of this lifted ray must therefore be contained in a noncontractible component $K_0$ of~$K$. There is a lift of universal covers $\ti h \from \wt K_0 \to \wt G$ such that $\ti h(\wt K_0)$ contains $\ti\gamma_2$ and such that the stabilizer of $\ti h(\wt K_0)$ is a subgroup $A \in \A_\na(\Lambdapmp)$ whose conjugacy class is the one determined by the immersion $h \from K_0 \to G$. By construction we have $Q \in \bdy A$. If $\wh\Phi(Q) \ne Q$ then $Q$ and $\wh\Phi(Q)$ bound a line that projects into $\<Z,\hat \rho_r\>$ and so is carried by $ \A_{\na}(\Lambdapmp)$, and by applying Lemma~\ref{LemmaZPClosed}~\pref{item:UniqueLift} it follows that $\wh Q \in \bdy A$; this is also true if $\wh\Phi(Q) = Q$. In particular $\Phi$, and therefore also $\Psi$, preserves~$A$.  By Fact~\refGM{FactPeriodicNonempty} applied to $\Psi \restrict A$ there exists an integer $q \ge 1$ so that $\Fix_N(\wh\Psi^q) \cap \partial A \ne \emptyset$; we choose $q$ to be the minimal such integer and then we choose $Q' \in \Fix_N(\wh\Psi^q) \cap \partial A$. If $Q \ne Q'$ then the line $\beta$ connecting $Q$ to $Q'$ is carried by $A$ and so $\beta \in \B_\good(\Lambdapmp;\psi)$. 
      
The component $C$ of $G_{s-1}$ that contains the ray $\gamma_1$ is noncontractible, and letting $\wt C$ be the component of the full pre-image of $C$ that contains $\gamma_1$, the stabilizer of $\wt C$ is a nontrivial free factor $B$ such that $\partial B = \partial \wt C$. By construction we have $P \in \bdy B$. Also $\wt C$ is invariant under $\ti f$ and so $B$ is invariant under $\Psi$. By Fact~\refGM{FactPeriodicNonempty} applied to $\Psi \restrict B$ there exists an integer $p \ge 1$ so that $\Fix_N(\wh\Psi^p) \intersect \bdy B \ne \emptyset$; we choose $p$ to be the minimal such integer and then we choose $P' \in \Fix_N(\wh\Psi^p) \cap \partial B$. If $P \ne P'$ then the line $\nu$ connecting $P$ to $P'$ has height~$<s$. 

For some least integer $m > 0$ we have $P',Q' \in \Fix_N(\Psi^m)$. If $P' \ne Q'$, consider the line $\mu$ connecting $P'$ to $Q'$. By hypothesis $\psi$ is rotationless and so $\Psi$ is principal if and only if $\Psi^m$ is principal. It follows that if $\Psi$ is principal then $m=1$ and $\mu \in \B_\sing(\psi)$, whereas if $\Psi$ is not principal then $\Fix_N(\Psi^m) =\{P',Q'\}$ so $m = p = q = 1$ or $2$ and either $\mu \in \B_\gen(\psi)$ or $\mu$ is a periodic line corresponding to a conjugacy class that is invariant under $\phi^2$. In all cases, $\mu \in \B_\good(\Lambdapmp;\psi)$. 

We therefore have an inductive concatenation of the form $\gamma = \nu \diamond \mu \diamond \bar\beta$, where $\nu$ is omitted if $P=P'$, $\mu$ is omitted if $P'=Q'$, and $\bar\beta$ is omitted if $Q'=Q$, but at least one of them is not omitted because $P \ne Q$. This completes the proof in the special case.

\medskip

\subparagraph{General case.} First we reduce to the subcase that $\gamma$ has a subray of height $s$ in~$\<Z,\hat\rho_r\>$. To carry out this reduction, after replacing $\gamma$ with some $\phi^k_\#(\gamma)$ we may assume that $\gamma$ contains a long piece of $\Lambda^+_s$ and so has a splitting $\gamma = R_- \cdot E \cdot R_+$ where $E$ is an edge of $H_s$ whose initial vertex and initial direction are principal. Lifting this splitting we have $\ti \gamma = \wt R_- \cdot \wt  E \cdot \wt R_+$. Let $\ti f$ be the principal lift that fixes the initial vertex of $\wt E$ and let $\wt R'$ be the principal ray determined by the initial direction of $\wt E$. Neither the line $R_- R'$ nor the line obtained by tightening $\bar R_+ R'$ is weakly attracted to $\Lambda^+_r$, because $\wt R_-$ and $\wt R_+$ are not weakly attracted and the ray $R'$ is contained in $\<Z,\hat \rho_r\>$. Each of these lines contains a subray of~$R'$, and any subray of $R'$ contains a further subray of height $s$ in $\<Z,\hat\rho_r\>$, and so it suffices to show that each of these lines is contained in $\B_\good(\phi)$, which completes the reduction.

Let $t$ be the highest integer in $\{r,\ldots,s-1\}$ for which $H_t$ is not contained in $Z$. Using that $\gamma$ has a subray of height $s$ in $\<Z,\hat\rho_r\>$, after making it a terminal subray by possibly inverting $\gamma$, there is a decomposition $\gamma = \ldots \nu_2\mu_1\nu_1 \mu_0$ into an alternating concatenation where the $\mu_l$'s are the maximal subpaths of $\gamma$ of height $>t$ that are in $\<Z,\hat \rho_r\>$, and the $\nu_l$'s are the subpaths of $\gamma$ that are complementary to the $\mu_l$'s. Each subpath $\nu_l$ has fixed endpoints, is contained in $G_t$, and is not an element of $\<Z,\hat \rho_r\>$. Further, $\nu_l$ is finite unless the decomposition of $\gamma$ is finite and $\nu_l$ is the leftmost term of the decomposition. Since $H_t$ is not a zero stratum, each component of $G_t$ is non-contractible and hence $f$-invariant. We prove that the above decomposition of $\gamma$ is finite by assuming that it is not and arguing to a contradiction. 
       
We claim that for all $l$ and all $m \ge 1$ the following hold:
\begin{enumerate}
\item If $\nu_l$ is finite, not all of $f^m_\#(\nu_l)$ is cancelled when $f^m_\#(\mu_{l})f^m_\#(\nu_l)f^m_\#( \mu_{l-1})$ is tightened to $f^m_\#(\mu_{l}\nu_l \mu_{l-1})$. Moreover, as $m \to \infty$ the part of $f^m_\#(\nu_l)$ that is not cancelled contains subpaths of $\Lambda^+_\phi$ which cross arbitrarily many edges of $H_r$.
\item Not all of $f^m_\#(\mu_l)$ is cancelled when $f^m_\#(\nu_{l+1})f^m_\#(\mu_l)f^m_\#(\nu_l)$ is tightened to $f^m_\#(\nu_{l+1}\mu_l \nu_l)$.  
\end{enumerate}        
Assuming without loss of generality that $m$ is large, (1) follows from finiteness of the path~$\nu_l$ by applying Lemma~\ref{defining Z}~\pref{ItemZPFinitePaths} which implies that the path $f^m_\#(\nu_l)$ contains subpaths of~$\Lambda^+_\phi$ that cross arbitrarily many edges of $H_r$, whereas $f^m_\#(\mu_l)$ and $f^m_\#(\mu_{l-1})$ contain no such subpaths. Item (2) follows from the fact that each component of $G_t$ is $f$-invariant which implies that $f^m_\#(\nu_{l+1}\mu_l \nu_l) \not \subset G_t$.

Items (1) and (2) together imply that if $\nu_l$ is finite, the only cancellation that occurs to $f_\#^m(\nu_{l})$ when the concatenation $\ldots f_\#^m(\nu_2) f_\#^m(\mu_1)  f_\#^m(\nu_1)  f_\#^m( \mu_0)$ is tightened to $f^m_\#(\gamma)$ is that which occurs when the subpath $f_\#^m(\mu_l) f_\#^m(\nu_{l}) f_\#^m(\mu_{l-1})$ is tightened to $f^m_\#(\mu_{l}\nu_l \mu_{l-1})$. But then $f^m_\#(\gamma)$ contains subpaths of a generic leaf of $\Lambda^+_\phi$ that cross arbitrarily many edges of $H_r$, in contradiction to the assumption that $\gamma$ is not weakly attracted to $\Lambda^+_\phi$.   
  
Not only have we shown that the decomposition of $\gamma$ is finite, we have shown that no $\nu_l$ term of the decomposition can be finite, and so either $\gamma = \mu_0$ or $\gamma = \nu_1\mu_0$. If $\gamma=\mu_0$ then $\gamma \in \<Z,\hat \rho_r\>$ and we are done. If $\gamma = \nu_1\mu_0$ then $\gamma$ falls into the special case and we are also done. This completes the proof of Theorem~\ref{ThmRevisedWAT}.

\bibliographystyle{amsalpha} 
\bibliography{mosher} 
 
 \printindex
 
 \end{document}